\documentclass[microtype]{gtpart}

\usepackage[normalem]{ulem}

\usepackage[usenames]{color}
\usepackage{xcolor}
\usepackage{graphicx}

\usepackage[shortlabels]{enumitem}

\usepackage{hyperref}

\theoremstyle{plain}
\newtheorem{proposition}{Proposition}[section]
\newtheorem{theorem}[proposition]{Theorem}
\newtheorem{lemma}[proposition]{Lemma}
\newtheorem{corollary}[proposition]{Corollary}
\theoremstyle{definition}
\newtheorem{example}[proposition]{Example}
\newtheorem{definition}[proposition]{Definition}
\newtheorem{observation}[proposition]{Observation}

\theoremstyle{remark}
\newtheorem{remark}[proposition]{Remark}

\newtheorem{question}[proposition]{Question}

\DeclareMathOperator{\Aut}{Aut}

\DeclareMathOperator{\diam}{diam}

\DeclareMathOperator{\GL}{GL}

\DeclareMathOperator{\PSL}{PSL}
\DeclareMathOperator{\PGL}{PGL}

\DeclareMathOperator{\End}{End}

\DeclareMathOperator{\Spanset}{Span} 
\DeclareMathOperator{\Gr}{Gr} 
\DeclareMathOperator{\id}{id} 
\DeclareMathOperator{\Haus}{Haus} 
\DeclareMathOperator{\CAT}{CAT} 
\DeclareMathOperator{\Isom}{Isom}

\DeclareMathOperator{\Stab}{Stab}

\DeclareMathOperator{\CH}{ConvHull}
\DeclareMathOperator{\relint}{rel-int}

\DeclareMathOperator{\dist}{d}

\DeclareMathOperator{\core}{core}
\DeclareMathOperator{\prl}{\diamond}
\DeclareMathOperator{\simplexcc}{max}

\DeclareMathOperator{\partiali}{\partial_{\, i}}
\DeclareMathOperator{\partialni}{\partial_{\, n}}

\newcommand{\triangleD}{\Delta}

\newcommand{\hil}{H_{\Omega}}
\newcommand{\bdry}{\partial}
\newcommand{\til}{\widetilde}

\DeclareMathOperator{\LG}{\Lambda}

\DeclareMathOperator{\sG}{\sigma}

\DeclareMathOperator{\gG}{\gamma}

\DeclareMathOperator{\Bc}{\mathcal{B}}
\DeclareMathOperator{\Cc}{\mathcal{C}}
\DeclareMathOperator{\Fc}{\mathcal{F}}
\DeclareMathOperator{\Gc}{\mathcal{G}}
\DeclareMathOperator{\Hc}{\mathcal{H}}
\DeclareMathOperator{\Kc}{\mathcal{K}}
\DeclareMathOperator{\Lc}{\mathcal{L}}
\DeclareMathOperator{\Nc}{\mathcal{N}}
\DeclareMathOperator{\Oc}{\mathcal{O}}

\DeclareMathOperator{\Tc}{\mathcal{T}}

\DeclareMathOperator{\Sc}{\mathcal{S}}

\DeclareMathOperator{\Hb}{\mathbb{H}}

\DeclareMathOperator{\Pb}{\mathbb{P}}
\DeclareMathOperator{\Rb}{\mathbb{R}}

\DeclareMathOperator{\Zb}{\mathbb{Z}}

\newcommand{\abs}[1]{\left|#1\right|}

\newcommand{\norm}[1]{\left\|#1\right\|}
\newcommand{\wt}[1]{\widetilde{#1}}
\newcommand{\wh}[1]{\widehat{#1}}
\newcommand{\ip}[1]{\left\langle #1\right\rangle}

\begin{document}

\title{Convex co-compact actions of relatively hyperbolic groups}
\author{Mitul Islam}
\address{Department of Mathematics, University of Michigan \newline Current Address: Mathematisches Institut, Heidelberg University}
\email{mislam@mathi.uni-heidelberg.de}
\author{Andrew Zimmer}
\address{Department of Mathematics, Louisiana State University \newline Current Address: Department of Mathematics, University of Wisconsin-Madison}
\email{amzimmer2@wisc.edu}
\date{\today}

\arxivreference{1910.08885}

\begin{abstract} 
In this paper we consider discrete groups in $\PGL_d(\Rb)$ acting convex co-compactly on a properly convex domain in real projective space. For such groups, we establish necessary and sufficient conditions for the group to be relatively hyperbolic in terms of the geometry of the convex domain. This answers a question of  Danciger-Gu{\'e}ritaud-Kassel and is analogous to a result of Hruska-Kleiner for $\CAT(0)$ spaces. \end{abstract}

\maketitle

\tableofcontents

\section{Introduction}
If $G$ is a connected simple Lie group with trivial center and no compact factors and $K \leq G$ is a maximal compact subgroup, then $X=G/K$ has a unique (up to scaling) Riemannian symmetric metric $g$ such that $G = \Isom_0(X,g)$. The metric $g$ is non-positively curved and $X$ is simply connected, hence every two points in $X$ are joined by a unique geodesic segment. A subset $\Cc \subset X$ is called \emph{convex} if for every $x,y \in \Cc$ the geodesic joining them is also in $\Cc$. Then, a discrete group $\Gamma \leq G$ is  \emph{convex co-compact} if there exists a non-empty closed convex set $\Cc \subset X$ such that $\gamma(\Cc) = \Cc$ for all $\gamma \in \Gamma$ and $\Gamma$ acts co-compactly on $\Cc$.

When $G$ has real rank one, for instance $G = \PSL_2(\Rb)$, there are an abundance of examples of convex co-compact subgroups in the context of Kleinian groups and hyperbolic geometry. When $G$ has higher real rank, for instance $G=\PSL_3(\Rb)$, there are few examples: Kleiner-Leeb~\cite{KL2006} and independently Quint~\cite{Q2005} proved that every Zariski dense convex co-compact subgroup is a co-compact lattice. 

Danciger-Gu{\'e}ritaud-Kassel~\cite{DGF2017} have recently introduced a different notion of convex co-compact subgroups in $\PGL_d(\Rb)$ based on the action of the subgroup on the projective space $\Pb(\Rb^d)$.  Their notion of convex co-compactness requires some preliminary definitions.  When $\Omega \subset \Pb(\Rb^d)$ is a properly convex domain, the \emph{automorphism group of $\Omega$} is defined to be
\begin{align*}
\Aut(\Omega) : = \{ g \in \PGL_d(\Rb) : g \Omega = \Omega\}.
\end{align*}
For a subgroup $\Lambda \leq \Aut(\Omega)$, the  \emph{full orbital limit set of $\Lambda$ in $\Omega$} is defined to be
\begin{equation*}
\Lc_{\Omega}(\Lambda):= \bigcup_{p \in \Omega} \Big( \overline{\Lambda \cdot p} \setminus \Lambda \cdot p \Big).
\end{equation*}
Next, let $\Cc_\Omega(\Lambda)$ denote the convex hull of $\Lc_\Omega(\Lambda)$ in $\Omega$. 

\begin{definition}[{Danciger-Gu{\'e}ritaud-Kassel~\cite[Definition 1.10]{DGF2017}}] 
\label{defn:cc}
Suppose $\Omega \subset \Pb(\Rb^d)$ is a properly convex domain. An infinite discrete subgroup $\Lambda \leq \Aut(\Omega)$ is called \emph{convex co-compact} if $\Cc_\Omega(\Lambda)$ is non-empty and $\Lambda$ acts co-compactly on $\Cc_\Omega(\Lambda)$. 
\end{definition}

When $\Lambda$ is word hyperbolic there is a close connection between this class of discrete groups in $\PGL_d(\Rb)$ and Anosov representations, see~\cite{DGF2017} for details and~\cite{DGF2018,Z2017} for related results. Further, by adapting an argument of Benoist~\cite{B2004}, Danciger-Gu{\'e}ritaud-Kassel proved a characterization of hyperbolicity in terms of the geometry of $\Cc_\Omega(\Lambda)$. To state their result we need two definitions.

\begin{definition} A subset $S \subset \Pb(\Rb^d)$ is a \emph{simplex} if there exist $g \in \PGL_d(\Rb)$ and $1 \leq k \leq d$ such that 
\begin{align*}
gS = \left\{ [x_1:\dots:x_{k}:0:\dots:0] \in \Pb(\Rb^d): x_1>0,\dots, x_k>0  \right\}.
\end{align*}  
In this case we define the dimension of $S$ to be $\dim(S)=k-1$ (notice that $S$ is homeomorphic to $\Rb^{k-1}$) and say that the $k$ points 
\begin{align*}
g^{-1}[1:0:\dots:0], g^{-1}[0:1:0:\dots:0], \dots, g^{-1}[0:\dots:0:1:0:\dots:0] \in \partial S
\end{align*}
are the vertices of $S$. 
\end{definition}

\begin{definition} Suppose $A \subset B \subset \Pb(\Rb^d)$. Then $A$ is \emph{properly embedded in $B$} if the inclusion map $A \hookrightarrow B$ is a proper map (relative to the subspace topology). 
\end{definition}

Finally, given a properly convex domain $\Omega \subset \Pb(\Rb^d)$ let $H_\Omega$ denote the Hilbert metric on $\Omega$ (see Section~\ref{subsec:Hilbert_metric} for the definition). 

\begin{theorem}[{Danciger-Gu{\'e}ritaud-Kassel~\cite[Theorem 1.15]{DGF2017}}]\label{thm:char_of_hyp}  Suppose $\Omega \subset \Pb(\Rb^d)$ is a properly convex domain and $\Lambda \leq \Aut(\Omega)$ is convex co-compact. Then the following are equivalent: 
\begin{enumerate}
\item $\Cc_\Omega(\Lambda)$ contains no properly embedded simplices with dimension at least two, 
\item $(\Cc_\Omega(\Lambda), H_\Omega)$ is Gromov hyperbolic, 
\item $\Lambda$ is word hyperbolic.
\end{enumerate}
\end{theorem}

\begin{remark}  In the special case when  $\Lambda$ acts co-compactly on $\Omega$, Theorem~\ref{thm:char_of_hyp} is due to Benoist~\cite{B2004}. 
\end{remark}

The case when $\Lambda$ is not word hyperbolic is less understood and Danciger-Gu{\'e}ritaud-Kassel asked the following. 

\begin{question}[{Danciger-Gu{\'e}ritaud-Kassel~\cite[Question A.2]{DGF2017}}]\label{quest:DGF} Suppose $\Omega \subset \Pb(\Rb^d)$ is a properly convex domain and $\Lambda \leq \Aut(\Omega)$ is convex co-compact. Under what conditions is $\Lambda$ relatively hyperbolic with respect to a collection of virtually Abelian subgroups?
\end{question}

In this paper we provide an answer to this question in terms of the geometry of the family of all maximal properly embedded simplices (note, a properly embedded simplex is called \emph{maximal} if it is not contained in a larger properly embedded simplex; they are not necessarily simplices of codimension one). 

Our approach is motivated by previous work of Hruska-Kleiner~\cite{HK2005} for $\CAT(0)$ spaces (see Section \ref{sec:motivation-cat-0} for details). In some ways the Hilbert metric on a properly convex domain behaves like a $\CAT(0)$-metric, see the discussion in~\cite{L2014}. However, an old result of Kelly-Strauss~\cite{KS1958} says that a Hilbert geometry $(\Omega, H_\Omega)$ is $\CAT(0)$ if and only if it is isometric to real hyperbolic $(d-1)$-space (in which case $\Omega$ coincides with the interior of the convex hull of an ellipsoid in some affine chart). Thus, one requires different techniques for studying the geometry of properly convex domains and the groups acting on them. 

The following theorem is the first main result of this paper.

\begin{theorem}\label{thm:main_cc}(see Section~\ref{sec:pf_of_thm_main_cc}) Suppose $\Omega \subset \Pb(\Rb^d)$ is a properly convex domain, $\Lambda \leq \Aut(\Omega)$ is convex co-compact, and $\Sc_{\simplexcc}$ is the family of all maximal properly embedded simplices in $\Cc_{\Omega}(\Lambda)$ of dimension at least two. Then the following are equivalent: 
\begin{enumerate}
\item $\Sc_{\simplexcc}$ is closed and discrete in the local Hausdorff topology induced by $H_\Omega$,
\item $(\Cc_\Omega(\Lambda), H_\Omega)$ is a relatively hyperbolic space with respect to $\Sc_{\simplexcc}$,
\item $(\Cc_\Omega(\Lambda), H_\Omega)$ is a relatively hyperbolic space with respect to a family of properly embedded simplices in $\Cc_\Omega(\Lambda)$ of dimension at least two,
\item $\Lambda$ is a relatively hyperbolic group with respect to a collection of virtually Abelian subgroups of rank at least two.
\end{enumerate}
\end{theorem}

Theorem~\ref{thm:main_cc} can be viewed as a real projective analogue of a result of Hruska-Kleiner~\cite{HK2005} for $\CAT(0)$ spaces (see Section \ref{sec:motivation-cat-0} for details). In this analogy, maximal properly embedded simplices correspond to maximal totally geodesic flats in $\CAT(0)$ spaces (see \cite{IZ2019, B2004}).

We also establish a number of properties for convex co-compact subgroups satisfying the conditions in Theorem~\ref{thm:main_cc}. Before stating these results, we informally introduce some notations (see Section~\ref{sec:notations_and_definitions} for precise definitions). Given a properly convex set $\Omega$ which is open in its span, let $\diam_\Omega(A)$ and $\Nc_\Omega(A;r)$ denote the diameter and $r$-neighborhood of a subset $A$ with respect to the Hilbert metric.  Also, given $x \in \overline{\Omega}$, let $F_\Omega(x) \subset \overline{\Omega}$ denote the open face of $x$ in $\Omega$ (see Definition~\ref{defn:open_faces}). Finally, given a properly convex set $C$, let $\relint(C)$ denote the relative interior (see Definition~\ref{defn:topology}) and let $\partial C = \overline{C} \setminus \relint(C)$ denote the boundary.

We will prove the following.

\begin{theorem}\label{thm:properties_of_cc}(see Section~\ref{sec:pf_of_thm_main_cc})  Suppose $\Omega \subset \Pb(\Rb^d)$ is a properly convex domain, $\Lambda \leq \Aut(\Omega)$ is convex co-compact, and $\Sc_{\simplexcc}$ is the family of all maximal properly embedded simplices in $\Cc_{\Omega}(\Lambda)$ of dimension at least two. If $\Sc_{\simplexcc}$ is closed and discrete in the local Hausdorff topology induced by $H_\Omega$, then:
\begin{enumerate}
\item $\Lambda$ has finitely many orbits in $\Sc_{\simplexcc}$. 
\item If $S \in \Sc_{\simplexcc}$, then $\Stab_{\Lambda}(S)$ acts co-compactly on $S$ and contains a finite index subgroup isomorphic to $\Zb^k$ where $k = \dim S$. 
\item If $A \leq \Lambda$ is an Abelian subgroup of rank at least two, then there exists a unique $S \in \Sc_{\simplexcc}$ with $A \leq \Stab_{\Lambda}(S)$. 
\item If $S \in \Sc_{\simplexcc}$ and $x \in \partial S$, then $F_{\Omega}(x) = F_S(x)$. 
\item If $S_1, S_2 \in \Sc_{\simplexcc}$ are distinct, then $\#(S_1 \cap S_2) \leq 1$ and $\partial S_1 \cap \partial S_2 = \emptyset$. 
\item For any $r > 0$ there exists $D(r) > 0$ such that: if $S_1, S_2 \in \Sc_{\simplexcc}$ are distinct, then 
\begin{align*}
\diam_\Omega \Big( \Nc_\Omega(S_1; r) \cap \Nc_\Omega(S_2; r) \Big) \leq D(r).
\end{align*}
\item If $\ell \subset \overline{\Cc_{\Omega}(\Lambda)} \cap \partial \Omega$ is a non-trivial line segment, then there exists $S \in \Sc_{\simplexcc}$ with $\ell \subset \partial S$. 
\item If $x \in \overline{\Cc_{\Omega}(\Lambda)} \cap \partial \Omega$ is not a $C^1$-smooth point of $\partial \Omega$, then there exists  $S \in \Sc_{\simplexcc}$ with $x \in \partial S$. 
\end{enumerate}
\end{theorem}

\begin{remark} In the special case when $d \leq 4$ and $\Lambda$ acts co-compactly on $\Omega \subset \Pb(\Rb^d)$, Theorems \ref{thm:main_cc} and \ref{thm:properties_of_cc} can be obtained from results of Benoist~\cite{B2006}. But, when $d > 4$ Theorems \ref{thm:main_cc} and \ref{thm:properties_of_cc} are new even in the special case when $\Lambda$ acts co-compactly on $\Omega$. 
\end{remark}

As alluded to above, convex co-compact subgroups can be seen as a way to extend the theory of Anosov representations to non-hyperbolic groups. Kapovich-Leeb~\cite{KL2018} and Zhu \cite{FZ2019} have also recently proposed notions of relative Anosov representations for relatively hyperbolic groups. However, in their definitions the peripheral subgroups will have unipotent image while convex co-compact subgroups never contain non-trivial unipotent elements. So the groups considered in this paper are very different from the ones studied in the work of Kapovich-Leeb and Zhu.

\subsection{Naive convex co-compact subgroups}We also establish a variant of Theorem~\ref{thm:main_cc} for a more general notion of convex co-compact subgroup. 

\begin{definition}\label{defn:cc_naive}  Suppose $\Omega \subset \Pb(\Rb^d)$ is a properly convex domain. An infinite discrete subgroup $\Lambda \leq \Aut(\Omega)$ is called \emph{naive convex co-compact} if there exists a non-empty closed convex subset $\Cc \subset \Omega$ such that 
\begin{enumerate}
\item $\Cc$ is $\Lambda$-invariant, that is, $g\Cc = \Cc$ for all $g \in \Lambda$, and
\item $\Lambda$ acts co-compactly on $\Cc$. 
\end{enumerate}
In this case, we say that $(\Omega, \Cc, \Lambda)$ is a \emph{naive convex co-compact triple}. 
\end{definition}

It is straightforward to construct examples where $\Lambda \leq \Aut(\Omega)$ is naive convex co-compact, but not convex co-compact (see Section~\ref{sec:ncc_examples} or \cite[Section 3.4]{DGF2017}). In these cases, the convex subset $\Cc$ in Definition~\ref{defn:cc_naive} is a strict subset of $\Cc_\Omega(\Lambda)$.

For naive convex co-compact subgroups, we also provide a characterization of relative hyperbolicity, but require a technical notion of  isolated simplices. In the naive convex co-compact case there exist examples where the group is relatively hyperbolic, but the family of maximal properly embedded simplices is not discrete. Instead, maximal properly embedded simplices can occur in parallel families, see Section~\ref{sec:ncc_examples}. These examples lead to the following definition.

\begin{definition}\label{defn:IS}
Suppose $(\Omega, \Cc, \Lambda)$ is a naive convex co-compact triple. A family $\Sc$ of maximal properly embedded simplices in $\Cc$ of dimension at least two is called: 
\begin{enumerate}
\item \emph{Isolated}, if $\Sc$ is closed and discrete in the local Hausdorff topology induced by $H_\Omega$.
\item \emph{Coarsely complete}, if any properly embedded simplex in $\Cc$ of dimension at least two is contained in a uniformly bounded tubular neighborhood of some maximal properly embedded simplex in $\Sc$.
\item \emph{$\Lambda$-invariant}, if $g \cdot S \in \Sc$ for all $S \in \Sc$ and $g \in \Lambda$.
\end{enumerate}
We say that $(\Omega, \Cc, \Lambda)$ has \emph{coarsely isolated simplices} if there exists an isolated, coarsely complete, and $\Lambda$-invariant family of maximal properly embedded simplices in $\Cc$ of dimension at least two. 
\end{definition}

\begin{remark} This definition is motivated by Hruska-Kleiner's notion of isolated flats in the first sense for $\CAT(0)$ spaces~\cite[pg. 1505]{HK2005} (see Section \ref{sec:motivation-cat-0}). \end{remark}

Informally, Definition~\ref{defn:IS} says that a naive convex co-compact triple has coarsely isolated simplices if it is possible to select a closed and discrete family of simplices that contain a representative from each parallel family of maximal properly embedded simplices.

We will prove the following. 

\begin{theorem}\label{thm:main_ncc}(see Section~\ref{sec:pf_of_thm_main_ncc}) Suppose $(\Omega, \Cc, \Lambda)$ is a naive convex co-compact triple. Then the following are equivalent: 
\begin{enumerate}
\item $(\Omega, \Cc, \Lambda)$ has coarsely isolated simplices,
\item $(\Cc, H_\Omega)$ is a relatively hyperbolic space with respect to a family of properly embedded simplices in $\Cc$ of dimension at least two,
\item $\Lambda$ is a relatively hyperbolic group with respect to a family  of virtually Abelian subgroups of rank at least two.
\end{enumerate}
\end{theorem}

\begin{remark} Using the basic theory of relatively hyperbolic spaces (see Theorem~\ref{thm:rh_embeddings_of_flats} below), if $(\Cc, H_\Omega)$ is a relatively hyperbolic space with respect to a family of properly embedded simplices in $\Cc$, then every simplex in that family is maximal. 
\end{remark}

There is one subtle aspect of Theorem~\ref{thm:main_ncc}: it does \textbf{not} say that any isolated, coarsely complete, and $\Lambda$-invariant family of simplices satisfies part (2). In fact, it is possible to construct an example of a naive convex compact triple $(\Omega, \Cc, \Lambda)$ and a family $\Sc$ of maximal properly embedded simplices where $\Sc$ is isolated, coarsely complete, and $\Lambda$-invariant, but $(\Cc, H_\Omega)$ is \textbf{not} relatively hyperbolic with respect to $\Sc$ (see Section~\ref{sec:ncc_examples} for details). This motivates the following definition.

\begin{definition} Suppose $(\Omega, \Cc, \Lambda)$ is a naive convex co-compact triple. A family $\Sc$ of maximal properly embedded simplices in $\Cc$ of dimension at least two is called \emph{strongly isolated} if for every $r>0$, there exists $D(r)>0$ such that: if $S_1, S_2 \in \Sc$ are distinct, then 
\begin{equation*}
\diam_{\Omega} \Big( \Nc_{\Omega}(S_1;r) \cap \Nc_{\Omega}(S_2;r)\Big) \leq D(r).
\end{equation*}
\end{definition}

\begin{remark} This definition is motivated by Hruska-Kleiner's notion of isolated flats in the second sense for $\CAT(0)$ spaces~\cite[pg. 1505]{HK2005} (see Section \ref{sec:motivation-cat-0}). \end{remark}

It is fairly easy to show that a family of strongly isolated simplices is also isolated (see Observation~\ref{obs:str_iso_implies_iso} below). Although the converse is not always true, we will prove the following theorem.

\begin{theorem} 
\label{thm:S-core-exists}(see Section~\ref{sec:intersection_of_nbhds})
Suppose $(\Omega, \Cc, \Lambda)$ is a naive convex co-compact triple with coarsely isolated simplices. Then there exists a strongly isolated, coarsely complete, and $\Lambda$-invariant family of maximal properly embedded simplices in $\Cc$ of dimension at least two. 
\end{theorem}

 We then prove the following refinement of the Theorem~\ref{thm:main_ncc}.

\begin{theorem}
\label{thm:IS-implies-rel-hyp}(see Section~\ref{sec:isolated-simplex-implies-rel-hyp})
Suppose $(\Omega, \Cc, \Lambda)$ is a naive convex co-compact triple with coarsely isolated simplices. Let $\Sc$ be a strongly isolated, coarsely complete, and $\Lambda$-invariant family of maximal properly embedded simplices in $\Cc$ of dimension at least two. Then:
\begin{enumerate}
\item  $(\Cc, H_\Omega)$ is a relatively hyperbolic space with respect to $\Sc$.
\item $\Lambda$ has finitely many orbits in $\Sc$ and if $\{S_1,\dots, S_m\}$ is a set of orbit representatives, then $\Lambda$ is a relatively hyperbolic group with respect to 
\begin{align*}
\left\{\Stab_{\Lambda}(S_1),\dots, \Stab_{\Lambda}(S_m)\right\}.
\end{align*} 
Further, each $\Stab_{\Lambda}(S_i)$ is virtually Abelian of rank at least two.
\end{enumerate}
\end{theorem}

Delaying definitions until later in the paper (see Definitions~\ref{defn:topology},~\ref{defn:open_faces}, and~\ref{defn:half_triangle}), we will also establish an analogue of Theorem~\ref{thm:properties_of_cc} for naive convex co-compact subgroups.

\begin{theorem}\label{thm:properties_of_ncc}(see Section~\ref{sec:pf_of_properties_of_ncc})  Suppose $(\Omega, \Cc, \Lambda)$ is a naive convex co-compact triple with coarsely isolated simplices. Let $\Sc$ be a strongly isolated, coarsely complete, and $\Lambda$-invariant family of maximal properly embedded simplices in $\Cc$ of dimension at least two.  Then: 
\begin{enumerate}
\item $\Lambda$ has finitely many orbits in $\Sc$. 
\item If $S \in \Sc$, then $\Stab_{\Lambda}(S)$ acts co-compactly on $S$ and contains a finite index subgroup isomorphic to $\Zb^k$ where $k = \dim S$. 
\item If $A \leq \Lambda$ is an Abelian subgroup with rank at least two, then there exists a unique $S \in \Sc$ with $A \leq \Stab_{\Lambda}(S)$. 
\item There exists $D > 0$ such that: If $S \in \Sc$ and $x \in \partial S$, then 
\begin{align*}
H_{F_\Omega(x)}^{\Haus}\Big( \overline{\Cc} \cap F_{\Omega}(x), F_S(x) \Big) \leq D.
\end{align*} 
\item If $S_1, S_2 \in \Sc$ are distinct, then $\#(S_1 \cap S_2) \leq 1$ and 
\begin{align*}
\left(\bigcup_{x \in \partial S_1} F_\Omega(x) \right) \bigcap \left( \bigcup_{x \in \partial S_2} F_\Omega(x) \right)= \emptyset.
\end{align*}
\item  If $a,b,c \in \overline{\Cc} \cap \partial \Omega$ form a half triangle in $\Omega$, then there exists $S \in \Sc$ where 
\begin{align*}
a,b,c \in \bigcup_{x \in \partial S} F_\Omega(x). 
\end{align*}
\end{enumerate}
\end{theorem}

\begin{remark} In general, properties (4), (7), and (8) in Theorem~\ref{thm:properties_of_cc} are not true in the naive convex co-compact case. But, properties (4) and (6) in Theorem~\ref{thm:properties_of_ncc} can be seen as their coarse analogues. 
\end{remark}

\subsection{Motivation from the theory of $\CAT(0)$ spaces}
\label{sec:motivation-cat-0} The main results of this paper are inspired by previous work of Hruska-Kleiner~\cite{HK2005} in the $\CAT(0)$-setting. They introduced two notions of isolated flats for $\CAT(0)$ spaces and then related these conditions to relative hyperbolicity. In this subsection we recall their definitions and results.

\begin{definition}[Hruska-Kleiner~\cite{HK2005}]  Suppose $X$ is a $\CAT(0)$-space and $\Gamma$ acts geometrically on $X$ (i.e. the action is properly discontinuous, co-compact, and isometric). 
\begin{enumerate}
\item $(X,\Gamma)$ has \emph{isolated flats in the first sense} if there exists a set $\Fc$ of flats of $X$ such that:  $\Fc$ is $\Gamma$-invariant; each flat in $X$ is contained in a uniformly bounded tubular neighborhood of some flat in $\Fc$; and $\Fc$ is closed and discrete in the local Hausdorff topology.
\item $(X,\Gamma)$ has \emph{isolated flats in the second sense} if there exists a set $\Fc$ of flats of $X$ such that: $\Fc$ is $\Gamma$-invariant; each flat in $X$ is contained in a uniformly bounded tubular neighborhood of some flat in $\Fc$; and for any $r > 0$ there exists $D(r) >0$ so that for any two distinct flats $F_1,F_2 \in \Fc$ we have
\begin{align*}
\diam_X \left(\Nc_X(F_1;r) \cap \Nc_X(F_2;r) \right) < D(r). 
\end{align*}
\end{enumerate}
\end{definition}

Clearly, if $(X,\Gamma)$ has isolated flats in the second sense then it also has isolated flats in the first sense. Hruska-Kleiner~\cite{HK2005} proved the following. 

\begin{theorem}[Hruska-Kleiner~\cite{HK2005}] Suppose $X$ is a $\CAT(0)$-space and $\Gamma$ acts geometrically on $X$. The following are equivalent: 
\begin{enumerate}
\item $(X,\Gamma)$ has isolated flats in the first sense,
\item $(X,\Gamma)$ has isolated flats in the second sense,
\item $X$ is a relatively hyperbolic space with respect to a family of flats,
\item $\Gamma$ is a relatively hyperbolic group with respect to a collection of virtually Abelian subgroups of rank at least two.
\end{enumerate}
\end{theorem}

Hruska-Kleiner's work motivated the results in this paper, but the methods of proof are very different. 

\subsection{Outline of paper}

Sections~\ref{sec:examples} through~\ref{sec:background_on_rel_hyp} are mostly expository. In Section~\ref{sec:examples}, we describe some examples. In Section~\ref{sec:notations_and_definitions}, we set our basic notations and definitions.  In Section~\ref{sec:background_on_rel_hyp}, we recall the definition of relatively hyperbolic metric spaces and some of their basic properties.

The rest of the paper is divided into three parts. In the first part, Sections~\ref{sec:basic_properties_of_simplices} through~\ref{sec:opposite_faces}, we study properly embedded simplices in general properly convex domains. In the second part of the paper, Sections~\ref{sec:inv_plus_iso_implies_periodic} through~\ref{sec:pf_of_thm_main_ncc}, we consider the naive convex co-compact case and prove Theorems~\ref{thm:main_ncc},~\ref{thm:S-core-exists},~\ref{thm:IS-implies-rel-hyp}, and~\ref{thm:properties_of_ncc}. 

An experienced reader would be able to follow the proof in the naive convex co-compact case by only reading the following: Section~\ref{sec:linear_projections}, Section~\ref{sec:inv_plus_iso_implies_periodic}, the statement of Theorem \ref{thm:bd_faces}, then Sections~\ref{sec:intersection_of_nbhds} through~\ref{sec:pf_of_thm_main_ncc}. In the final part of the paper, Sections \ref{sec:lines_corners_in_the_bd} and \ref{sec:pf_of_thm_main_cc}, we consider the convex co-compact case and prove Theorems~\ref{thm:main_cc} and~\ref{thm:properties_of_cc}. In Section \ref{sec:pf_of_thm_main_cc}, we explain how to deduce the convex co-compact case from the naive convex co-compact case. Section \ref{sec:lines_corners_in_the_bd} proves parts (7) and (8) of Theorem \ref{thm:properties_of_cc}, which is a refinement of Theorem \ref{thm:properties_of_ncc}.

We now describe some of the proofs in the second part of the paper in the order they are presented.

\subsubsection{Outline of the proof of Theorem~\ref{thm:S-core-exists}:} (see Section~\ref{sec:intersection_of_nbhds}) When $(\Omega, \Cc, \Lambda)$ has coarsely isolated simplices, we use the following algorithm to construct $\Sc_{\core}$: a canonical strongly isolated, coarsely complete, and $\Lambda$-invariant family of maximal properly embedded simplices of dimension at least two. 

First, let $\Sc_{\max}$ denote the family of \textbf{all} maximal properly embedded simplices in $\Cc$ of dimension at least two. In the naive convex co-compact case, this family can have several  undesirable properties:
\begin{enumerate}[(a)]
\item A maximal simplex could be contained in a tubular neighborhood of a properly embedded simplex with strictly larger dimension. 
\item $\Sc_{\max}$ could contain families of parallel maximal simplices (see Definition~\ref{defn:parallel}). 
\end{enumerate}
 In Section~\ref{sec:max_not_coarsely_max} (respectively Section~\ref{sec:ncc_examples}) we construct a simple example where the first (respectively second) problem occurs. 

To deal with the first problem, we consider $\wh{\Sc}_{\max} \subset \Sc_{\max}$ the family of all maximal properly embedded simplices that are not contained in a tubular neighborhood of a properly embedded simplex with strictly larger dimension and show that this subfamily is still coarsely complete. 

To deal with the second problem, we select from each family of parallel simplices a canonical ``core'' simplex.  This is accomplished by studying the open boundary faces $F_\Omega(x)$ of points $x \in \partial \Omega$. 

For every $S \in \wh{\Sc}_{\max}$ and vertex $v$ of $S$, we show that $F_\Omega(v) \cap \overline{\Cc}$ is a compact subset of $F_\Omega(v)$. Then we exploit the fact that every compact set in a properly convex domain has a well defined ``center of mass'' (see Proposition~\ref{prop:center_of_mass} below). Using this, for each simplex $S \in \wh{\Sc}_{\max}$ we construct a new simplex $\Phi(S)$ as follows. Let $v_1,\dots, v_p$ be the vertices of $S$. Then for $1 \leq j \leq p$, define $w_j$ to be the center of mass of $F_\Omega(v_j) \cap \overline{\Cc}$ in $F_\Omega(v_j)$. Next define $\Phi(S)$ to be the convex hull of $w_1,\dots, w_p$ in $\Omega$. 

Then $\Phi(S)$ is a properly embedded simplex parallel to $S$ (see Lemma~\ref{lem:slide_along_faces}). Moreover, if $S_1, S_2 \in \wh{\Sc}_{\max}$ are parallel, then $\Phi(S_1) = \Phi(S_2)$. Finally, we define 
\begin{align*}
\Sc_{\core} :=\left\{ \Phi(S) : S \in \wh{\Sc}_{\max}\right\}.
\end{align*}
Showing that this procedure actually produces a strongly isolated, coarsely complete, and invariant family requires results from Sections~\ref{sec:inv_plus_iso_implies_periodic} and~\ref{sec:bd_faces}.  

In Section~\ref{sec:inv_plus_iso_implies_periodic}, we show that any isolated and invariant family of maximal properly embedded simplices satisfies properties (1) and (2) in Theorem~\ref{thm:properties_of_ncc}. 

Then in Section~\ref{sec:bd_faces} we show that if $\Sc_0$ is an isolated, coarsely complete, and invariant family of maximal properly embedded simplices of dimension at least two, then there exists a subfamily $\Sc \subset \Sc_0$ which satisfies property (4) in Theorem~\ref{thm:properties_of_ncc} while still being isolated, coarsely complete,  and invariant. In the proof we first construct an explicit subfamily and then argue by contradiction that it must satisfy property (4). The main idea is to use the structure theorem from Section~\ref{sec:opposite_faces} and the action of $\Lambda$ to construct lots of properly embedded simplices. Then we use these simplices to obtain a contradiction. This result is a key step in showing that the map $\Phi$ is well defined and that $\Sc_{\core}$ is strongly isolated.

\subsubsection{Outline of the proof of Theorem~\ref{thm:properties_of_ncc}:}  

(see Section \ref{sec:pf_of_properties_of_ncc})
Properties (1) and (2) are established in Section~\ref{sec:inv_plus_iso_implies_periodic}. Properties (3) and (5) are straightforward consequences of the strong isolation property. Property (4) follows from the results in Section~\ref{sec:bd_faces}. 

We establish property (6) in Section~\ref{sec:half_triangles} by combining a Benz\'ecri~\cite{B1960} recentering argument with the strong isolation property. With the notation in Theorem~\ref{thm:properties_of_ncc}, let $V:=\Spanset\{a,b,c\}$. By a recentering argument, for any $r > 0$ there exists a neighborhood $\Oc$ of $b$ such that: if $x \in \Oc \cap \Pb(V) \cap \Cc$, then there exists a simplex $S_x \in \Sc$ of dimension at least two with
\begin{align*}
B_\Omega(x;r) \subset \Nc_\Omega(S_x;D+1)
\end{align*}
where $D$ is the constant from the coarsely complete condition. By picking $r > 0$ sufficiently large and using the fact the family $\Sc$  is strongly isolated, we then show that $S_x$ is independent of $x$ and hence
\begin{align*}
\Oc \cap \Pb(V) \cap \Cc \subset \Nc_\Omega(S; D+1)
\end{align*}
for some $S \in \Sc$. Then it is easy to show that $a,b,c \in \cup_{x \in \partial S} F_\Omega(x)$. 

This use of the strong isolation property  is similar to the proofs of Lemma 3.3.2 and Proposition 3.2.5 in~\cite{HK2005}.

\subsubsection{Outline of the proof of Theorem~\ref{thm:IS-implies-rel-hyp}:} (see Section~\ref{sec:isolated-simplex-implies-rel-hyp}) Our proof uses Theorem~\ref{thm:properties_of_ncc} and a characterization of relative hyperbolicity due to Sisto~\cite{S2013}, which is stated in Theorem~\ref{thm:Sisto_equiv} below. This characterization involves the existence of a system of projection maps onto the simplices in $\Sc$ with certain nice metric properties and a technical condition concerning thinness of certain geodesic triangles whose edges  infrequently intersect neighborhoods of simplices in $\Sc$.

In Section~\ref{sec:linear_projections}, we use supporting hyperplanes to construct natural linear projections from a properly convex domain onto any properly embedded simplex. In the setup of Theorem \ref{thm:IS-implies-rel-hyp}, these linear  projections end up being coarsely equivalent to the closest point projection onto simplices of dimension at least two in the Hilbert metric (see Definition \ref{defn:closest-point-proj} and Proposition \ref{prop:proj-coarsely-equiv}). The following property of these linear projections plays a key role: if the linear projections of two points onto a simplex are far apart, then the geodesic between those two points spends a significant amount of time in a tubular neighbourhood of $S$ (see Proposition \ref{prop:projection-to-simplices-contracting} and Corollary \ref{cor:penetration-of-simplex-nbd}).

Many of the proofs in Section~\ref{sec:isolated-simplex-implies-rel-hyp} proceed by contradiction and involve constructing a half triangle in $\partiali \Cc$ using an argument similar to the proof of Proposition 2.5 in~\cite{B2004}. Then property (6) in Theorem~\ref{thm:properties_of_ncc} is used.

\subsubsection{Outline of the proof of Theorem~\ref{thm:main_ncc}:}(see Section~\ref{sec:pf_of_thm_main_ncc})  The $(1) \Rightarrow (3)$ direction is a consequence of Theorems~\ref{thm:S-core-exists} and~\ref{thm:IS-implies-rel-hyp}. We show that $(2) \Rightarrow (1)$ using the general theory of relatively hyperbolic metric spaces. And we establish $(3) \Rightarrow (2)$ by using the following real projective analogue of the  flat torus theorem~\cite{GW1971,LY1972}.

\begin{theorem}[I.-Z.~\cite{IZ2019}] \label{thm:max_abelian}Suppose that $(\Omega, \Cc, \Lambda)$ is a naive convex co-compact triple. If $A \leq \Lambda$ is a maximal Abelian subgroup of $\Lambda$, then there exists a properly embedded simplex $S \subset \Cc$ such that:
\begin{enumerate}
\item $S$ is $A$-invariant,
\item $A$ fixes each vertex of $S$, and
\item $A$ acts co-compactly on $S$. 
\end{enumerate}
Moreover, $A$ contains a finite index subgroup isomorphic to $\Zb^{\dim (S)}$. 
\end{theorem}

\subsection*{Acknowledgements} The authors thank Ralf Spatzier for many helpful conversations. They also thank the referees for their useful comments and corrections. A. Zimmer also thanks the University of Michigan for hospitality during a visit where work on this project started. 

M. Islam is partially supported by the National Science Foundation under grant 1607260 and A. Zimmer is partially supported by the National Science Foundation under grants 1904099, 2105580, and 2104381.

\section{Examples}\label{sec:examples}

\subsection{Divisible examples} A properly convex domain $\Omega \subset \Pb(\Rb^d)$ is called \emph{divisible} if there exists a discrete group $\Lambda \leq \Aut(\Omega)$ which acts co-compactly on $\Omega$. Clearly, in this case $\Lambda$ is also a convex co-compact subgroup. Divisible domains have been extensively studied and in this subsection we will recall some examples, for more details see the survey papers~\cite{B2008,Q2010,L2014}. 

An open convex cone $C \subset \Rb^d$ is  \emph{reducible} if there exist a non-trivial decomposition $\Rb^d = V_1 \oplus V_2$ and convex cones $C_1 \subset V_1$ and $C_2 \subset V_2$ such that $C=C_1 + C_2$. Otherwise $C$ is said to be \emph{irreducible}. The preimage in $\Rb^d$ of a properly convex domain $\Omega \subset \Pb(\Rb^d)$ is the union of a cone and its negative, when this cone is reducible (respectively irreducible) we say that $\Omega$ is \emph{reducible} (respectively \emph{irreducible}). 

The Klein-Beltrami model of real hyperbolic $d$-space is the fundamental example of a convex divisible domain. In particular, if $\Bc \subset \Pb(\Rb^{d})$ is the interior of the convex hull of an ellipsoid in some affine chart, then $(\Bc, H_{\Bc})$ is isometric to $\Hb^{d-1}_{\Rb}$ real hyperbolic $(d-1)$-space and $\Aut(\Bc)$ coincides with $\Isom(\Hb^{d-1}_{\Rb})$. Further, $\Bc$ is a divisible convex domain because $\Aut(\Bc)$, being a simple Lie group, contains co-compact lattices.

There are many other examples of divisible convex domains, for instance: for every $d \geq 5$, Kapovich~\cite{K2007} has constructed divisible convex domains $\Omega \subset \Pb(\Rb^{d})$ such that $\Aut(\Omega)$ is discrete, Gromov hyperbolic, and not quasi-isometric to any symmetric space. 

When $d=3$, results of Benz\'{e}cri~\cite{B1960} imply that every irreducible divisible convex domain has word hyperbolic dividing group (see~\cite[Section 2]{B2006} for details). In $d=4$, Benoist established the following dichotomy. 

\begin{theorem}[Benoist~\cite{B2006}]\label{thm:benoist_d_4} If $\Omega \subset \Pb(\Rb^4)$ is an irreducible properly convex domain and $\Lambda \leq \Aut(\Omega)$ is a discrete group acting co-compactly on $\Omega$, then either
\begin{enumerate}
\item $\Lambda$ is word hyperbolic, or 
\item  $\Lambda$ is a relatively hyperbolic group with respect to a non-empty collection of virtually Abelian subgroups of rank two.
\end{enumerate}
\end{theorem}

Benoist~\cite{B2006} and Ballas-Danciger-Lee~\cite{BDL2018} have constructed examples of the second case in Theorem~\ref{thm:benoist_d_4}.

The case when $d>4$ is fairly mysterious. When $d=5$, $6$, or $7$ Choi-Lee-Marquis~\cite{CLM2016} have constructed examples where $\Lambda$ is a relatively hyperbolic group with respect to a collection of virtually Abelian subgroups of rank at least two. They also ask whether Benoist's result is true in any dimension.

\begin{question}[{Choi-Lee-Marquis~\cite[Remark 1.11]{CLM2016}}] Are groups dividing non-symmetric irreducible properly convex domains always relatively hyperbolic with respect to a (possibly empty) collection of virtually Abelian subgroups of rank at least two? \end{question}

In the context of the above question, we should mention a recent result of Bobb~\cite{MB2020} who proved that if $\Omega \subset \Pb(\Rb^d)$ is divisible, then the family of properly embedded simplices in $\Omega$ of dimension $(d-2)$ (i.e. codimension 1) is closed and discrete in the local Hausdorff topology induced by the Hilbert metric. Thus if $\Omega$ only contains simplices of dimension $(d-2)$, then Theorem \ref{thm:main_cc} above implies that $\Lambda$ is relatively hyperbolic with respect  to a collection of virtually Abelian subgroups of rank $(d-2)$.

\subsection{Convex co-compact examples} 

In this subsection we recall a class of examples constructed by  Danciger-Gu{\'e}ritaud-Kassel. For details and other examples see~\cite[Section 12]{DGF2017}. 
 
\begin{proposition}[{Danciger-Gu{\'e}ritaud-Kassel~\cite[Section 12.2.2]{DGF2017}}] For any $d \geq 4$ there exists a properly convex domain $\Omega \subset \Pb(\Rb^d)$ with a convex co-compact subgroup $\Lambda \leq \Aut(\Omega)$ such that:
\begin{enumerate}
\item $\Lambda$ acts irreducibly on $\Rb^d$ and
\item $\Lambda$ is a relatively hyperbolic group with respect to a non-empty collection of virtually Abelian subgroups of rank two.
\end{enumerate}
\end{proposition}

\subsection{ Naive convex co-compact examples I }
\label{sec:ncc_examples} In this subsection we construct examples of the following:
\begin{enumerate}[(a)]
\item  a naive convex co-compact triple which is not convex co-compact,
\item a naive convex co-compact triple where the group is relatively hyperbolic, but the family of all maximal properly embedded simplices of dimension at least two is not discrete in the local Hausdorff topology, and
\item a naive convex co-compact triple $\left(\Omega_{\star}, \Cc_{\star}^{(R)}, \Lambda_{\star}\right)$ and a family $\Sc$ of maximal properly embedded simplices of dimension at least two which is isolated, coarsely complete, and $\Lambda_{\star}$-invariant; but $( \Cc_{\star}^{(R)}, H_{\Omega_\star})$ is not relatively hyperbolic with respect to $\Sc$. 
\end{enumerate} 

For the rest of the subsection, we will freely use the notation introduced in Section~\ref{sec:notations_and_definitions} below and make the following assumptions:

 \medskip

\noindent \textbf{Assumptions:} $\Omega \subset \Pb(\Rb^d)$ is a properly convex domain and $\Lambda \leq \Aut(\Omega)$ is a discrete group which acts co-compactly on $\Omega$. 

\medskip

Let $\pi: \Rb^d \rightarrow \Pb(\Rb^d)$ be the natural projection. Then $\pi^{-1}(\Omega) = C \cup -C$ where $C \subset \Rb^d$ is some properly convex cone.   Then define
\begin{align*}
\Omega_{\star} & := \{ [(v,w)] : v,w \in C\} \subset \Pb(\Rb^{2d}), \\
\Cc_{\star}& :=\{ [(v,v)] : v \in C\}  \subset \Pb(\Rb^{2d}), \text{ and} \\
\Lambda_{\star}&:=\{ [g \oplus g] : g \in \GL_d(\Rb), [g] \in \Lambda\} \subset \PGL_{2d}(\Rb). 
\end{align*}
Then, by construction, $(\Omega_{\star}, \Cc_{\star}, \Lambda_{\star})$ is a naive convex co-compact triple. 

\begin{observation} 
\label{obs:ncc_but_not_cc}
$\Cc_{\Omega_{\star}}(\Lambda_{\star}) = \Omega_{\star}$. In particular, $\Lambda_{\star} \leq \Aut(\Omega_{\star})$ is not a convex co-compact subgroup.
\end{observation}

 \begin{proof}
Since $\Lambda$ acts co-compactly on $\Omega$, it is easy to show that $\Lc_\Omega(\Lambda) = \partial \Omega$ (see for instance the proof of Lemma 2.2 in~\cite{B2003b}). By convexity, to prove that $\Cc_{\Omega_{\star}}(\Lambda_{\star}) = \Omega_{\star}$ it is enough to show that 
\begin{align*}
\overline{\Cc_{\Omega_{\star}}(\Lambda_{\star})} \supset \{ [(v,0)] : v \in \partial C\} \cup \{ [(0,v)] : v \in \partial C\}.
\end{align*}
Fix $v \in \partial C$. Then $[v] \in \partial \Omega$ and so there exist $p \in \Omega$ and a sequence $g_n \in \Lambda$ such that $\lim_{n \to \infty} g_n(p) = [v]$. Let $\overline{p} \in C$ be a lift of $p$ and $\overline{g}_n \in \GL_d(\Rb)$ be a lift of $g_n$. Then for any $t > 0$ 
\begin{align*}
[(tv,v)] = \lim_{n \rightarrow \infty}  [\overline{g}_n \oplus \overline{g}_n]  \cdot [(t\overline{p}, \overline{p})] \in \Lc_{\Omega_\star}(\Lambda_\star).
\end{align*}
Taking limits as $t \to 0$ and $t \to \infty$ respectively, 
\begin{align*}
\Big\{ [(v,0)] , [(0,v)] \Big\} \subset \overline{\Lc_{\Omega_\star}(\Lambda_\star)} \subset  \overline{\Cc_{\Omega_{\star}}(\Lambda_{\star})}.
\end{align*}
Since $v \in \partial C$ was arbitrary, the proof is  complete.
 \end{proof}

We can thicken $\Cc_{\star}$ to obtain a one parameter family of naive convex co-compact triples: for $R \geq 0$ define
\begin{align*}
\Cc_{\star}^{(R)} :=   \left\{ y \in \Omega_{\star} : H_{\Omega_{\star}}(y, \Cc_{\star}) \leq R \right\}.
\end{align*}

\begin{observation}\label{obs:bounds}
 For any $R \geq 0$, $\left(\Omega_{\star}, \Cc_{\star}^{(R)}, \Lambda_{\star}\right)$ is a naive convex co-compact triple. 
 \end{observation}
 
 \begin{proof}
 Note that $\Cc_{\star}^{(R)}$ is the closed $R$-neighborhood of the convex set $\Cc_{\star}$ in the Hilbert metric $H_{\Omega_{\star}}$. So $\Cc_{\star}^{(R)}$ is a closed convex set containing $\Cc_{\star}$, see for instance~\cite[Corollary 1.10]{CLT2015}. The observation then follows since $\Lambda_{\star}$ acts co-compactly on $\Cc_{\star}$. 
 \end{proof}

For $x \in \overline{\Omega}$, define $\overline{x} \in \overline{C}$ to be the unique lift of $x$ with $\norm{\overline{x}}_2=1$. Then define $x_{\star}:=[(\overline{x},\overline{x})] \in \overline{\Omega}_{\star}$. 

By definition we have the following description of the faces of $\Omega_\star$. 

\begin{observation}\label{obs:ex_faces} If $x \in \overline{\Omega}$, then 
\begin{align*}
F_{\Omega_\star}(x_{\star}) = \left\{ [(v,w)]  : v,w \in \overline{C} \text{ and } [v], [w] \in F_\Omega(x) \right\}.
\end{align*}
In particular, if $x$ is an extreme point of $\Omega$, then
\begin{align*}
F_{\Omega_\star}(x_{\star}) = \left\{ [(s\overline{x},\overline{x})]  : s \in (0,+\infty) \right\} = \left\{ [(\overline{x},t\overline{x})]  : t \in (0,+\infty) \right\} .
\end{align*}
\end{observation}

Observations~\ref{obs:bounds},~\ref{obs:ex_faces} and the definition of the Hilbert metric imply the following. 
 
 \begin{observation}\label{obs:ex_faces_of_ext_pts} If $x \in \partial\Omega$ is an extreme point and $R > 0$, then 
 \begin{align*}
\overline{ \Cc_{\star}^{(R)}} \cap F_{\Omega_\star}(x_\star) 
 \end{align*}
 is a compact interval containing $x_\star$ with non-empty interior. 
 \end{observation}

We now specialize our assumptions.

\medskip

\noindent \textbf{Additional Assumption:} $\Lambda$ is relatively hyperbolic group with respect to a non-empty collection of virtually Abelian subgroups of rank at least two.

\medskip

Let $\Sc_{\max}$ denote the set of all maximal properly embedded simplices in $\Omega$ of dimension at least two. Then by Theorem~\ref{thm:main_cc} 
\begin{enumerate}
\item $\Sc_{\max}$ is closed and discrete in local Hausdorff topology induced by $H_\Omega$, 
\item $(\Omega, H_\Omega)$ is relatively hyperbolic with respect to $\Sc_{\max}$, and
\item if $S \in \Sc_{\max}$, then each vertex of $S$ is an extreme point of $\Omega$. 
\end{enumerate}
For each $S \in \Sc_{\max}$ define 
\begin{align*}
S_\star := \{ x_\star : x \in S\} \subset \Cc_\star.
\end{align*}
Then $S_\star$ is a maximal properly embedded simplex in $\Cc_\star$. 

\begin{observation} 
 For any $R> 0$, the metric space $(\Cc_\star^{(R)}, H_{\Omega_\star})$ is relatively hyperbolic with respect to $\Sc_\star = \{ S_\star: S \in \Sc_{\max}\}$.
 \end{observation}
\begin{proof}[Proof sketch] Consider the map $F : \Omega \rightarrow \Omega_\star$ defined by $F(x) = x_\star$. Then $F$ induces a quasi-isometry $(\Omega, H_\Omega) \rightarrow (\Cc_\star^{(R)}, H_{\Omega_\star})$ and so the observation follows from the general theory of relatively hyperbolic spaces (see Theorem~\ref{thm:rh_quasi_isometry_inv} below). 
\end{proof}

Next we show that the family of all maximal properly embedded simplices in $\Cc^{(R)}_{\star}$ of dimension at least two is not discrete in the local Hausdorff topology. By construction, if $S \subset \Omega$ is a properly embedded simplex of dimension at least two (i.e. $S \in \Sc_{\max}$), then $S_{\star}$ is a maximal properly embedded simplex in $\Omega_\star$ of the same dimension. Let $v_1,\dots, v_p$ be the vertices of $S$. Then by Lemma~\ref{lem:slide_along_faces} below 
 \begin{align*}
 \Pb(\Spanset\{w_1,\dots,w_p\}) \cap \Omega_{\star}
 \end{align*}
 is a maximal properly embedded simplex in $\Omega_\star$ for any choice of 
 \begin{align*}
 w_j \in \overline{ \Cc_{\star}^{(R)}} \cap F_{\Omega_\star}(v_{j,\star}) \quad j=1,\dots, p.
 \end{align*}
 This construction combined with Observation~\ref{obs:ex_faces_of_ext_pts} yields the following.

\begin{observation} 
\label{obs:ex_parallel_simplices}
For any $R > 0$, the family of maximal properly embedded simplices in $\Cc_{\star}^{(R)}$ of dimension at least two is not discrete. In particular, $\Cc_{\star}^{(R)}$ contains parallel properly embedded simplices (see Definition \ref{defn:parallel}).
\end{observation}

We also can construct the following.

\begin{observation} For any $R > 0$, there exists a family $\Sc^{(R)}_{\prl}$  of maximal properly embedded simplices in $\Cc_{\star}^{(R)}$ of dimension at least two where
\begin{enumerate}
\item $\Sc^{(R)}_{\prl}$ is isolated, coarsely complete, and $\Lambda_\star$-invariant
\item $(\Cc_{\star}^{(R)}, H_{\Omega_\star})$ is not a relatively hyperbolic space with respect to $\Sc^{(R)}_{\prl}$.
\end{enumerate}
\end{observation}

\begin{proof} For each extreme point $x \in \partial \Omega$ define $x^+, x^- \in F_{\Omega_\star}(x_\star)$ to be the points such that
 \begin{align*}
[x^+, x^-]=\overline{ \Cc_{\star}^{(R)}} \cap F_{\Omega_\star}(x_\star).
 \end{align*}
Notice that $x^+ \neq x^-$ by Observation~\ref{obs:ex_faces_of_ext_pts}.
 
Given $S \in \Sc_{\max}$, fix a labeling $v_1,\dots, v_p$ of the vertices of $S$. Then for $\sigma=(\sigma_1,\dots,\sigma_p) \in \{ +,- \}^p$ define 
 \begin{align*}
S_\sigma:=\Pb(\Spanset\{v_1^{\sigma_1},\dots,v_p^{\sigma_p}\}) \cap \Omega_\star.
\end{align*}
Then by Lemma~\ref{lem:slide_along_faces}  below, $S_\sigma$ is a properly embedded simplex in $\Omega_\star$. 

By construction, the set 
\begin{align*}
\Sc^{(R)}_{\prl} :=\left\{ S_\sigma: S \in \Sc_{\max}, \ \sigma \in \{ +,- \}^{\dim S+1}\right\}
\end{align*}
is isolated, coarsely complete, and $\Lambda_\star$-invariant. However, if $S \in \Sc_{\max}$ and $\sigma,\tau \in \{+,-\}^{\dim S+1}$ are distinct, then $S_\sigma \neq S_\tau$ and 
\begin{align*}
H^{\Haus}_\Omega(S_\sigma, S_\tau) < +\infty
\end{align*}
by Lemma~\ref{lem:slide_along_faces} below. So by Theorem~\ref{thm:rh_intersections_of_neighborhoods} below, $(\Cc_{\star}^{(R)}, H_{\Omega_\star})$ is not a relatively hyperbolic space with respect to $\Sc^{(R)}_{\prl}$.
\end{proof}

\subsection{Naive convex co-compact examples II } \label{sec:max_not_coarsely_max}

In this section we construct a naive convex co-compact triple $(\Omega, \Cc, \Lambda)$ where $\Cc$ contains a maximal properly embedded simplex which is contained in a bounded neighborhood of a properly embedded simplex with strictly larger dimension. 

Let $C := \left\{ (x_1,x_2,x_3) \in \Rb^3 : x_1^2 + x_2^2 < x_3^2\right\}$ and $\Bc := \{ [v] : v \in C\}$. Then $(\Bc,H_{\Bc})$ is the Klein-Beltrami model of real hyperbolic 2-space. In particular, if we fix $x,y \in \partial\Bc$ distinct, there exists $h \in \Aut(\Bc)$ which translates along the line segement $(x,y) \subset \Bc$. Let $\overline{x}, \overline{y} \in \partial C$ be lifts of $x,y$ respectively and let $\overline{h} \in \Aut(C)$ be a lift of $h$.  

Next define 
\begin{align*}
\Omega: = \left\{ [(v,w)] : v,w \in C\right\} \subset \Pb(\Rb^6) 
\end{align*}
and let 
\begin{align*}
x_1 :=\left [ \left(\overline{x}, 0\right)\right], \, x_2 := \left [ \left(0, \overline{x}\right)\right], \, y_1 := \left [ \left(\overline{y}, 0\right)\right], \, y_2 :=\left [ \left(0, \overline{y}\right)\right].
\end{align*}
Then $x_1, x_2, y_1, y_2$ are the vertices of a properly embedded simplex $S \subset \Omega$. Further, if 
\begin{align*}
\Lambda := \ip{ \left[ ~\overline{h} \oplus \id\right],\left[ \id \oplus \overline{h} ~\right], \left[ (2 \id) \oplus \left(\frac{1}{2} \id \right) \right] } \leq \Aut(\Omega),
\end{align*}
then $(\Omega, S, \Lambda)$ is a naive convex co-compact triple. 

Fix $R > 0$. Then $\Cc : = \{ p \in \Omega : H_\Omega(p,S) \leq R\}$ is a closed convex subset of $\Omega$, see for instance~\cite[Corollary 1.10]{CLT2015}. So $(\Omega, \Cc, \Lambda)$ is also a naive convex co-compact triple. Now 
\begin{align*}
F:= \{ [(v,0)] : v \in C\}
\end{align*}
is an open boundary face of $\Omega$ and $\overline{\Cc} \cap F$ has non-empty interior in $F$. So there exists 
\begin{align*}
z \in \overline{\Cc} \cap F \setminus (x_1, y_1).
\end{align*}
Then $z,x_2,y_2$ are the vertices of a properly embedded simplex $S^\prime$ in $\Cc$. Further, $S^\prime$ is maximal and $S^\prime$ is contained in a bounded neighborhood of $S$.

\section{Some notations and definitions}\label{sec:notations_and_definitions}

In this section we set some notation that we will use for the rest of the paper. 

\subsection{Basic notations in projective geometry} 

If $V \subset \Rb^d$ is a non-zero linear subspace, we will let $\Pb(V) \subset \Pb(\Rb^d)$ denote its projectivization. In most other cases, we will use $[o]$ to denote the projective equivalence class of an object $o$, for instance: 
\begin{enumerate}
\item if $v \in \Rb^{d} \setminus \{0\}$, then $[v]$ denotes the image of $v$ in $\Pb(\Rb^{d})$, 
\item if $\phi \in \GL_{d}(\Rb)$, then $[\phi]$ denotes the image of $\phi$ in $\PGL_{d}(\Rb)$, and 
\item if $T \in \End(\Rb^{d}) \setminus\{0\}$, then $[T]$ denotes the image of $T$ in $\Pb(\End(\Rb^{d}))$. 
\end{enumerate}
We also identify $\Pb(\Rb^d) = \Gr_1(\Rb^d)$, so for instance: if $x \in \Pb(\Rb^d)$ and $V \subset \Rb^d$ is a linear subspace, then $x \in \Pb(V)$ if and only if $x \subset V$.

Given a subset $X$ of $\Rb^d$ or $\Pb(\Rb^d)$ we will let $\Spanset X \subset \Rb^d$ denote the smallest linear subspace containing $X$. 

Next, for a subset $X \subset \Pb(\Rb^d)$ we define the \emph{automorphism group of $X$} to be 
\begin{align*}
\Aut(X) : = \{ g \in \PGL_d(\Rb): g X = X\}.
\end{align*}
Further, given a group $G \leq \PGL_d(\Rb)$ the \emph{stabilizer of $X$ in $G$} is
\begin{align*}
\Stab_G(X) := \{ g \in G : g X = X\} = G \cap \Aut(X). 
\end{align*}

\subsection{Convexity} 

\begin{definition} \ 
\begin{enumerate}
\item A subset $C \subset \Pb(\Rb^d)$ is \emph{convex} if there exists an affine chart $\mathbb{A}$ of $\Pb(\Rb^d)$ where $C \subset \mathbb{A}$ is a convex subset. 
\item A subset $C \subset \Pb(\Rb^d)$ is \emph{properly convex} if there exists an affine chart $\mathbb{A}$ of $\Pb(\Rb^d)$ where $C \subset \mathbb{A}$ is a bounded convex subset. 
\item An open subset of $\Pb(\Rb^d)$ that is properly convex is called a \emph{properly convex domain} in $\Pb(\Rb^d)$.
\end{enumerate}
\end{definition}

Notice that if $C \subset \Pb(\Rb^d)$ is convex, then $C$ is a convex subset of every affine chart that contains it. 

A \emph{line segment} in $\Pb(\Rb^{d})$ is a connected subset of a projective line. Given two points $x,y \in \Pb(\Rb^{d})$ there is no canonical line segment with endpoints $x$ and $y$, but we will use the following convention: if $C \subset \Pb(\Rb^d)$ is a properly convex set and $x,y \in \overline{C}$, then when the context is clear we will let $[x,y]$ denote the closed line segment joining $x$ to $y$ which is contained in $\overline{C}$. In this case, we will also let $(x,y)=[x,y]\setminus\{x,y\}$, $[x,y)=[x,y]\setminus\{y\}$, and $(x,y]=[x,y]\setminus\{x\}$.

Along similar lines, given a properly convex subset $C \subset \Pb(\Rb^d)$ and a subset $X \subset \overline{C}$ we will let 
\begin{align*}
{\rm ConvHull}_C(X)
\end{align*}
 denote the smallest convex subset of $\overline{C}$ which contains $X$. For instance, with our notation  $[x,y] = {\rm ConvHull}_{C}(\{x,y\})$ when $x,y \in \overline{C}$. 

We also make the following topological definitions.

\begin{definition}\label{defn:topology} Suppose $C \subset \Pb(\Rb^d)$ is a convex set. The \emph{relative interior of $C$}, denoted by $\relint(C)$, is  the interior of $C$ in $\Pb(\Spanset C)$. In the case that $C = \relint(C)$, we will say that $C$ is  \emph{open in its span}. The \emph{boundary of $C$} is $\partial C : = \overline{C} \setminus \relint(C)$, the \emph{ideal boundary of $C$} is
\begin{align*}
\partiali C := \partial C \setminus C,
\end{align*}
and the \emph{non-ideal boundary of $C$} is
\begin{align*}
\partialni C := \partial C \cap C
\end{align*}
Finally, we define $\dim C$ to be the dimension of $\relint(C)$ (notice that $\relint(C)$ is homeomorphic to $\Rb^{\dim C}$). 
\end{definition}

Recall that a subset $A \subset B \subset \Pb(\Rb^d)$ is properly embedded if the inclusion map $A \hookrightarrow B$ is proper. With the notation in Definition~\ref{defn:topology} we have the following characterization of properly embedded subsets.

\begin{observation} Suppose $C \subset \Pb(\Rb^d)$ is a convex set. A convex subset $S \subset C$ is properly embedded if and only if $\partiali S \subset \partiali C$. \end{observation}

We also recall the definition of supporting hyperplanes. 

\begin{definition}\label{defn:supp_hyp_C1_point} Suppose $\Omega \subset \Pb(\Rb^d)$ is a properly convex domain. 
\begin{enumerate}
\item A projective hyperplane $H \subset \Pb(\Rb^d)$ is a \emph{supporting hyperplane of $\Omega$} if $H \cap \Omega = \emptyset$ and $H \cap \partial \Omega \neq \emptyset$. 
\item A boundary point $x \in \partial \Omega$ is a \emph{$C^1$-smooth point} of $\partial \Omega$ if there exists a unique supporting hyperplane containing $x$.  
\end{enumerate}
\end{definition}

\subsection{The Hilbert metric}\label{subsec:Hilbert_metric}

For distinct points $x,y \in \Pb(\Rb^{d})$ let $\overline{xy}$ be the projective line containing them. Suppose $C \subset \Pb(\Rb^{d})$ is a properly convex set which is open in its span. If $x,y \in C$ are distinct, let $a,b$ be the two points in $\overline{xy} \cap \partial C$ ordered $a, x, y, b$ along $\overline{xy}$. Then define \emph{the Hilbert distance between $x$ and $y$ to be}
\begin{align*}
H_{C}(x,y) = \frac{1}{2}\log [a, x,y, b]
\end{align*}
 where 
 \begin{align*}
 [a,x,y,b] = \frac{\abs{x-b}\abs{y-a}}{\abs{x-a}\abs{y-b}}
 \end{align*}
 is the cross ratio. We also define $H_{C}(x,x)=0$ for all $x \in C$. Using the invariance of the cross ratio under projective maps and the convexity of $C$ it is possible to establish the following (see for instance~\cite[Section 28]{BK1953}). 
 
 \begin{proposition}\label{prop:hilbert_basic}
Suppose $C \subset \Pb(\Rb^{d})$ is a properly convex set which is open in its span. Then $H_{C}$ is a complete $\Aut(C)$-invariant metric on $C$ which generates the standard topology on $C$. Moreover, if $p,q \in C$, then there exists a geodesic joining $p$ and $q$ whose image is the line segment $[p,q]$.
\end{proposition}

As a corollary to Proposition~\ref{prop:hilbert_basic}, we observe the following. 

\begin{corollary} 
Suppose $\Omega \subset \Pb(\Rb^d)$ is a properly convex domain. Then $\Aut(\Omega)$ acts properly on $\Omega$. 
\end{corollary} 

We will frequently use the following notation. 

\begin{definition} Suppose $\Omega \subset \Pb(\Rb^d)$ is a properly convex domain. 
\begin{enumerate}
\item For $x \in \Omega$ and $r > 0$ define 
\begin{align*}
B_\Omega(x;r) = \{ y \in \Omega : H_\Omega(x,y) < r\}.
\end{align*}
\item For a subset $A\subset \Omega$ and $r > 0$ define 
\begin{align*}
\Nc_\Omega(A;r) = \cup_{x \in A} B_\Omega(x;r).
\end{align*}
\item For a subset $A\subset \Omega$ define
\begin{align*}
\diam_\Omega(A) = \sup\left\{ H_\Omega(x,y) : x,y \in A\right\}.
\end{align*}
\end{enumerate}
\end{definition}

\subsection{The center of mass of a compact subset}

It is possible to define a ``center of mass'' for a compact set in a properly convex domain. Let $\Kc_d$ denote the set of all pairs $(\Omega, K)$ where $\Omega \subset \Pb(\Rb^d)$ is a properly convex domain and $K \subset \Omega$ is a compact subset. 

\begin{proposition}\label{prop:center_of_mass} There exists a function
\begin{align*}
(\Omega, K) \in \Kc_d \longrightarrow {\rm CoM}_\Omega(K) \in \Pb(\Rb^d)
\end{align*}
such that:
\begin{enumerate}
\item ${\rm CoM}_\Omega(K)  \in {\rm ConvHull}_\Omega(K)$, 
\item ${\rm CoM}_\Omega(K) = {\rm CoM}_\Omega({\rm ConvHull}_\Omega(K))$, and
\item if $g \in \PGL_d(\Rb)$, then $g{\rm CoM}_\Omega(K)={\rm CoM}_{g\Omega}(gK)$,
\end{enumerate}
for every $(\Omega, K) \in \Kc_d$. 
\end{proposition} 

There are several constructions of such a center of mass, see for instance~\cite[Lemma 4.2]{L2014} or~\cite[Proposition 4.5]{IZ2019}. The approach in~\cite{IZ2019} is based on an argument of Frankel~\cite[Section 12]{Fra1989} in several complex variables. 

\subsection{The faces of a convex domain} In this section we define the faces of a convex set and then describe some of their properties. 

\begin{definition}\label{defn:open_faces}
Given a properly convex domain $\Omega \subset \Pb(\Rb^d)$ and $x \in \overline{\Omega}$ let $F_\Omega(x)$ denote the \emph{open face} of $x$, that is 
\begin{align*}
F_\Omega(x) = \{ x\} \cup \left\{ y \in \overline{\Omega} : \text{ $\exists$ an open line segment in $\overline{\Omega}$ containing $x$ and $y$} \right\}.
\end{align*}
\end{definition}

\begin{remark}
Notice that $F_\Omega(x) = \Omega$ when $x \in \Omega$.  Further, note that a properly convex set $C$ that is open in its span is a properly convex domain in $\Pb(\Spanset C)$. Thus, the above defintion (and the results of this subsection) apply to any properly convex set open in its span.
\end{remark}

We also introduce the following notation. 

\begin{definition}\label{defn:faces_of_subsets}
Given a properly convex domain $\Omega \subset \Pb(\Rb^d)$ and a subset $X \subset \overline{\Omega}$ define
\begin{align*}
F_\Omega(X) := \cup_{x \in X} F_\Omega(x).
\end{align*}
\end{definition}

The next observation is a simple consequence of convexity. 

\begin{observation}\label{obs:faces} Suppose $\Omega \subset \Pb(\Rb^d)$ is a properly convex domain. 
\begin{enumerate}
\item $F_\Omega(x)$ is open in its span,
\item $y \in F_\Omega(x)$ if and only if $x \in F_\Omega(y)$ if and only if $F_\Omega(x) = F_\Omega(y)$,
\item if $y \in \partial F_\Omega(x)$, then $F_\Omega(y) \subset \partial F_\Omega(x)$,
\item if $x, y \in \overline{\Omega}$, $z \in (x,y)$, $p \in F_{\Omega}(x)$, and $q \in F_{\Omega}(y)$, then 
\begin{align*}
(p,q) \subset F_\Omega(z).
\end{align*}
In particular, $(p,q) \subset \Omega$ if and only if $(x,y) \subset \Omega$.
\end{enumerate}
\end{observation}

The next two results relate the faces to the Hilbert metric. 

\begin{proposition}\label{prop:dist_est_and_faces} Suppose $\Omega \subset \Pb(\Rb^d)$ is a properly convex domain, $x_n$ is a sequence in $\Omega$, and $\lim_{n \to \infty} x_n =x \in \overline{\Omega}$. If $y_n$ is another sequence in $\Omega$, $\lim_{n \to \infty} y_n =  y \in \overline{\Omega}$, and 
\begin{align*}
\liminf_{n \rightarrow \infty} H_\Omega(x_n,y_n) < + \infty,
\end{align*}
then $y \in F_\Omega(x)$ and 
\begin{align*}
H_{F_\Omega(x)}(x,y) \leq \liminf_{n \rightarrow \infty} H_\Omega(x_n,y_n).
\end{align*}
\end{proposition}

\begin{proof} Straightforward consequence of the definition of the Hilbert metric. \end{proof}

\begin{proposition}
\label{prop:Crampons_dist_est_2} Suppose $\Omega \subset \Pb(\Rb^d)$ is a properly convex domain. Assume $p_1,p_2,q_1,q_2 \in \overline{\Omega}$, $F_\Omega(p_1) = F_\Omega(p_2)$, and $F_\Omega(q_1) = F_\Omega(q_2)$. If $(p_1,q_1) \cap \Omega \neq \emptyset$, then
\begin{align*}
H_{\Omega}^{\Haus}\Big((p_1,q_1), (p_2, q_2) \Big) \leq \max \left\{ H_{F_\Omega(p_1)}(p_1,p_2), H_{F_\Omega(q_1)}(q_1,q_2) \right\}.
\end{align*}
\end{proposition}

\begin{remark} Notice that Observation~\ref{obs:faces} implies that $(p_1,q_1), (p_2,q_2) \subset \Omega$. \end{remark}

\begin{proof} This is a straightforward consequence of the fact that a $R$-neighborhood of a closed convex set in the Hilbert metric is convex (see for instance~\cite[Corollary 1.10]{CLT2015}). For details see~\cite[Proposition 5.3]{IZ2019}. \end{proof}

The final result of this subsection relates the faces to the behavior of automorphisms. 

\begin{proposition}\cite[Proposition 5.6]{IZ2019}\label{prop:dynamics_of_automorphisms}
Suppose $\Omega \subset \Pb(\Rb^d)$ is a properly convex domain, $p_0 \in \Omega$, and $g_n \in \Aut(\Omega)$ is a sequence such that 
\begin{enumerate}
\item $\lim_{n \to \infty} g_n (p_0) = x \in \partial \Omega$, 
\item $\lim_{n \to \infty} g_n^{-1} (p_0) = y \in \partial \Omega$, and
\item $g_n$ converges in $\Pb(\End(\Rb^d))$ to $T \in \Pb(\End(\Rb^d))$. 
\end{enumerate}
Then ${\rm image}(T) \subset \Spanset \{ F_\Omega(x)\}$, $\Pb(\ker T) \cap \Omega = \emptyset$, and $y \in \Pb(\ker T)$. 
\end{proposition}

\subsection{Parallel properly embedded simplices} 

\begin{definition}\label{defn:parallel} Suppose $\Omega \subset \Pb(\Rb^d)$ is a properly convex domain. Two properly embedded simplices $S_1, S_2 \subset \Omega$ are called \emph{parallel} if $\dim S_1 = \dim S_2\geq 1$ and there is a labeling  $v_1,\dots, v_p$ of the vertices of $S_1$ and a labeling $w_1,\dots, w_p$ of the vertices of $S_2$ such that $F_\Omega(v_k) = F_\Omega(w_k)$ for all $1 \leq k \leq p$. 
\end{definition} 

The following lemma allows us to ``wiggle'' the vertices of a properly embedded simplex and obtain a new parallel properly embedded simplex. 

\begin{lemma}\label{lem:slide_along_faces} Suppose that $\Omega \subset \Pb(\Rb^d)$ is a properly convex domain and $S \subset \Omega$ is a properly embedded simplex with vertices $v_1,\dots, v_p$. If $w_j \in F_\Omega(v_j)$ for $1 \leq j \leq p$, then 
\begin{align*}
S^\prime : = \Omega \cap \Pb(\Spanset\{ w_1,\dots, w_p\})
\end{align*}
is a properly embedded simplex with vertices $w_1,\dots,w_p$. Moreover, 
\begin{align*}
H_\Omega^{\Haus}\left(S, S^\prime\right) \leq \max_{1 \leq j \leq p} H_{F_\Omega(v_j)}(v_j, w_j). 
\end{align*}
\end{lemma}

\begin{proof} For $1 \leq q \leq p$ and $1 \leq i_1 < i_2 < \dots < i_q \leq p$ define
\begin{align*}
S(i_1,\dots,i_q) : = \relint\left( \CH_\Omega\{ v_{i_1},\dots, v_{i_q}\}\right)
\end{align*}
and
\begin{align*}
S^\prime(i_1,\dots,i_q) : = \relint\left( \CH_\Omega\{ w_{i_1},\dots, w_{i_q}\}\right).
\end{align*}
We claim that for each  $1 \leq q \leq p$ and $1 \leq i_1 < i_2 < \dots < i_q \leq p$ there exists a face $F(i_1,\dots,i_q)$ of $\Omega$ such that 
\begin{align*}
S(i_1,\dots,i_q) \cup S^\prime(i_1,\dots,i_q) \subset F(i_1,\dots,i_q).
\end{align*}
When $q=1$ this is by hypothesis. Then the claim follows from induction on $q$ and Observation~\ref{obs:faces} part (4).

Then $S^\prime(i_1,\dots,i_q)  \subset \partial \Omega$ if and only if $S(i_1,\dots,i_q)  \subset \partial \Omega$. Hence $S^\prime$ is a properly embedded simplex. 

The moreover part follows from a similar induction argument and Proposition~\ref{prop:Crampons_dist_est_2}. 
\end{proof}

\subsection{The Hausdorff distance and local Hausdorff topology}When $(X,\dist)$ is a metric space, the \emph{Hausdorff distance} between two subsets $A,B \subset X$ is defined by 
\begin{align*}
\dist^{\Haus}(A,B) = \max \left\{ \sup_{a \in A} \inf_{b \in B} \dist(a,b), \ \sup_{b \in B} \inf_{a \in A} \dist(a,b) \right\}.
\end{align*}
When $(X,\dist)$ is a complete metric space space,  $d^{\Haus}$ is a complete metric on the set of non-empty compact subsets of $X$. 

The local Hausdorff topology is a natural topology on the set of closed sets in $X$. For a closed set $C_0$, a base point $x_0 \in X$, and $r_0, \epsilon_0 > 0$ define the set $U(C_0,x_0,r_0,\epsilon_0)$ to consist of all closed subsets $C \subset X$ where
\begin{align*}
\dist^{\Haus}\Big(C_0 \cap B_X(x_0;r_0),\, C \cap B_X(x_0;r_0)\Big) < \epsilon_0
\end{align*}
where 
\begin{align*}
B_X(x_0;r_0) =\{ x\in X : \dist(x_0,x) < r_0\}.
\end{align*}
The \emph{local Hausdorff topology} on the set of non-empty closed subsets of $X$ is the topology generated by the sets $U(\cdot, \cdot, \cdot, \cdot)$. 

\subsection{The local Hausdorff topology on slices of a properly convex domain}

Fix a distance $d_{\Pb}$ on $\Pb(\Rb^d)$ induced by a Riemannian metric. Then the following observation is a consequence of convexity and the definition of the Hilbert metric.

\begin{observation}\label{obs:slice_convergence} Suppose $\Omega \subset \Pb(\Rb^d)$ is a properly convex domain. If $V_n \in \Gr_p(\Rb^d)$ is a sequence of $p$-dimensional linear subspaces,  $V_n \rightarrow V$ in $\Gr_{p}(\Rb^d)$, and $\Pb(V) \cap \Omega \neq \emptyset$, then
\begin{enumerate}
\item $\overline{\Pb(V_n) \cap \Omega}$ converges to $\overline{\Pb(V) \cap \Omega}$ in the Hausdorff topology induced by $d_{\Pb}$,
\item $\Pb(V_n) \cap \Omega$ converges to $\Pb(V) \cap \Omega$ in the local Hausdorff topology induced by $H_\Omega$. 
\end{enumerate}
\end{observation}

\begin{proof} To prove (1), notice that the set of non-empty compact subsets in $\Pb(\Rb^d)$ endowed with $d_{\Pb}^{\Haus}$ is a compact metric space. Hence to show that $\overline{\Pb(V_n) \cap \Omega}$ converges to $\overline{\Pb(V) \cap \Omega}$ it suffices to show that every convergent subsequence of $\overline{\Pb(V_n) \cap \Omega}$ converges to $\overline{\Pb(V) \cap \Omega}$. So suppose that $\overline{\Pb(V_{n_j}) \cap \Omega}$ converges to some compact set $C$. Since $\Omega$ is open, it is clear that $\overline{\Pb(V) \cap \Omega} \subset C$. Since every point in $\partial \Omega$ has a supporting hyperplane it is clear that $C \subset \overline{\Pb(V) \cap \Omega}$. Thus $\overline{\Pb(V_{n_j}) \cap \Omega}$ converges to $\overline{\Pb(V) \cap \Omega}$. 

The proof of (2) is similar. 
\end{proof}

This observation has the following consequence. 

\begin{observation}\label{obs:PES_closed} Suppose $\Omega \subset \Pb(\Rb^d)$ is a properly convex domain. Then the set of properly embedded simplices in $\Omega$ of dimension at least two is closed in the local Hausdorff topology. 
\end{observation}

\begin{proof} 
Suppose $S_n \subset \Omega$ is a sequence of properly embedded simplices of dimension at least two and $S_n$ converges to $S$ in the local Hausdorff topology. We claim that $S$ is also a properly embedded simplex. By passing to a sequence we can suppose that $\dim S_n = p-1 \geq 2$ for all $n$.

Let $V_n = \Spanset S_n$, then $S_n = \Pb(V_n) \cap \Omega$. By passing to a subsequence we can suppose that 
\begin{align*}
V:=\lim_{n \rightarrow \infty} V_n
\end{align*}
exists in $\Gr_{p}(\Rb^d)$. Then Observation~\ref{obs:slice_convergence} part (2) implies that $S = \Pb(V) \cap \Omega$. Next let $v_1^{(n)}, \dots, v_p^{(n)}$ be the vertices of $S_n$. Then 
\begin{align}
\label{eq:S_n_vertices}
S_n = \relint \left( \CH_\Omega\left\{v_1^{(n)},\dots, v_p^{(n)}\right\} \right).
\end{align}

By passing to a subsequence we can suppose that $v_1^{(n)}, \dots, v_p^{(n)}$ converge to $v_1,\dots, v_p$. Then Observation~\ref{obs:slice_convergence} part (1) and Equation~\eqref{eq:S_n_vertices} imply that 
\begin{align*}
S = \Pb(V) \cap \Omega = \relint\left( \CH_\Omega\left\{v_1,\dots, v_p\right\}\right).
\end{align*}
Since $\dim \Pb(V) \cap \Omega = (p-1)$, the lines $v_1,\dots, v_p$ must be linearly independent. Hence $S$ is a properly embedded simplex. 
\end{proof}

As a consequence of Observation~\ref{obs:PES_closed} we will show the following. 

\begin{observation}\label{obs:str_iso_implies_iso} Suppose $\Omega \subset \Pb(\Rb^d)$ is a properly convex domain and $\Sc$ is a family of maximal properly embedded simplices of dimension at least two. If $\Sc$ is strongly isolated, then it is also isolated. 
\end{observation}

\begin{proof} 
Suppose $S_n \in \Sc$ is a sequence and $S_n$ converges to $S$ in the local Hausdorff topology induced by $H_{\Omega}$. Then $S$ is a properly embedded simplex by Observation~\ref{obs:PES_closed}. Since $S$ is unbounded in $(\Omega, H_\Omega)$, for any $r > 0$ 
\begin{align*}
\lim_{n \rightarrow \infty} \diam_\Omega \Big( \Nc_\Omega(S_n; r) \cap \Nc_\Omega(S_{n+1}; r) \Big) =\infty.
\end{align*}
Then, since $\Sc$ is strongly isolated, there exists $N \geq 0$ such that $S_n = S_m$ for all $n,m \geq N$. Then $S = S_N$. So $\Sc$ is closed and discrete in  the local Hausdorff topology.
\end{proof}

\section{Background on relatively hyperbolic metric spaces}\label{sec:background_on_rel_hyp}

In this section we recall the definition of relatively hyperbolic metric spaces and then state a useful characterization due to Sisto. 

\begin{definition} Suppose $\omega$ is a non-principal ultrafilter, $(X,\dist)$ is a metric space, $(x_n)$ is a sequence of points in $X$, and $(\lambda_n)$ is a sequence of positive numbers with $\lim_{\omega} \lambda_n = \infty$. The \emph{asymptotic cone of $X$ with respect to $(x_n)$ and $(\lambda_n)$}, denoted by $C_\omega(X,x_n,\lambda_n)$, is the ultralimit $\lim_{\omega} (X,\lambda_n^{-1} \dist, x_n)$. 
\end{definition}

For more background on asymptotic cones, see~\cite{D2002}. 

\begin{definition}[{Dru{\c t}u-Sapir~\cite[Definition 2.1]{DS2005}}] Let $(X,\dist)$ be a complete geodesic metric space and let $\Sc$ be a collection of closed geodesic subsets (called \emph{pieces}). Suppose that the following two properties are satisfied:
\begin{enumerate}
\item Every two different pieces have at most one common point. 
\item Every simple geodesic triangle (a simple loop composed of three geodesics) in $X$ is contained in one piece. 
\end{enumerate}
Then we say that the metric space $(X,\dist)$ is \emph{tree-graded with respect to $\Sc$}. 
\end{definition}

\begin{definition} A a complete geodesic metric space $(X,\dist)$ is \emph{relatively hyperbolic with respect to a collection of subsets $\Sc$} if all its asymptotic cones, with respect to a fixed non-principal ultrafilter, are tree-graded with respect to the collection of ultralimits of the elements of $\Sc$. 
\end{definition}

\begin{definition}\label{defn:rel-hyp-group}
A finitely generated group $G$ is \emph{relatively hyperbolic with respect to a family of subgroups $\{H_1,\dots, H_k\}$} if the Cayley graph of $G$ with respect to some (hence any) finite set of generators is relatively hyperbolic with respect to the collection of left cosets $\{g H_i : g \in G, i=1,\dots,k\}$. 
\end{definition}

\begin{remark}This is one among several equivalent definitions of relatively hyperbolic groups. We direct the interested reader to \cite{DS2005} and the references therein for more details. 

\end{remark}

We now recall some basic properties of relatively hyperbolic metric spaces.  Given a metric space $(X,\dist)$, a subset $A \subset X$, and $r > 0$ define
\begin{align*}
\Nc_X(A;r) = \{ x \in X: \dist(x,a) < r \text{ for some } a \in A\}.
\end{align*}

\begin{theorem}[{Dru{\c t}u-Sapir~\cite[Theorem 4.1]{DS2005}}]\label{thm:rh_intersections_of_neighborhoods} Suppose $(X,\dist)$ is relatively hyperbolic with respect to $\Sc$. For any $r > 0$ there exists $Q(r) > 0$ such that: if $S_1, S_2 \in \Sc$ are distinct, then 
\begin{align*}
\diam_X \big( \Nc_X(S_1;r)\cap \Nc_X(S_2;r) \big) \leq Q(r).
\end{align*}
\end{theorem}

The next two results involve quasi-isometric embeddings. 

\begin{definition} Suppose $(X,\dist_X)$, $(Y,\dist_Y)$ are complete geodesic metric spaces. A map $f: X \rightarrow Y$ is a \emph{$(A,B)$-quasi-isometric embedding} if 
\begin{align*}
\frac{1}{A} \dist_X(x_1,x_2) - B \leq \dist_Y(f(x_1),f(x_2)) \leq A \dist_X(x_1,x_2) + B
\end{align*}
for all $x_1, x_2 \in X$. If, in addition, there exists a quasi-isometry $g : Y \rightarrow X$ and $R > 0$ such that 
\begin{align*}
\dist_X( x, (g \circ f)(x)) \leq R
\end{align*}
for all $x \in X$ and 
\begin{align*}
\dist_Y( y, (f \circ g)(y)) \leq R
\end{align*}
for all $y \in Y$, then $f$ is a \emph{$(A,B)$-quasi-isometry}. 
\end{definition} 

Being relatively hyperbolic is invariant under quasi-isometries. 

\begin{theorem}[{Dru{\c t}u-Sapir~\cite[Theorem 5.1]{DS2005}}]\label{thm:rh_quasi_isometry_inv} Suppose $(X, \dist_X)$, $(Y, \dist_Y)$ are complete geodesic metric spaces and $f: X \rightarrow Y$ is a quasi-isometry. Then $(X, \dist_X)$ is relatively hyperbolic with respect to $\Sc$ if and only if $(Y, \dist_Y)$ is relatively hyperbolic with respect to $f(\Sc)$.
\end{theorem} 

Being relatively hyperbolic also constrains the possible quasi-isometric embeddings of $\Rb^k$ when $k \geq 2$. 

\begin{theorem}[{Dru{\c t}u-Sapir~\cite[Corollary 5.8]{DS2005}}]\label{thm:rh_embeddings_of_flats} Suppose $(X, \dist_X)$ is relatively hyperbolic with respect to $\Sc$. Then for any $A \geq 1$ and $B \geq 0$ there exists $M =M(A,B)$ such that: if $k \geq 2$ and $f : \Rb^k \rightarrow X$ is an $(A,B)$-quasi-isometric embedding, then there exists some $S \in \Sc$ such that 
\begin{align*}
f(\Rb^k) \subset \Nc_X(S;M).
\end{align*}
\end{theorem} 

We end this section by describing a useful characterization of relative hyperbolicity due to Sisto. 

\begin{definition}
Let $(X, \dist)$ be a complete geodesic metric space and $\Sc$ a collection of subsets of $X$.
A family of maps $\Pi=\{\pi_S : X \to S \}_{S \in \Sc}$ is an \emph{almost-projection system for $\Sc$} if there exists $C>0$ such that for all $S \in \Sc$:
\begin{enumerate}
\item If $x \in X$ and $p \in S$, then
\begin{align*}
\dist(x,p) \geq \dist(x,\pi_S(x))+ \dist(\pi_S(x),p)-C,
\end{align*}
\item $\diam_X \pi_S(S')  \leq C$ for all $S, S' \in \Sc$ distinct, and
\item if $x \in X$ and $\dist(x,S)=R$, then $\diam_X \pi_S(B(x;R)) \leq C$.
\end{enumerate}
\end{definition}

In a relatively hyperbolic metric space, almost projection systems appear naturally. 

\begin{theorem}[{Sisto~\cite[Theorem 2.14]{S2013}}]\label{thm:rel-hyp-sisto}
 Suppose $(X,\dist)$ is relatively hyperbolic with respect to a collection $\Sc$. For each $S \in \Sc$ and $x \in X$, let $\pi_S(x)$ be any point in $S$ satisfying 
\begin{align*}
\dist(x,\pi_S(x)) \leq \dist(x,S) + 1,
\end{align*}
then $\Pi=\{\pi_S : X \to S \}_{S \in \Sc}$ is an almost-projection system for $\Sc$.
\end{theorem}

To obtain a characterization of relatively hyperbolicity in terms of the existence of almost-projection systems, one needs an additional property. 

\begin{definition} Let $(X,\dist)$ be a complete geodesic metric space. A collection of geodesics $\Gc$ is a \emph{geodesic path system} if 
\begin{enumerate}
\item[(1)] for every distinct $x_1, x_2 \in X$ there exists a geodesic in $\Gc$ whose endpoints are $x_1$ and $x_2$.
\item[(2)] if $\alpha \in \Gc$, then every sub-path of $\alpha$ is also in $\Gc.$ 
\end{enumerate}
\end{definition}

\begin{definition}\label{defn:asym-trans-free}
 Let $(X,\dist)$ be a complete geodesic metric space and $\Sc$ a collection of subsets of $X$.
\begin{enumerate}
\item A geodesic triangle $\Tc$ in $X$ is \emph{$\Sc$-almost-transverse with constants $\kappa$ and  $\triangleD$} if
\begin{align*}
\diam_X(\Nc_X(S;\kappa) \cap \gamma) \leq \triangleD
\end{align*}
for every $S \in \Sc$ and edge $\gG$ of $\Tc$.
 \item The collection $\Sc$ is \emph{asymptotically transverse-free} if there exist $\lambda,~\sigma$ such that for each $\triangleD \geq 1, ~\kappa \geq \sG$ the following holds: if $\Tc$ is a geodesic triangle in $X$ which is $\Sc$-almost-transverse with constants $\kappa$ and $\triangleD$, then $\Tc$ is $(\lambda \triangleD)$-thin.
 \item The collection $\Sc$ is \emph{asymptotically transverse-free relative to a geodesic path system $\Gc$} if there exist $\lambda,~\sigma$ such that for each $\triangleD \geq 1, \kappa \geq \sG$ the following holds: if $\Tc$ is a geodesic triangle in $X$ whose sides are in $\Gc$ and is $\Sc$-almost-transverse with constants $\kappa$ and $\triangleD$, then $\Tc$ is $(\lambda \triangleD)$-thin.
 \end{enumerate}
\end{definition}

\begin{observation}\label{obs:almost-transverse}
Suppose $\Tc$ is geodesic triangle that is $\Sc$-almost-transverse with constants $\kappa$ and $\triangleD$.  If $\kappa' < \kappa$ and $\triangleD' > \triangleD$, then $\Tc$ is $\Sc$-almost-transverse with constants $\kappa'$ and $\triangleD'.$ 
\end{observation}

\begin{theorem}[{Sisto~\cite[Theorem 2.14]{S2013}}]\label{thm:Sisto_equiv} Let $(X,\dist)$ be a complete geodesic metric space and $\Sc$ a collection of subsets of $X$. Then the following are equivalent: 
\begin{enumerate}
\item $X$ is relatively hyperbolic with respect to $\Sc$,
\item $\Sc$ is asymptotically transverse-free and there exists an almost-projection system for $\Sc$,
\item $\Sc$ is asymptotically transverse-free relative to a geodesic path system and there exists an almost-projection system for $\Sc$,
\end{enumerate}
\end{theorem}

To be precise, Sisto~\cite[Theorem 2.14]{S2013} only explicitly proves that (1) and (2) are equivalent, however his proof can be adapted to show the equivalence of (2) and (3). We will explain how in Appendix~\ref{sec:pf_of_thm_Sisto}.

\part{General remarks about properly embedded simplices}

\section{Basic properties of simplices}\label{sec:basic_properties_of_simplices}

In this section we explain some properties of simplices that we will use throughout the paper. We begin by considering the standard open simplex in $\Pb(\Rb^{d})$. 

\begin{example}\label{ex:basic_properties_of_simplices} Let 
\begin{align*}
S = \left\{ [x_1:\dots:x_{d}] \in \Pb(\Rb^{d}) : x_1>0, \dots, x_{d}> 0\right\}.
\end{align*}
Then $S$ is a $(d-1)$-dimensional simplex. Let $G \leq \GL_d(\Rb)$ denote the group generated by the group of diagonal matrices with positive entries and the group of permutation matrices. Then 
\begin{align*}
\Aut(S) = \{ [g] \in \PGL_d(\Rb) : g \in G\}.
\end{align*}
The Hilbert metric on $S$ can be explicitly computed as:
\begin{align*}
H_S\Big([x_1:\dots:x_{d}], [y_1:\dots:y_{d}] \Big) =\max_{1\leq i,j \leq d} \frac{1}{2} \abs{\log \frac{x_i y_j}{y_i x_j}}.
\end{align*}
For details, see~\cite[Proposition 1.7]{N1988}, ~\cite{dlH1993}, or ~\cite{V2014}.
\end{example}

We also observe the following. 

\begin{observation}\label{obs:cocompact-action-on-simplices}
Suppose $S \subset \Pb(\Rb^d)$ is a simplex and $H \leq \Aut(S)$ acts co-compactly on $S$. Then:
\begin{enumerate}
\item  If $H_0 \leq H$ is the subgroup of elements that fix the vertices of $S$, then $H_0$ also acts co-compactly on $S$.
\item If $F \subset \partial S$ is a face of $S$, then $\Stab_H(F)$ acts co-compactly on $F$. 
\end{enumerate}
\end{observation}

\begin{proof}
Notice that $H_0$ is a finite index subgroup of $H$ and so (1) follows. 

By changing coordinates we can assume that 
\begin{align*}
S &= \{[x_1: \ldots:x_k:0:\ldots:0]: x_1>0, \ldots, x_k>0\} \text{ and } \\
F&=\{[x_1: \ldots:x_\ell:0:\ldots:0]: x_1>0, \ldots, x_\ell>0\}
\end{align*}
where $\ell < k$. 

Consider the onto map $\pi : S \rightarrow F$ which projects to the first $\ell$ coordinates. Then $h \circ \pi = \pi \circ h$ for all $h \in H_0$. By (1) there exists a compact set $K \subset S$ with $H_0 \cdot K = S$. Then $\pi(K) \subset F$ is compact and 
\begin{align*}
H_0 \cdot \pi(K) = \pi(H_0 \cdot K) = \pi(S) = F.
\end{align*}
So $H_0$ acts co-compactly on $F$. Since $H_0 \leq \Stab_H(F)$ this proves (2). 
\end{proof}

\begin{observation}\label{obs:QI_to_Rk}
If $S \subset \Pb(\Rb^d)$ is a simplex, then $(S,H_S)$ is quasi-isometric to real Euclidean $(\dim S)$-space. 
\end{observation}

\begin{proof} By replacing $\Rb^d$ with $\Spanset S$ and picking suitable coordinates we can assume that 
\begin{align*}
S = \left\{ [x_1:\dots:x_{k+1}] \in \Pb(\Rb^{k+1}) : x_1>0, \dots, x_{k+1}> 0\right\}.
\end{align*}
Next consider the map $\Phi: S \rightarrow \Rb^{k}$ defined by
\begin{align*}
\Phi\Big([x_1:\dots:x_{k+1}]\Big) = \left( \log \frac{x_2}{x_1}, \dots,  \log \frac{x_{k+1}}{x_1} \right)
\end{align*}
and define a distance $d$ on $\Rb^{k}$ by 
\begin{align*}
d(v,w) = \frac{1}{2} \max\left\{ \max_{1 \leq i \leq k} \abs{v_i-w_i}, \max_{1\leq i,j \leq k} \abs{(v_i-v_j)-(w_i-w_j)} \right\}.
\end{align*}

Since $d$ is induced by a norm, $(\Rb^{k},d)$ is quasi-isometric to real Euclidean $k$-space. Further,  
\begin{align*}
d & \Big( \Phi \big( [x_1:\ldots:x_{k+1}] \big),\Phi \big( [y_1:\ldots:y_{k+1}] \big) \Big) = \max_{1\leq i,j \leq k+1} \frac{1}{2} \abs{\log \frac{x_i y_j}{y_i x_j}}
\end{align*}
and so Example ~\ref{ex:basic_properties_of_simplices} implies that  $\Phi$ induces an isometry $(S,H_S) \rightarrow (\Rb^{k}, d)$. Hence, $(S,H_S)$ is quasi-isometric to real Euclidean $k$-space. 
\end{proof}

 We will frequently use the following observation about the faces of properly embedded simplices. 

\begin{observation}\label{obs:faces_of_simplices_are_properly_embedded}
Suppose $\Omega \subset \Pb(\Rb^d)$ is a properly convex domain and $S \subset \Omega$ is a properly embedded simplex. If $x \in \partial S$, then 
\begin{enumerate}
\item $F_S(x)$ is properly embedded in $F_\Omega(x)$. 
\item $F_S(x) = \overline{S} \cap F_\Omega(x)$. 
\end{enumerate}
\end{observation}

\begin{proof} By definition $F_S(x) \subset F_\Omega(x)$. So to show that $F_S(x)$ is properly embedded in $F_\Omega(x)$ it is enough to show that 
\begin{align*}
\partial F_S(x) \subset \partial F_\Omega(x).
\end{align*}
Suppose not. Then, since $\overline{F_S(x)} \subset \overline{F_\Omega(x)}$, there exists some $y \in F_\Omega(x) \cap  \partial F_S(x)$. Let $v_1,\dots, v_p$ denote the vertices of $S$. Then after relabelling there exist $0 < \ell < k < p$ such that:
\begin{align*}
\overline{F_S(x)} = \partial S \cap \Pb(\Spanset\{v_1,\dots, v_k\})
\end{align*}
and 
\begin{align*}
\overline{F_S(y)} = \partial S \cap \Pb(\Spanset\{v_1,\dots, v_\ell\}).
\end{align*}
Pick $z \in \partial S$ such that 
\begin{align*}
\overline{F_S(z)} = \partial S \cap \Pb(\Spanset\{v_{k+1},\dots, v_p\}).
\end{align*}
Then $(z,x) \subset \Omega$, but 
\begin{align*}
[z,y] \subset \overline{S} \cap \Pb(\Spanset\{v_1,\dots, v_{k-1}, v_{k+1},\dots, v_p\}) \subset \partial S \subset \partial \Omega.
\end{align*}
But, since $y \in F_\Omega(x)$, this contradicts Observation~\ref{obs:faces} part (4). Hence $\partial F_S(x) \subset \partial F_\Omega(x)$ and so $F_S(x)$ is properly embedded in $F_\Omega(x)$. 

Next we show that $F_S(x) = \overline{S} \cap F_\Omega(x)$. By definition $F_S(x) \subset \overline{S} \cap F_\Omega(x)$. To establish the other inclusion, fix $u \in \overline{S} \cap F_\Omega(x)$ and then fix $w \in (x,u)$. Then, since $F_\Omega(x) = F_\Omega(w)$, we have
\begin{align*}
x,u \in  F_\Omega(w) \cap \overline{F_S(w)}.
\end{align*}
But by part (1)
\begin{align*}
\emptyset = F_\Omega(w) \cap \partial F_S(w).
\end{align*}
So $x,u \in F_S(w)$. So $F_S(x)=F_S(w) = F_S(u)$ and, in particular, $u \in F_S(x)$. Since $u \in \overline{S} \cap F_\Omega(x)$ was arbitrary, we see that $F_S(x) \supset \overline{S} \cap F_\Omega(x)$. 
\end{proof}

\section{Linear projections onto simplices}\label{sec:linear_projections}

In this section we construct certain linear maps associated to a properly embedded simplex  in a properly convex domain. 

\begin{definition} Suppose that $\Omega \subset \Pb(\Rb^d)$ is a properly convex domain and $S \subset \Omega$ is a properly embedded simplex with $\dim S = (q-1) \geq 1$. A set of co-dimension one linear subspaces $\Hc:=\{H_1, \dots, H_{q}\}$ is \emph{$S$-supporting} when:
\begin{enumerate}
\item Each $\Pb(H_j)$ is a supporting hyperplane of $\Omega$,
\item If $F_1, \dots, F_{q} \subset \partial S$ are the boundary faces of maximal dimension, then (up to relabelling) $F_j \subset \Pb(H_j)$ for all $1 \leq j \leq q$.
\end{enumerate}
\end{definition}

\begin{proposition} \label{prop:W-direct-sum}
Suppose that $\Omega \subset \Pb(\Rb^d)$ is a properly convex domain, $S \subset \Omega$ is a properly embedded simplex, and $\Hc$ is a set of  $S$-supporting hyperplanes. Then 
\begin{align*}
\Spanset S \oplus \left(\cap_{H \in \Hc} H\right)= \Rb^d ~~~ \text{ and } ~~~ \Omega \cap  \Pb \left( \cap_{H \in \Hc} H \right)= \emptyset.
\end{align*}
\end{proposition}

\begin{proof} Suppose $\Hc:=\{H_1, \dots, H_{q}\}$, $F_1, \dots, F_{q} \subset \partial S$  are the boundary faces of maximal dimension, and $v_1,\dots, v_q$ are the vertices of $S$ labelled so that $F_j \subset \Pb(H_j)$ and $v_j \notin \overline{F_j}$.  Let $\overline{v}_1,\dots,\overline{v}_q \in \Rb^d \setminus \{0\}$ be lifts of $v_1,\ldots, v_q$ respectively.

First notice that 
\begin{align*}
\Omega \cap   \Pb \left( \cap_{H \in \Hc} H \right) = \emptyset
\end{align*}
 since $\Pb(H_j) \cap \Omega =\emptyset$ for every $j$.

Since $S \subset \Pb(v_j+H_j)$ and $S \cap \Pb(H_j) = \emptyset$, we must have $v_j \notin \Pb(H_j)$ and hence 
\begin{align}
\label{eq:direct_sum_vertex_opp_hyp}
v_j \oplus H_j=\Rb^d
\end{align} 
for every $j$. Further, 
\begin{align}
\label{eq:other_vertices_in_face}
v_1,\dots, v_{j-1}, v_{j+1}, \dots, v_q \in \overline{F_j} \subset \Pb(H_j)
\end{align}
for each $j$.

Define $W : =\cap_{H \in \Hc} H$. We claim that 
\begin{align*}
\Spanset S \oplus W= \Rb^d.
\end{align*}
Since 
\begin{align*}
\dim W +\dim \Spanset S \geq (d - q)+q = d,
\end{align*}
it suffices to show that 
\begin{align*}
\Spanset S  \cap W= \{0\}. 
\end{align*}
If not, we can find  $\alpha_1,\dots, \alpha_q \in \Rb$ such that
\begin{align*}
0 \neq \sum_{j=1}^q \alpha_j \overline{v}_j \in W.
\end{align*}
By relabelling we can assume that $\alpha_1 \neq 0$. Then by Equation~\eqref{eq:other_vertices_in_face}
\begin{align*}
v_1 \subset \Spanset\{ v_2,\dots, v_q\} + W \subset H_1
\end{align*}
which contradicts Equation~\eqref{eq:direct_sum_vertex_opp_hyp}. So
\begin{equation*}
\Spanset S \oplus W= \Rb^d. \qedhere
\end{equation*}

\end{proof}

Using Proposition \ref{prop:W-direct-sum}, we define the following linear projection.

\begin{definition}
Suppose $\Omega \subset \Pb(\Rb^d)$ is a properly convex domain, $S \subset \Omega$ is a properly embedded simplex, and $\Hc$ is a set of $S$-supporting hyperplanes. Define $L_{S,\Hc} \in \End(\Rb^d)$ to be the linear projection 
\begin{align*}
\Spanset S \oplus \left(\cap_{H \in \Hc} H\right) \longrightarrow \Spanset S
\end{align*}
We call $L_{S,\Hc}$ the \emph{linear projection of $\Omega$ onto $S$ relative to $\Hc$}. 
\end{definition}

Calling $L_{S,\Hc}$ the linear projection of $\Omega$ onto $S$ is motivated by the following observation. 

\begin{observation} \label{obs:lin-proj}
$L_{S,\Hc}(\Omega)=S.$
\end{observation}
\begin{proof}
By Proposition \ref{prop:W-direct-sum}, $\Pb( \ker L_{S,\Hc} ) \cap \Omega = \emptyset$, so $L_{S,\Hc}$ is well-defined on $\Omega$. The set $L_{S,\Hc}(\Omega) \subset \Pb(\Spanset S)$ is  connected and contains $S=L_{S,\Hc}(S)$. Further
\begin{align*}
L^{-1}_{S,\Hc}(\partial S) =\cup_{j=1}^q L^{-1}_{S,\Hc}\Big(\overline{F_j}\Big) \subset \cup_{j=1}^q  \Pb(H_j)
\end{align*}
and so $\Omega \cap L^{-1}_{S,\Hc}(\partial S) =\emptyset$. Thus $L_{S,\Hc}(\Omega) = S$. 
\end{proof}

We now derive some basic properties of these projection maps. First recall, from Definition~\ref{defn:faces_of_subsets}, that 
\begin{align*}
F_\Omega(X) = \cup_{x \in X} F_\Omega(x)
\end{align*}
when $\Omega$ is a properly convex domain and $X \subset \overline{\Omega}$. 

\begin{proposition} \label{prop:W-intersect-boundary}
Suppose $\Omega \subset \Pb(\Rb^d)$ is a properly convex domain, $S \subset \Omega$ is a properly embedded simplex, and $\Hc$ is a set of  $S$-supporting hyperplanes. Then 
\begin{enumerate}
\item If $x \in \partial \Omega \cap \Pb \left( \cap_{H \in \Hc} H \right)$ and $y \in \partial S$, then $[x,y] \subset \partial \Omega$. 
\item $\Pb\left(\cap_{H \in \Hc} H \right) \cap F_\Omega(\partial S) = \emptyset$.
\end{enumerate}
\end{proposition}

\begin{proof} Suppose $x \in \partial \Omega \cap \Pb(\cap_{H \in \Hc} H)$ and $y \in \partial S$. Then there exists a boundary face $F \subset \partial S$ of maximal dimension such that $y \in \overline{F}$. Then there exists some $H \in \Hc$ such that $F \subset \Pb(H)$. Then $[x,y] \subset \Pb(H)$ and so $[x,y] \cap \Omega = \emptyset$. Thus $[x,y] \subset \partial \Omega$. 

Next, suppose for a contradiction that 
\begin{align*}
x \in \Pb(\cap_{H \in \Hc} H) \cap F_\Omega(\partial S).
\end{align*}
Then there exists $y \in \partial S$ with $x \in F_\Omega(y)$. Pick $y^\prime \in \partial S$ such that $(y,y^\prime) \subset S$. Then by Observation~\ref{obs:faces} part (4) we also have $(x,y^\prime) \subset \Omega$. But this contradicts part (1). 
\end{proof}

\begin{proposition}\label{prop:image-lin-proj-in-simplex-bdry} 
Suppose $\Omega \subset \Pb(\Rb^d)$ is a properly convex domain, $S \subset \Omega$ is a properly embedded simplex, and $\Hc$ is a set of  $S$-supporting hyperplanes. If $x \in F_\Omega(\partial S)$, then $L_{S, \Hc}(x) \in F_\Omega(x) \cap \overline{S}$. 
\end{proposition}

\begin{remark} Notice that $L_{S,\Hc}$ is defined and continuous on 
\begin{align*}
\Pb(\Rb^d) \setminus \Pb(\ker L_{S,\Hc}) = \Pb(\Rb^d) \setminus \Pb( \cap_{H \in \Hc} H)
\end{align*}
 and so the previous Proposition implies that $L_{S, \Hc}(x)$ is well defined. \end{remark}

\begin{proof} Fix $y \in \partial S$ such that $x \in F_\Omega(y)$. Then there exists an open line segment $\ell \subset F_\Omega(y)$ with $x,y \in \ell$. Then $L_{S,\Hc}(y) =y$ and $L_{S,\Hc}(\ell)$ is either an open line segment or a single point. So either $L_{S,\Hc}(x)=y$ or $L_{S,\Hc}(x)$ and $y$ are contained in an open line segment in $\overline{S} \subset \overline{\Omega}$. So $L_{S, \Hc}(x) \in F_\Omega(y)=F_\Omega(x)$.
\end{proof}

For a general properly embedded simplex, there could be many different sets of supporting hyperplanes, but the next result shows that the corresponding linear projections form a compact set. 

\begin{definition} \label{defn:LS}
Suppose that $\Omega \subset \Pb(\Rb^d)$ is a properly convex domain and $S \subset \Omega$ is a properly embedded simplex. Define 
\begin{align*}
\Lc_S : = \{ L_{S,\Hc} : \text{  $\Hc$ is a set of  $S$-supporting hyperplanes} \} \subset \End(\Rb^d).
\end{align*}
\end{definition}

\begin{proposition} \label{prop:compactness-linear-proj}
Suppose that $\Omega \subset \Pb(\Rb^d)$ is a properly convex domain and $S \subset \Omega$ is a properly embedded simplex. Then $\Lc_S$ is a compact subset of $\End(\Rb^d)$. 
\end{proposition}

\begin{proof} Suppose that $F_1, \dots, F_{q} \subset \partial S$  are the boundary faces of $S$ of maximal dimension. Fix a sequence $L_{S,\Hc_n}$ of projections. Then 
\begin{align*}
\Hc_n=\{H_{n,1}, \dots, H_{n,q}\}
\end{align*}
where $F_j \subset \Pb( H_{n,j} )$. Since $\Gr_{d-1}(\Rb^d)$ is compact we can find $n_k \rightarrow \infty$  such that 
\begin{align*}
H_j : = \lim_{k \rightarrow \infty} H_{n_k,j}
\end{align*}
exists in $\Gr_{d-1}(\Rb^d)$ for every $1 \leq j \leq q$. Then $F_j \subset \Pb(H_j)$ and $\Pb(H_j) \cap \Omega = \emptyset$ for every $1 \leq j \leq q$. So $\Hc = \{ H_1, \dots, H_q\}$ is a set of $S$-supporting hyperplanes. Further, by definition, 
\begin{align*}
L_{S,\Hc} = \lim_{k \rightarrow \infty} L_{S,\Hc_{n_k}}
\end{align*}
in  $\End(\Rb^d)$. Since $L_{S, \Hc_n}$ was an arbitrary sequence, $\Lc_S$ is compact. 
\end{proof}

\section{Opposite faces of periodic properly embedded simplices}\label{sec:opposite_faces}

In this section we establish a structure theorem for periodic properly embedded simplices. 

\begin{definition} Given a properly convex domain $\Omega \subset \Pb(\Rb^d)$ and a properly embedded simplex $S \subset \Omega$, we say that $S$ is \emph{periodic} if $\Stab_{\Aut(\Omega)}(S)$ acts co-compactly on $S$.
\end{definition}

\begin{definition} Suppose $S \subset \Pb(\Rb^d)$ is a simplex. Two faces $F_1, F_2 \subset \partial S$ are \emph{opposite} when
\begin{enumerate}
\item $\overline{F}_1 \cap \overline{F}_2 = \emptyset$ and
\item $\Spanset S = \Spanset \big(  F_1 \cup F_2 \big) $ .
\end{enumerate}
Two points $x_1,x_2 \in \partial S$ are \emph{opposite} if their faces $F_S(x_1)$ and $F_S(x_2)$ are opposite. 
\end{definition}

\begin{observation} If $S \subset \Pb(\Rb^d)$ is a simplex, then two faces $F_1, F_2 \subset \partial S$ are opposite if and only if there exists a labeling $v_1,\dots, v_p$ of the vertices of $S$ such that
\begin{align*}
F_1 =\relint \left( \CH_S \{v_1,\dots, v_q\}\right)
\end{align*}
and 
\begin{align*}
F_2 =\relint \left( \CH_S\{v_{q+1},\dots, v_p\} \right)
\end{align*}
for some $1 \leq q \leq p-1$.
\end{observation}

\begin{observation} If $S \subset \Pb(\Rb^d)$ is a simplex and $x_1,x_2 \in \partial S$ are opposite, then $(x_1,x_2) \subset S$. 
\end{observation}

\begin{theorem}\label{thm:CH_of_opp_faces} Suppose $\Omega \subset \Pb(\Rb^d)$ is a properly convex domain, $S$ is a properly embedded simplex in $\Omega$, and $\Stab_{\Aut(\Omega)}(S)$ acts co-compactly on $S$. Assume $x_1, x_2 \in \partial S$ are opposite points. If
\begin{align*}
V_1 = \Spanset F_\Omega(x_1), \ V_2 = \Spanset F_\Omega(x_2), \text{ and } V= V_1 + V_2,
\end{align*}
then $V_1 \cap V_2 = \{0\}$ and 
\begin{align*}
\Omega \cap \Pb(V) = \relint \left( \CH_\Omega \left(  F_\Omega(x_1) \cup  F_\Omega(x_2)\right)\right).
\end{align*}
\end{theorem}

\begin{remark} \ \begin{enumerate}
\item For the last equality, notice that ``$\supset$'' is a consequence of convexity. 
\item Since we are assuming that opposite points exist, we must have $\dim S \geq 1$.
\end{enumerate}
\end{remark}

Before proving the theorem we state and prove two corollaries. 

\begin{corollary}\label{cor:lines_between_opp_faces} With the hypothesis of Theorem~\ref{thm:CH_of_opp_faces}:
\begin{enumerate}
\item if $y_1 \in F_\Omega(x_1)$ and $y_2 \in F_\Omega(x_2)$, then $(y_1, y_2) \subset  \Omega$
\item if $y_1 \in \partial F_\Omega(x_1)$ and $y_2 \in \overline{F_\Omega(x_2)}$, then $[y_1, y_2] \subset \partial \Omega$,
\item if $y_1 \in \overline{F_\Omega(x_1)}$ and $y_2 \in \partial F_\Omega(x_2)$, then $[y_1, y_2] \subset \partial \Omega$
\end{enumerate}
\end{corollary}

\begin{proof} This follows immediately from Theorem~\ref{thm:CH_of_opp_faces} and the fact that $\partial(\Omega \cap \Pb(V)) \subset \partial \Omega$. \end{proof}

\begin{corollary}\label{cor:building_simplices} With the hypothesis of Theorem~\ref{thm:CH_of_opp_faces}: if $S_1 \subset F_\Omega(x_1)$ and $S_2 \subset F_\Omega(x_2)$ are properly embedded simplices, then 
\begin{align*}
S^\prime: = \relint \left(  \CH_\Omega \left( S_1 \cup  S_2 \right)\right)
\end{align*}
is a properly embedded simplex of $\Omega$ with 
\begin{align*}
\dim S^\prime = \dim S_1 + \dim S_2+1.
\end{align*}
\end{corollary} 

\begin{proof} Let $v_1, \dots, v_p$ be the extreme points of $S_1$ and let $w_1, \dots,w_q$ be the extreme points of $S_2$. Since $V_1 \cap V_2 = \{0\}$, the lines $v_1,\dots, v_p, w_1,\dots, w_q$ are linearly independent in $\Rb^d$. Then Corollary~\ref{cor:lines_between_opp_faces} implies that $S^\prime$ is a properly embedded simplex with vertices $v_1,\dots, v_p, w_1,\dots, w_q$ and so
\begin{equation*}
\dim S^\prime = p+q-1 = (p-1)+(q-1)+1= \dim S_1+\dim S_2 +1. \qedhere
\end{equation*}
\end{proof}

The rest of the section is devoted to the proof of Theorem~\ref{thm:CH_of_opp_faces}. So fix $S \subset \Omega \subset \Pb(\Rb^d)$ and $x_1,x_2 \in \partial S$ as in the statement of the theorem.

First, $\Stab_{\Aut(\Omega)}(S)$ permutes the faces of $S$ and $S$ has finitely many faces, so there exists a finite index subgroup $G \leq \Stab_{\Aut(\Omega)}(S)$ which stabilizes each face of $S$. Then  $G(V_1) = V_1$, $G(V_2) = V_2$, $G|_{V_1} \leq \Aut(F_\Omega(x_1))$, and $G|_{V_2} \leq \Aut(F_\Omega(x_2))$. Further, $G$ acts co-compactly on $S$ since $G$ has finite index in $\Stab_{\Aut(\Omega)}(S)$.

\begin{lemma} Fix $p_0 \in (x_1, x_2)$. There exist  $y_1 \in F_S(x_1)$, $y_2 \in F_S(x_2)$, and a sequence $a_n \in G$ such that: 
\begin{enumerate}
\item $y_1 = \lim_{n \rightarrow \infty} a_n p_0$, 
\item $y_1= \lim_{n \rightarrow \infty} a_n x_1$, and 
\item $y_2 = \lim_{n \rightarrow \infty} a_n x_2$. 
\end{enumerate}
\end{lemma}

\begin{proof}
Let $v_1,\dots, v_p$ be the extreme points of $S$ labelled so that 
\begin{align*}
F_S(x_1) =\relint \left( \CH_\Omega \{v_1,\dots, v_q\}\right)
\end{align*}
and 
\begin{align*}
F_S(x_2)=\relint \left( \CH_\Omega\{v_{q+1},\dots, v_p\} \right)
\end{align*}
where $q =1+ \dim F_S(x_1)$. 

Let $W: = \Spanset S$. Since $G$ fixes each extreme point and acts co-compactly on $S$, we can find a sequence $a_n \in G$ with 
\begin{align*}
a_n|_W = \begin{bmatrix} 
\lambda_{n,1} & & \\ & \ddots & \\ & & \lambda_{n,p}
\end{bmatrix}
\end{align*}
relative to the basis $v_1,\dots, v_p$ of $W$ and
\begin{align*}
\lim_{n \rightarrow \infty} \frac{\lambda_{n,i}}{\lambda_{n,j}} =  \left\{ 
\begin{array}{ll} 
c_{i,j} \in (0,\infty) & \text{ if } 1 \leq i,j \leq q \\
\infty & \text{ if } 1 \leq i\leq q < j \leq p \\
c_{i,j} \in (0,\infty) & \text{ if } q < i,j \leq p.
\end{array}
\right.
\end{align*}
Then define
\begin{align*}
y_1 : = [ {\rm diag}(c_{1,1}, \dots, c_{q,1}, 0,\dots,0)]\cdot x_1 = \lim_{n \rightarrow \infty} a_n x_1 \in F_S(x_1)
\end{align*}
and 
\begin{align*}
y_2 : = [ {\rm diag}(0, \dots, 0, c_{q+1,q+1},\dots,c_{p,q+1})]\cdot x_2 = \lim_{n \rightarrow \infty} a_n x_2 \in F_S(x_2).
\end{align*}
Then $\lim_{n \to \infty} a_n p_0 = y_1$.
\end{proof}

Next let $a_{1,n} = a_n |_{V_1} \in \Aut(F_\Omega(x_1))$ and $a_{2,n} = a_n |_{V_2} \in \Aut(F_\Omega(x_2))$. Since $\lim_{n \to \infty}a_{1,n}(x_1) = y_1 \in F_S(x_1)$ and $\Aut(F_\Omega(x_1))$ acts properly on $F_\Omega(x_1)$, we see that 
\begin{align*}
\{ a_{1,n} : n \geq 0\} \subset \Aut(F_\Omega(x_1))
\end{align*}
is relatively compact. The same argument implies that 
\begin{align*}
\{ a_{2,n} : n \geq 0\} \subset\Aut(F_\Omega(x_2))
\end{align*}
is relatively compact. So by passing to a subsequence we can suppose that $a_1 : = \lim_{n \to \infty} a_{1,n}$ exists in  $\Aut(F_\Omega(x_1))$ and $a_2:= \lim_{n \to \infty} a_{2,n}$ exists  in $\Aut(F_\Omega(x_1))$. 

\begin{lemma} $V_1 \cap V_2 = \{0\}$. \end{lemma}

\begin{proof} Suppose not. Then fix a decomposition 
\begin{align*}
V = (V_1 \cap V_2) \oplus W_1 \oplus W_2
\end{align*}
 with $V_1 = (V_1 \cap V_2) \oplus W_1$ and $V_2 = (V_1 \cap V_2) \oplus W_2$. Then relative to this decomposition 
\begin{align*}
a_n|_V = \begin{bmatrix} A_{1,1}^{(n)} & A_{1,2}^{(n)}  & A_{1,3}^{(n)}  \\ 0 & A_{2,2}^{(n)}  & 0 \\ 0 & 0 & A_{3,3}^{(n)} \end{bmatrix}
\end{align*}
for some $A_{1,1}^{(n)}  \in \GL(V_1 \cap V_2)$, $A_{1,2}^{(n)} \in {\rm Lin}(W_1, V_1 \cap V_2)$, $A_{1,3}^{(n)} \in { \rm Lin}(W_2, V_1 \cap V_2)$, $A_{2,2}^{(n)} \in \GL(W_1)$, and $A_{3,3}^{(n)} \in \GL(W_2)$. 

Then 
\begin{align*}
a_{1,n} =  \begin{bmatrix} A_{1,1}^{(n)} & A_{1,2}^{(n)}   \\ 0 & A_{2,2}^{(n)}  \end{bmatrix}
\end{align*}
and 
\begin{align*}
a_{2,n} =  \begin{bmatrix} A_{1,1}^{(n)} & A_{1,3}^{(n)}   \\ 0 & A_{3,3}^{(n)}  \end{bmatrix}
\end{align*}
relative to the decompositions $V_1 = (V_1 \cap V_2) \oplus W_1$ and $V_2 = (V_1 \cap V_2) \oplus W_2$ respectively. 

Since $a_{1,n}$ converges to $a_1$ in $\PGL(V_1)$ there exists a sequence $t_{1,n} \in \Rb$ such that 
\begin{align*}
A_1: =\lim_{n \rightarrow \infty} t_{1,n}  \begin{pmatrix} A_{1,1}^{(n)} & A_{1,2}^{(n)}   \\ 0 & A_{2,2}^{(n)}  \end{pmatrix} \in \GL(V_1)
\end{align*}
and $a_1=[A_1]$. For the same reasons, there exists a sequence $t_{2,n} \in \Rb$ such that 
\begin{align*}
A_2: =\lim_{n \rightarrow \infty} t_{2,n}  \begin{pmatrix} A_{1,1}^{(n)} & A_{1,3}^{(n)}   \\ 0 & A_{3,3}^{(n)}  \end{pmatrix}  \in \GL(V_2)
\end{align*}
and $a_2=[A_2]$.

Since $t_{1,n} A_{1,1}^{(n)}$ and $t_{2,n} A_{1,1}^{(n)}$ both converge in $\GL(V_1 \cap V_2)$ we must have 
\begin{align*}
0 \neq c:=\lim_{n \rightarrow \infty} \frac{t_{2,n}}{t_{1,n}}.
\end{align*}
But then 
\begin{align*}
t_{1,n}  \begin{pmatrix} A_{1,1}^{(n)} & A_{1,2}^{(n)}  & A_{1,3}^{(n)}  \\ 0 & A_{2,2}^{(n)}  & 0 \\ 0 & 0 & A_{3,3}^{(n)} \end{pmatrix}
\end{align*}
converges in $\GL(V)$ which implies that $a_n|_V$ converges in $\PGL(V)$. But, by construction, $a_n|_V$ diverges in $\PGL_d(\Rb)$. So we have a contradiction and hence $V_1 \cap V_2 = \{0\}$. \end{proof}

Next let $\pi_1 : V \rightarrow V_1$ and $\pi_2 : V \rightarrow V_2$ be the projections relative to the decomposition $V=V_1 \oplus V_2$. To show that 
\begin{align*}
\Omega \cap \Pb(V) = \relint \left( \CH_\Omega \left(  F_\Omega(x_1) \cup  F_\Omega(x_2)\right)\right).
\end{align*}
it is enough to show that $\pi_1(\Omega \cap \Pb(V)) = F_\Omega(x_1)$ and $\pi_2(\Omega \cap \Pb(V)) = F_\Omega(x_2)$.

\begin{lemma} $\pi_1(\Omega \cap \Pb(V)) = F_\Omega(x_1)$. \end{lemma}

\begin{proof} Since $F_\Omega(x_1) \subset \overline{\Omega \cap \Pb(V)}$, we clearly have 
\begin{align*}
F_\Omega(x_1) \subset \pi_1(\Omega \cap \Pb(V)).
\end{align*}

Relative to the decomposition $V=V_1\oplus V_2$ we have 
\begin{align*}
a_n|_V = \begin{bmatrix} B_{1,n} & 0 \\ 0 & B_{2,n} \end{bmatrix}
\end{align*}
where $B_{1,n} \in \GL(V_1)$, $[B_{1,n}]=a_{1,n}$, $B_{2,n} \in \GL(V_2)$, and $[B_{2,n}]=a_{2,n}$. Since $a_{1,n}$ converges to $a_1$ in $\PGL(V_1)$ there exists a sequence $s_{1,n} \in \Rb$ such that 
\begin{align*}
B_1:=\lim_{n \rightarrow \infty} s_{1,n} B_{1,n}  \in \GL(V_1)
\end{align*}
and $a_1=[B_1]$. Similarly, there exists a sequence $s_{s,n} \in \Rb$ such that 
\begin{align*}
B_2:=\lim_{n \rightarrow \infty} s_{2,n} B_{2,n}  \in \GL(V_2)
\end{align*}
and $a_2 = [B_2]$. 

Since $\lim_{n \to \infty} a_n p_0 = y_1 \in F_\Omega(x_1)$, we must have 
\begin{align*}
\lim_{n \rightarrow \infty} \frac{s_{1,n}}{s_{2,n}} = 0. 
\end{align*}
Then $a_n|_V$ converges to  $[T] \in \Pb(\End(V))$ where  
\begin{align*}
T = \begin{pmatrix} B_1 & 0 \\ 0 & 0 \end{pmatrix} \in \End(V).
\end{align*} 
Notice that $a_1^{-1} \circ [T] = [\pi_1]$.

Fix $p \in \Omega \cap \Pb(V)$. We claim that $T(p) \in F_\Omega(x_1)$. Since $\ker T = V_2$ and $\Pb(V_2) \cap \Omega = \emptyset$, we have $p \notin \Pb(\ker T)$. So 
\begin{align*}
T(p) = \lim_{n \rightarrow \infty} a_n(p).
\end{align*}
Then, since $\lim_{n \to \infty} a_n(p_0)= y_1 \in F_\Omega(x_1)$, Proposition~\ref{prop:dist_est_and_faces} implies that $T(p) \in F_\Omega(x_1)$. Since $p \in \Omega \cap \Pb(V)$ was arbitrary, 
\begin{align*}
T(\Omega \cap \Pb(V)) \subset F_\Omega(x_1).
\end{align*}
Then
\begin{align*}
\pi_1(\Omega \cap \Pb(V)) = (a_1^{-1} \circ [T])(\Omega \cap \Pb(V)) \subset a_1^{-1} \big( F_\Omega(x_1) \big)=F_\Omega(x_1), 
\end{align*}
since $a_1 \in \Aut(F_\Omega(x_1))$. 
\end{proof}

Applying the same argument to $a_n^{-1}$ shows that
\begin{align*}
\pi_2(\Omega \cap \Pb(V)) = F_\Omega(x_2)
\end{align*}
which completes the proof.

\part{The naive convex co-compact case}

\section{Invariant and isolated sets of simplices are periodic}\label{sec:inv_plus_iso_implies_periodic}

In this section we show that any isolated and invariant family of properly embedded simplices of dimension at least two satisfies properties (1) and (2) of Theorem~\ref{thm:properties_of_ncc}.

\begin{proposition}\label{prop:HK_section_3_1}
Suppose $(\Omega, \Cc, \Lambda)$ is a naive convex co-compact triple and $\Sc$ is an isolated and $\Lambda$-invariant family of maximal properly embedded simplices in $\Cc$ of dimension at least two. Then:
\begin{enumerate}
\item $\Sc$ is a locally finite collection, that is for any compact set $K \subset \Omega$ the set
\begin{align*}
\{ S \in \Sc : S \cap K \neq \emptyset\}
\end{align*}
is finite. 
\item $\LG$ has finitely many orbits in $\Sc$.
\item If $S \in \Sc$, then $\Stab_{\Lambda}(S)$ acts co-compactly on $S$ and contains a finite index subgroup isomorphic to $\Zb^k$ where $k = \dim S$. 
\end{enumerate}
\end{proposition}

Parts (1) and (2) are simple consequences of the definition. The proof of the first assertion in part (3) is nearly identical to the proof of the analogous result in the $\CAT(0)$ setting, see Wise~\cite[Proposition 4.0.4]{W1996}, Hruksa~\cite[Theorem 3.7]{H2005}, or Hruska-Kleiner~\cite[Section 3.1]{HK2005}. 

\begin{proof}
Since $\Sc$ is closed and discrete in the local Hausdorff topology part (1) is true. To prove part (2), fix a compact set $K \subset \Cc$ such that $\Lambda \cdot K = \Cc$. Then each $\LG$-orbit of $\Sc$ intersects $K$ and by part (1) there are only finitely many such intersections. Hence, there are only finitely many $\LG$-orbits in $\Sc$.

To prove part (3), fix $S \in \Sc$. Let
\begin{align*}
X := \{ g \in \Lambda : S \cap gK \neq \emptyset\}.
\end{align*}
Then $S= \cup_{g \in X} (S \cap gK)$. Since $(g^{-1}S) \cap K \neq \emptyset$ when $g \in X$, Part (1) implies that the set
\begin{align*}
\{ g^{-1}S : g \in X\}
\end{align*}
is finite. Since $g^{-1}S = h^{-1}S$ if and only if $gh^{-1} \in \Stab_{\Lambda}(S)$ if and only if $\Stab_{\Lambda}(S)g=\Stab_{\Lambda}(S)h$, there exist $g_1,\dots, g_m \in X$ such that 
\begin{align*}
\bigcup_{g\in X} \Stab_{\Lambda}(S)g = \bigcup_{j=1}^m \Stab_{\Lambda}(S)g_j .
\end{align*}
Then the set $\wh{K} : = \cup_{j=1}^m S \cap g_j K$ is compact and 
\begin{align*}
\Stab_{\Lambda}(S) \cdot \wh{K} = \cup_{g \in X}S \cap g K = S.
\end{align*}
So $ \Stab_{\Lambda}(S)$ acts co-compactly on $S$. 

Finally, we show that $\Stab_{\Lambda}(S)$ contains a finite index subgroup isomorphic to $\Zb^k$ where $k = \dim S$. First, let $G \leq \Stab_{\Lambda}(S)$ denote the subgroup of elements which fix every vertex of $S$. Then $G$ has finite index in $\Stab_{\Lambda}(S)$. Next let $V := \Spanset S$ and consider the homomorphism 
\begin{align*}
\varphi &:  G \rightarrow \Aut(S) \leq \PGL(V) \\
\varphi &(g) = g|_V.
\end{align*}
Fix $v_1,\dots, v_{k+1} \in V$ such that $[v_1],\dots,[v_{k+1}]$ are the vertices of $S$. Then relative to the basis $v_1,\dots, v_{k+1}$
\begin{align*}
\varphi(G) \leq
\mathcal{D}:=\left\{ \left[{\rm diag}(a_1,\dots,a_{k+1})\right] : a_1,\dots,a_{k+1} > 0\right\} \cong (\Rb^k,+).
\end{align*}
Further, $\varphi(G)$ is a lattice of $\mathcal{D}$ since $G$ acts properly and co-compactly on $S$. So $\varphi(G)$ is isomorphic to $\Zb^k$. 

Every element of $\ker \varphi$ acts trivially on $S$. Then, since $\Aut(\Omega)$ acts properly on $\Omega$, the group $\ker \varphi \leq G$ is finite. By Selberg's Lemma, there exists a torsion-free finite index subgroup $\Lambda_0 \leq \Lambda$. Then 
\begin{align*}
H := \Lambda_0 \cap G \leq \Stab_{\Lambda}(S)
\end{align*}
has finite index. Further, $\ker \varphi \cap H = \{1\}$ since $H$ is torsion-free. So $H$ is isomorphic to $\varphi(H)$. Finally, $\varphi(H) \leq \varphi(G) \cong \Zb^k$ has finite index so $\varphi(H)$ is also isomorphic to $\Zb^k$. 
\end{proof}

\section{Faces of properly embedded simplices}\label{sec:bd_faces}

In this section we show that any isolated, coarsely complete, and invariant family of properly embedded simplices can be refined to a family which also satisfies property (4) of Theorem~\ref{thm:properties_of_ncc}.

\begin{theorem}\label{thm:bd_faces}  Suppose $(\Omega, \Cc, \Lambda)$ is a naive convex co-compact triple with coarsely isolated simplices. Let $\Sc_0$ be an isolated, coarsely complete, and $\Lambda$-invariant family of maximal properly embedded simplices in $\Cc$ of dimension at least two. Then there exists an isolated, coarsely complete, and $\Lambda$-invariant subfamily $\Sc \subset \Sc_0$ with the following additional property: 

\begin{itemize}
\item[$(\star)$]  There exists $D_1 > 0$ such that: if $S \in \Sc$ and $x \in \partial S$, then 
\begin{align*}
H_{F_\Omega(x)}^{ \rm Haus}\Big( \overline{\Cc} \cap F_\Omega(x), F_S(x)\Big) \leq D_1. 
\end{align*}
\end{itemize}
\end{theorem}

\begin{remark} By Observation \ref{obs:faces_of_simplices_are_properly_embedded}: If $\Omega \subset \Pb(\Rb^d)$ is a properly convex domain, $S \subset \Omega$ is a properly embedded simplex, and $x \in \partial S$, then 
\begin{align*}
\overline{S} \cap F_\Omega(x)  =  \overline{F_{S}(x)} \cap F_\Omega(x)=F_S(x).
\end{align*}
However, if $\Cc \subset \Omega$ is a general convex subset and $x \in \partiali \Cc$, then
\begin{align*}
\overline{\Cc} \cap F_\Omega(x)  \supset \overline{F_{\Cc}(x)} \cap F_\Omega(x) \supset F_{\Cc}(x)
\end{align*}
and both inclusions can be strict. 
\end{remark}

The rest of the section is devoted to the proof of this theorem. We will need the following observation about simplices. 

\begin{lemma}\label{lem:line_segments_cover} 
Suppose $S \subset \Pb(\Rb^d)$ is a simplex and $(a,b) \subset S$ is a properly embedded line. If $p \in S$, then there exist $a^\prime \in F_S(a)$ and $b^\prime \in F_S(b)$ such that $p \in (a^\prime, b^\prime)$. \end{lemma}

\begin{proof} Let $X \subset S$ be the set of points $p \in S$ where there exist $a^\prime \in F_S(a)$ and $b^\prime \in F_S(b)$ such that $p \in (a^\prime, b^\prime)$. By hypothesis, $X$ is non-empty. Next let $G \leq \Aut(S)$ denote the group of automorphisms that fix all vertices. Then $G$ acts transitively on $S$ (see Example \ref{ex:basic_properties_of_simplices}). Further, $G \cdot F_S(a) = F_S(a)$ and $G \cdot F_S(b) = F_S(b)$. So $G \cdot X = X$. Then $X=S$ since $G$ acts transitively on $S$. 
\end{proof}

 We start by making an initial reduction.

\begin{lemma}\label{lem:refining_isolated_simplices} 
Suppose $(\Omega, \Cc, \Lambda)$ is a naive convex co-compact triple with coarsely isolated simplices. Let $\Sc_0$ be an isolated, coarsely complete, and $\Lambda$-invariant family of maximal properly embedded simplices in $\Cc$ of dimension at least two. Then there exists an isolated, coarsely complete, and $\Lambda$-invariant subfamily $\Sc \subset \Sc_0$  with the following additional property: 

\begin{itemize}
\item[$(\star)$] If $S_1, S_2 \in \Sc$ and $\dim S_1 < \dim S_2$, then 
\begin{align*}
\infty = \sup_{p \in S_1} H_\Omega \left( p, S_2 \right). 
\end{align*}
\end{itemize}
\end{lemma}

\begin{proof} Since $\Sc_0$ is coarsely complete there exists $D_0 > 0$ such that: If $S \subset \Cc$ is a properly embedded simplex of dimension at least two, then there exists some $S^\prime \in \Sc_0$ with
\begin{align*}
S \subset \Nc_\Omega(S^\prime; D_0).
\end{align*}

Let $X \subset \Sc_0$ be the set of simplices $S \in \Sc_0$ where there exists some $S^\prime \in \Sc_0$ with $\dim S < \dim S^\prime$ and
\begin{align*}
 \sup_{p \in S} H_\Omega \left( p, S^\prime \right) < +\infty.
\end{align*}
Notice that $X$ is a $\Lambda$-invariant subset of $\Sc_0$. Next, for each $S \in X$ define
\begin{align*}
m(S): = \inf_{\substack{S^\prime \in \Sc_0, \\
\dim S < \dim S^\prime}}  \sup_{p \in S} H_\Omega \left( p, S^\prime \right).
\end{align*}
Then $m(S)$ is finite and $\Lambda$-invariant. Further, Proposition~\ref{prop:HK_section_3_1} implies that there are only finitely many $\Lambda$-orbits in $X$. So 
\begin{align*}
M:=\sup_{S \in X} m(S) = \max_{S \in X} m(S) < +\infty.
\end{align*}

We claim that $\Sc:=\Sc_0\setminus X$ satisfies the conclusion of the proposition. By construction, if $S_1, S_2 \in \Sc$ and $\dim S_1 < \dim S_2$, then 
\begin{align*}
\infty = \sup_{p \in S_1} H_\Omega \left( p, S_2 \right). 
\end{align*}
Further, since $X$ is $\Lambda$-invariant, the set $\Sc$ is $\Lambda$-invariant. Also, since $\Sc\subset \Sc_0$, the set $\Sc$ is closed and discrete in the local Hausdorff topology. The following claim proves that $\Sc$ is coarsely complete.

\medskip
\noindent \textbf{Claim:} If $S \subset \Cc$ is a properly embedded simplex of dimension at least two, then there exists some $S^\prime \in \Sc$ with
\begin{align*}
S \subset \Nc_\Omega(S^\prime; D)
\end{align*}
where $D:=D_0+(d-3)M$. 
\medskip

Fix a properly embedded simplex $S\subset\Cc$ of dimension at least two. Then there exists some $S_1 \in \Sc_0$ such that 
\begin{align*}
S \subset \Nc_\Omega(S_1; D_0).
\end{align*}
If $S_1 \in \Sc$, then we are done. Otherwise there exist $k \in \{2,\dots, d-2\}$; $S_2, S_3, \dots, S_{k-1} \in X$; and $S_k \in \Sc$ such that:
\begin{enumerate}
\item $S_j \subset \Nc_\Omega(S_{j+1}; M)$ for $j=1,\dots,k-1$ and
\item $\dim S_j < \dim S_{j+1}$ for $j=1,\dots,k-1$.
\end{enumerate}
Then 
\begin{align*}
S \subset \Nc_\Omega(S_k; D_0+(k-1)M) \subset \Nc_\Omega(S_k; D).
\end{align*}
Since $S$ was an arbitrary properly embedded simplex of dimension at least two in $\Cc$ this completes the proof of the claim and the proposition. 
\end{proof}

For the rest of the section fix $\Omega$, $\Cc$, $\Lambda$, $\Sc_0$ satisfying the hypothesis of Theorem~\ref{thm:bd_faces}. Fix $\Sc \subset \Sc_0$ satisfying Lemma \ref{lem:refining_isolated_simplices}. Since $\Sc$ is coarsely complete there exists $D_0 > 0$ such that: if $S \subset \Cc$ is a properly embedded simplex of dimension at least two, then there exists some $S^\prime \in \Sc$ such that 
\begin{align*}
S \subset \Nc_\Omega(S^\prime; D_0).
\end{align*}

Since each simplex has finitely many faces and there are only finitely many distinct orbits of properly embedded simplices in $\Sc$ (see Proposition~\ref{prop:HK_section_3_1}), it is enough to show: if $S \in \Sc$ and $x \in \partial S$, then
\begin{align}
\label{eq:inclusion}
H_{F_\Omega(x)}^{ \rm Haus}\Big( \overline{\Cc} \cap F_\Omega(x), F_S(x)\Big) < +\infty.
\end{align}

Suppose for a contradiction that Equation~\eqref{eq:inclusion} fails for some choice of $S \in \Sc$ and $x \in \partial S$. We can choose $S$ and $x$ so that
\begin{align}
\label{eq:minimality_assumption}
\Big( \dim F_\Omega(x), \dim F_\Omega(x)-\dim F_S(x) \Big)
\end{align}
is the minimal in lexographical order among all examples which fail to satisfy Equation~\eqref{eq:inclusion}.

\subsection{The vertex case:} In this subsection we show that $x$ is not a vertex of $S$.

\begin{lemma}\label{lem:base_case}   $x$ is not a vertex of $S$. 
\end{lemma}

\begin{proof} Suppose for a contradiction that $x$ is a vertex of $S$. Then let $v_2,\dots, v_p$ be the other vertices of $S$. 

We will first show that $\overline{\Cc} \cap F_\Omega(x)$ contains a properly embedded line and then use this line to obtain a contradiction. 


\begin{figure}
\centering
\includegraphics[scale=0.3]{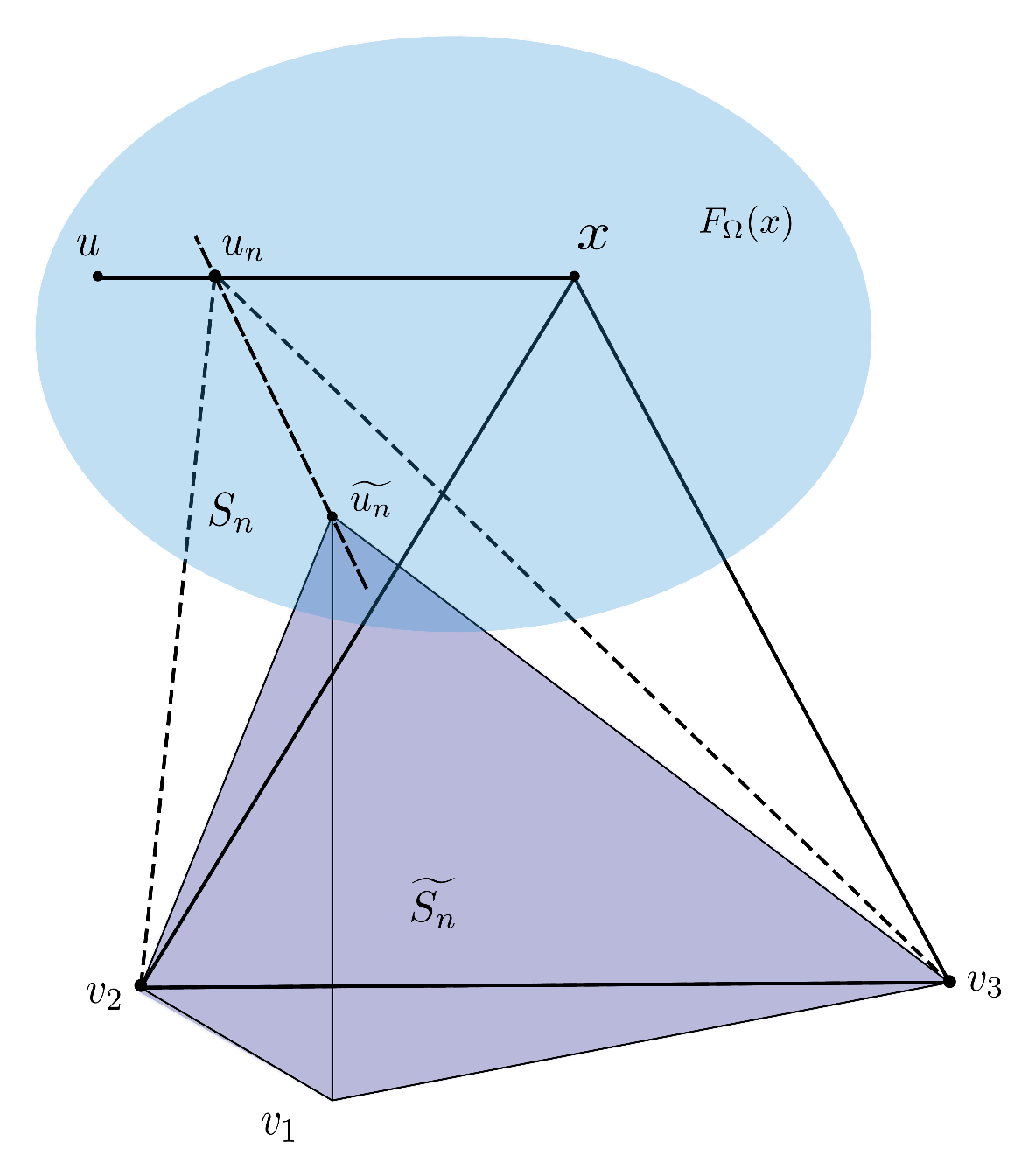}
\caption{Diagram illustrating the proof of Lemma \ref{lem:base_case}.} 
\label{fig:proof_of_vertex_case}
\end{figure}

\medskip
\noindent \textbf{Claim:} $\overline{\Cc} \cap F_\Omega(x)$ contains a properly embedded line.
\medskip

By assumption,
\begin{align*}
\infty = H_{F_\Omega(x)}^{ \rm Haus}\Big( \overline{\Cc} \cap F_\Omega(x), F_S(x)\Big) =H_{F_\Omega(x)}^{ \rm Haus}\Big( \overline{\Cc} \cap F_\Omega(x), \{x\}\Big).
\end{align*}
So there exists $u \in \overline{\Cc} \cap \partial F_\Omega(x)$. Fix a sequence 
\begin{align*}
u_n \in [x,u) \subset F_\Omega(x)
\end{align*}
such that $\lim_{n \to \infty} u_n =u$. Lemma~\ref{lem:slide_along_faces} implies that $u_n, v_2,\dots, v_p$ are the vertices of a properly embedded simplex $S_n \subset \Omega$ of dimension $\dim S_n=\dim S \geq 2$. Then for each $n$ there exists $\wt{S}_n \in \Sc$ such that 
\begin{align*}
S_n \subset \Nc_\Omega\left( \wt{S}_n; D_0\right).
\end{align*}
Then by Proposition~\ref{prop:dist_est_and_faces}, for each $n$ there exists $\wt{u}_n \in \partial \wt{S}_n$ such that 
\begin{align*}
H_{F_\Omega(x)}(\wt{u}_n, u_n) \leq D_0.
\end{align*}
If $\wt{u}_n$ is not a vertex of $\wt{S}_n$, then Observation  \ref{obs:faces_of_simplices_are_properly_embedded} implies that $F_{\wt{S}_n}(\wt{u}_n)$ is properly embedded in $F_{\Omega}(x)$ and   the claim is established (see Figure \ref{fig:proof_of_vertex_case}). So we may assume that $\wt{u}_n$ is a vertex of $\wt{S}_n$.

Next pick $g_n \in \Lambda$ and a compact set $K \subset \Omega$ such that  $g_n \wt{S}_n \cap K \neq \emptyset$ for all $n \geq 0$. Then by Proposition~\ref{prop:HK_section_3_1} the set
\begin{align*}
\left\{ g_n \wt{S}_n : n \geq 0\right\}
\end{align*}
is finite. So by passing to a subsequence we can suppose that 
\begin{align*}
\wt{S}:=g_n\wt{S}_n=g_m\wt{S}_m
\end{align*} 
for all $n,m \geq 0$.  Since $\wt{u}_n$ is a vertex of $\wt{S}_n$, $g_n \wt{u}_n$ is a vertex of $\wt{S}$ for all $n \geq 0$. Since $\wt{S}$ has finitely many vertices, by passing to a subsequence we can suppose that 
\begin{align*}
\wt{u}:=g_n\wt{u}_n=g_m\wt{u}_m
\end{align*} 
for all $n,m \geq 0$.  Then for all $n \geq 0$, let
\begin{align*}
F := F_\Omega(\wt{u}) = g_n F_{\Omega}(x)
\end{align*}

After passing to a subsequence, we can assume that $g_n(u,x)$ converges to $(u_{\infty}, x_{\infty}) \subset \overline{F}.$ We claim that $(u_{\infty},x_{\infty})$ is properly embedded in $F.$ By construction $g_n u \in \partial F$ for all $n$ so $u_\infty \in \partial F$. Since 
\begin{align*}
\lim_{n \to \infty} H_F(\wt{u}, g_n u_n) = \lim_{n \to \infty} H_{F_{\Omega}(x)}(\wt{u}_n, u_n) \leq D_0,
\end{align*}
we can pass to another subsequence so that $\lim_{n \to \infty} g_n u_n$ exists in $F$. Since $g_n u_n \in g_n(u,x)$ this implies that $(u_{\infty}, x_{\infty}) \subset F$. Further
\begin{align*}
\lim_{n \rightarrow \infty} H_{F}(g_n u_n, g_n x) = \lim_{n \rightarrow \infty}  H_{F_\Omega(x)}(u_n, x) =\infty.
\end{align*}
Hence, $x_\infty \in \partial F$ and so $(u_{\infty}, x_{\infty})$ is a properly embedded line  in $F$. Then $F_\Omega(x) = g_1^{-1} F$ also contains a properly embedded line.  This completes the proof of the claim.

Now suppose that  $(a,b) \subset \overline{\Cc} \cap F_\Omega(x)$ is a properly embedded line. Then Corollary~\ref{cor:building_simplices} implies that $a,b,v_2,\dots, v_p$ are the vertices of a properly embedded simplex $S^\prime$ in $\Omega$ of dimension $p$. Then there exists $\wt{S} \in \Sc$ such that 
\begin{align*}
S^\prime \subset \Nc_\Omega\left( \wt{S}; D_0\right).
\end{align*}
Then Observation~\ref{obs:QI_to_Rk} implies that  $\dim \wt{S} \geq p > \dim S$. 

Now fix $x^\prime \in (a,b)$. Since $x^\prime \in F_\Omega(x)$, Lemma~\ref{lem:slide_along_faces} implies that $x^\prime, v_2,\dots, v_p$ are the vertices of a properly embedded simplex $S^{\prime\prime}$ and 
\begin{align*}
H_\Omega^{\Haus}\left( S, S^{\prime\prime} \right) \leq H_{F_\Omega(x)}(x,x^\prime).
\end{align*}
Since $S^{\prime\prime} \subset S^\prime$ we then have
\begin{align*}
S \subset \Nc_\Omega\left( \wt{S}; D_0+H_{F_\Omega(x)}(x,x^\prime)\right).
\end{align*}
Finally, since $\dim S < \dim \wt{S}$, we have a contradiction with the condition in Lemma~\ref{lem:refining_isolated_simplices}. 
\end{proof}

\subsection{Using the fact that simplices are coarsely isolated}  By Lemma~\ref{lem:base_case} and Observation~\ref{obs:faces_of_simplices_are_properly_embedded}, $F_S(x)$ is a properly embedded simplex in $F_{\Omega}(x)$ with dimension at least one. In particular, $\partial F_S(x) \neq \emptyset$. Next, recall from Definition~\ref{defn:faces_of_subsets} that 
\begin{align*}
F_\Omega\big(\partial F_S(x)\big) =  \bigcup_{y \in \partial F_S(x)} F_\Omega(y).
\end{align*}
In this subsection we will prove the following. 

\begin{proposition}\label{prop:local_equal} $\overline{\Cc} \cap F_\Omega\big(\partial F_S(x)\big)$ is a connected component of $\overline{\Cc} \cap \partial F_\Omega(x)$.
\end{proposition}

The rest of this subsection is devoted to the proof of Proposition~\ref{prop:local_equal}. 

\begin{lemma} $\overline{\Cc} \cap F_\Omega\big(\partial F_S(x)\big)$ is closed in $\overline{\Cc} \cap \partial F_\Omega(x)$.\end{lemma}

\begin{proof} Suppose $v_n \in \overline{\Cc} \cap F_\Omega\big(\partial F_S(x)\big)$ converges to $v_\infty \in \overline{\Cc} \cap \partial F_\Omega(x)$. Since $F_S(x)$ has finitely many faces, by passing to a subsequence we can suppose that there exists $y \in \partial F_S(x)$ such that $v_n \in F_\Omega(y)$ for all $n$. 

Since $F_\Omega(y) \subset \partial F_\Omega(x)$, our minimality assumption, see Equation~\eqref{eq:minimality_assumption}, implies that
\begin{align*}
R : & = H_{F_\Omega(y)}^{\Haus}\left( F_S(y), \overline{\Cc} \cap F_\Omega(y)  \right)  < +\infty. 
\end{align*}
So for every $n$ there exists $v_n' \in F_S(y) \subset \partial F_S(x)$ such that $H_{F_\Omega(y)}(v_n,v_n') \leq R$. Passing to a subsequence we can suppose that $v'_\infty:=\lim_{n \rightarrow \infty} v'_n \in \partial F_S(x)$ exists. Then by Proposition~\ref{prop:dist_est_and_faces}
\begin{align*}
v_\infty \in F_\Omega(v'_\infty) \subset F_\Omega\big(\partial F_S(x)\big)
\end{align*}
and the proof is complete. 
\end{proof}


\begin{figure}
\centering
\includegraphics[scale=0.5]{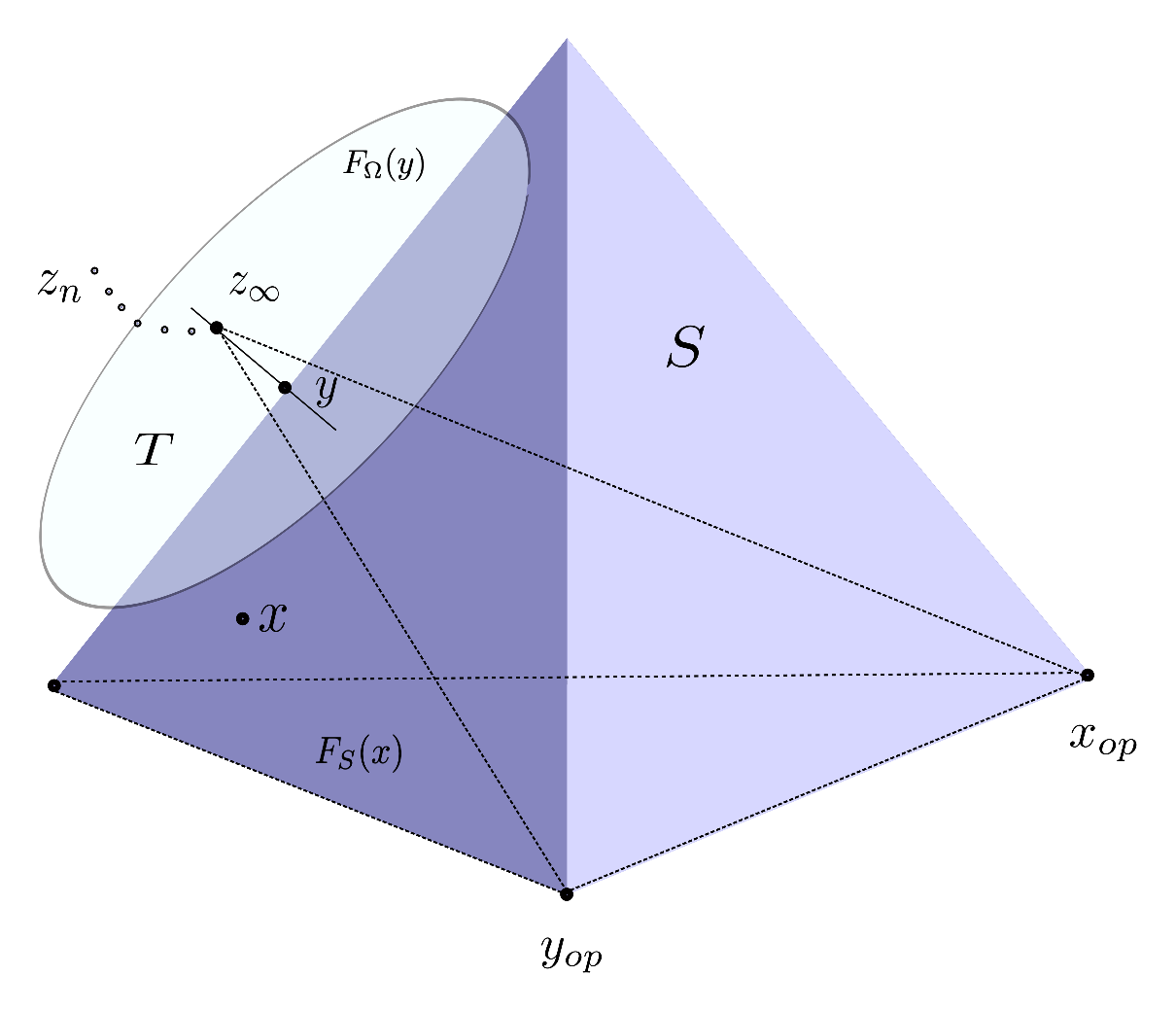}
\caption{ Diagram illustrating the proof by contradiction that $\overline{\Cc} \cap F_{\Omega}(\partial F_S(x))$ is open in $\overline{\Cc} \cap \partial F_{\Omega}(x)$. The sequence of points $z_n \not \in F_{\Omega}(\partial F_S(x))$ but converge to $z_{\infty} \in F_{\Omega}(y)$ where $y \in \partial F_S(x)$.}
\label{fig:proof_of_openness}
\end{figure}

Proving that $\overline{\Cc} \cap F_\Omega\big(\partial F_S(x)\big)$ is open in $\overline{\Cc} \cap \partial F_\Omega(x)$ is much more involved (see Figure \ref{fig:proof_of_openness}). Assume for a contradiction that this is false. Then there exist $y \in \partial F_S(x)$, $z_\infty \in F_\Omega(y)$, and a sequence 
\begin{align*}
z_n \in \overline{\Cc} \cap \partial F_\Omega(x) \setminus F_\Omega \big( \partial F_S(x) \big)
\end{align*}
such that $\lim_{n \to \infty}z_n = z_\infty$. 

Then let $y_{op} \in \partial F_S(x)$ be a point opposite to $y$ in $F_S(x)$ and let $x_{op} \in \partial S$ be a point opposite to $x$ in $S$. For each $n$, define
\begin{align*}
T_n:= \relint \Big(\CH_\Omega\{z_n, y_{op}, x_{op}\}\Big).
\end{align*}
Then define
\begin{align*}
T:= \relint \Big(\CH_\Omega\{z_\infty, y_{op}, x_{op}\}\Big).
\end{align*}

\begin{lemma} $T$ is a properly embedded simplex in $\Cc$. \end{lemma}

\begin{proof} By construction 
\begin{align*}
 \relint \Big(\CH_\Omega\{y, y_{op}, x_{op}\}\Big).
\end{align*}
is a properly embedded simplex in $S$ and hence $\Omega$ (see Corollary~\ref{cor:lines_between_opp_faces}). Then $T$ is also a properly embedded simplex in $\Omega$ by Lemma~\ref{lem:slide_along_faces}.
\end{proof}

\begin{lemma} For $n$ sufficiently large, $T_n$ is a properly embedded simplex in $\Cc$. \end{lemma}

\begin{proof} By construction $[y_{op},x_{op}] \subset \partial S \subset \partial \Omega$ and 
\begin{align*}
[z_n, y_{op}] \subset \overline{F_\Omega(x)} \subset \partial \Omega.
\end{align*}
Since $z_n \in \partial F_\Omega(x)$, Corollary~\ref{cor:lines_between_opp_faces} implies that $[z_n, x_{op}] \subset \partial \Omega$. Finally, $T_n$ converges to $T$ and so for $n$ large enough, $T_n$ intersects $\Omega$. Thus for $n$ sufficiently large, $T_n$ is a properly embedded simplex in $\Omega$.
\end{proof}

Then by passing to a subsequence we can suppose that $T_n \subset \Cc$ is a properly embedded simplex of dimension at least two for all $n$. Since $\Sc$ is coarsely complete, for each $n$ there exists $\wt{S}_n \in \Sc$ such that 
\begin{align*}
T_n \subset \Nc_\Omega\left( \wt{S}_n; D_0\right).
\end{align*}
Since the sequence $T_n$ converges to $T$, there exists some compact subset $K \subset \Omega$ such that $\wt{S}_n \cap K \neq \emptyset$ for all $n$. Thus by Proposition~\ref{prop:HK_section_3_1} and passing to a subsequence we can suppose that 
\begin{align*}
\wt{S} : = \wt{S}_n
\end{align*}
 for all $n \geq 0$. Then 
 \begin{align*}
T \cup \bigcup_{n \geq 0} T_n \subset \Nc_\Omega\left( \wt{S}; D_0\right).
\end{align*}
By Proposition~\ref{prop:dist_est_and_faces} there exist $\wt{x}, \wt{y}_{op}, \wt{y}, \wt{z}_n, \wt{x}_{op} \in \partial \wt{S}$ such that 
\begin{align*}
\wt{x} \in F_\Omega(x), \ \wt{y}_{op} \in F_\Omega(y_{op}), \ \wt{y} \in F_\Omega(y)=F_\Omega(z_\infty), \ \wt{z}_n \in F_\Omega(z_n), \text{ and } \wt{x}_{op} \in F_\Omega(x_{op}).
\end{align*}

\begin{lemma}\label{lem:haus_dist_S_Swt} $H_{F_\Omega(x)}^{\Haus}\left( F_S(x), F_{\wt{S}}(\wt{x}) \right) <+\infty$. 
\end{lemma}

\begin{proof} Since 
\begin{align*}
F_\Omega(y_{op})=F_\Omega(\wt{y}_{op}) \subset \partial F_\Omega(x),
\end{align*}
our minimality assumption, see Equation~\eqref{eq:minimality_assumption}, implies that
\begin{align*}
R_1 : & = H_{F_\Omega(y_{op})}^{\Haus}\left( F_S(y_{op}), F_{\wt{S}}(\wt{y}_{op}) \right) \\
& \leq H_{F_\Omega(y_{op})}^{\Haus}\left( F_S(y_{op}), \overline{\Cc} \cap F_\Omega(y_{op}) \right) +H_{F_\Omega(\wt{y}_{op})}^{\Haus}\left( \overline{\Cc} \cap F_\Omega(\wt{y}_{op}),  F_{\wt{S}}(\wt{y}_{op}) \right) < +\infty. 
\end{align*}
Likewise,
\begin{align*}
R_2:=H_{F_\Omega(y)}^{\Haus}\left( F_S(y), F_{\wt{S}}(\wt{y}) \right) < +\infty. 
\end{align*}

By Lemma~\ref{lem:line_segments_cover}, for every $p \in F_S(x)$, there exist $p_1 \in F_S(y)$ and $p_2 \in F_S(y_{op})$ such that $p \in (p_1, p_2)$. Then there exist $\wt{p}_1 \in F_{\wt{S}}(\wt{y})$ and $\wt{p}_2 \in F_{\wt{S}}(\wt{y}_{op})$ such that 
\begin{align*}
H_{F_\Omega(y)}(p_1, \wt{p}_1), \ H_{F_\Omega(y_{op})}(p_2, \wt{p}_2) \leq  \max\{ R_1, R_2\}.
\end{align*}
Then by Proposition~\ref{prop:Crampons_dist_est_2}
\begin{align*}
H_{F_\Omega(x)}&(p, F_{\wt{S}}(\wt{x})) \leq H_{F_S(x)}(p, ( \wt{p}_1, \wt{p}_2) ) \leq H_{F_S(x)}^{\Haus}\Big((p_1,p_2), ( \wt{p}_1, \wt{p}_2) \Big)\\
& \leq \max \left\{ H_{F_\Omega(y)}(p_1, \wt{p}_1), H_{F_\Omega(y_{op})}(p_2,\wt{p}_2) \right\} \leq \max\{R_1, R_2\}. 
\end{align*}
The same argument shows that if $q \in F_{\wt{S}}(\wt{x})$, then
\begin{align*}
H_{F_\Omega(x)}(q, F_{S}(x)) \leq \max\{R_1, R_2\}. 
\end{align*}
So 
\begin{equation*}
H_{F_\Omega(x)}^{\Haus}\left( F_S(x), F_{\wt{S}}(\wt{x}) \right) \leq \max \{R_1, R_2\}. \qedhere
\end{equation*}
\end{proof}

Since $\wt{z}_n \in \partial F_{\wt{S}}(\wt{x})$, Lemma~\ref{lem:haus_dist_S_Swt} and Proposition~\ref{prop:dist_est_and_faces} imply that there exists $a_n \in \partial F_S(x)$ with $a_n \in F_\Omega(\wt{z}_n) = F_\Omega(z_n)$. So $z_n \in F_\Omega(a_n) \subset F_\Omega(\partial F_S(x))$ which contradicts our assumption that 
\begin{align*}
z_n \in \overline{\Cc} \cap \partial F_\Omega(x) \setminus F_\Omega \big( \partial F_S(x) \big).
\end{align*}
Hence $\overline{\Cc} \cap F_\Omega\big(\partial F_S(x)\big)$ is open in $\overline{\Cc} \cap \partial F_\Omega(x)$.

\subsection{Using the group action} In this subsection we use the action of $\Stab_{\Lambda}(F_S(x))$ to upgrade Proposition~\ref{prop:local_equal}. We begin with the following observation

\begin{observation}\label{obs:stab-invariant}
$\Stab_{\Lambda}(F_S(x))$ acts co-compactly on $F_S(x)$ and $\Stab_{\Lambda}(F_S(x)) \leq \Stab_{\Lambda}(F_\Omega(x))$. 
\end{observation}

\begin{proof}
The first assertion follows from Proposition~\ref{prop:HK_section_3_1} and Observation \ref{obs:cocompact-action-on-simplices}.

For the second, if $g \in \Stab_{\Lambda}(F_S(x))$, then 
\begin{align*}
F_{\Omega}(x) \cap g F_{\Omega}(x)\supset  F_S(x) \cap gF_S(x) =F_S(x) \neq \emptyset.
\end{align*}
Hence $g F_{\Omega}(x)=F_{\Omega}(x)$. 
\end{proof}

\begin{proposition}\label{prop:local_equal_upgrade} $ \overline{\Cc} \cap F_\Omega\big(\partial F_S(x)\big)=\overline{\Cc} \cap \partial F_\Omega(x)$.
\end{proposition}

\begin{proof} By Observation \ref{obs:faces_of_simplices_are_properly_embedded}, $ \overline{\Cc} \cap F_\Omega\big(\partial F_S(x)\big)\subset \overline{\Cc} \cap \partial F_\Omega(x)$. For the other inclusion, fix $z \in \overline{\Cc} \cap \partial F_\Omega(x)$. 

\medskip

\noindent \textbf{Case 1:} There exists  $y \in \partial F_S(x)$ such that $[y,z] \subset \partial F_\Omega(x)$. 

\medskip

Then $y$ and $z$ are in the same connected component of $\overline{\Cc} \cap  \partial F_\Omega(x)$. So Proposition~\ref{prop:local_equal} implies that
\begin{align*}
z \in  \overline{\Cc} \cap F_\Omega(\partial F_S(x)).
\end{align*}

\medskip

\noindent \textbf{Case 2:} $(y,z) \subset F_\Omega(x)$ for every $y \in \partial F_S(x)$. 

\medskip

Using Observation~\ref{obs:stab-invariant},  there exists an unbounded sequence $g_n \in \Stab_{\Lambda}(F_S(x))$. By passing to a subsequence we can suppose that $y_1:=\lim_{n \to \infty} g_n(x)$ and $y_2:=\lim_{n \to \infty}g_n^{-1}(x)$ both exist. Then  $y_1,y_2 \in \partial F_S(x)$. Observation~\ref{obs:stab-invariant} also implies that
\begin{align*}
\{g_n z : n \geq 0\} \subset  \overline{\Cc} \cap \partial F_\Omega(x).
\end{align*}
We claim that this set intersects $F_\Omega(\partial F_S(x))$.

Let $V := \Spanset  F_\Omega(x)$. By passing to a subsequence, we can assume that $g_n|_V$ converges in $\Pb(\End(V))$ to some $T \in \Pb(\End(V))$. Then
\begin{align*}
T(w) = \lim_{n \rightarrow \infty} g_n(w)
\end{align*}
for all $w \in \Pb(V) \setminus \Pb(\ker T)$.
By Proposition~\ref{prop:dynamics_of_automorphisms}
\begin{align*}
{ \rm image}(T) \subset \Spanset F_{F_\Omega(x)}(y_1)= \Spanset F_{\Omega}(y_1),
\end{align*}
$y_2 \in \Pb(\ker T)$, and $\Pb(\ker T) \cap F_\Omega(x) = \emptyset$. 

We claim that $z \notin \Pb(\ker T)$. Otherwise, $[y_2,z] \subset \Pb(\ker T)$ and so $[y_2,z] \subset \partial F_\Omega(x)$. This contradicts our assumptions for Case 2. Then 
\begin{align*}
T(z) = \lim_{n \rightarrow \infty} g_n(z) \in \overline{\Cc} \cap \Pb( \Spanset F_{\Omega}(y_1)) = \overline{\Cc} \cap\overline{F_{\Omega}(y_1)} \subset \overline{\Cc} \cap \overline{F_{\Omega}(\partial F_S(x))}.
\end{align*}
Then Proposition~\ref{prop:local_equal} implies that 
\begin{align*}
g_n z \in  \overline{\Cc} \cap F_\Omega\big(\partial F_S(x)\big)
\end{align*}
for $n$ sufficiently large. Then 
\begin{align*}
z \in  g_n^{-1} \Big( \overline{\Cc} \cap F_\Omega\big(\partial F_S(x)\big)\Big) =  \overline{\Cc} \cap F_\Omega\big(\partial F_S(x)\big)
\end{align*}
and the proof is complete in this case. 
\end{proof}

\subsection{Finishing the proof of Theorem~\ref{thm:bd_faces}}

Since $H^{\Haus}_{F_\Omega(x)}(\overline{\Cc}\cap F_{\Omega}(x), F_S(x))=\infty$, for every $n \geq 1$,  there exists $w_n \in \overline{\Cc} \cap F_\Omega(x)$ with 
\begin{align*}
H_{F_\Omega(x)}\left(w_n, F_S(x)\right) \geq n. 
\end{align*}
Then for each $n$, pick $x_n \in F_S(x)$ such that
\begin{align*}
H_{F_\Omega(x)}\left(w_n, x_n\right) = H_{F_\Omega(x)}\left(w_n, F_S(x)\right).
\end{align*}
Using Observation~\ref{obs:stab-invariant}, translating by elements in $\Stab_{\Lambda}(F_S(x))$, and passing to a subsequence,  we can assume that
\begin{align*}
x_\infty : = \lim_{n \rightarrow \infty} x_n  \text{ exists and  } x_\infty \in F_S(x). 
\end{align*}
By passing to a further subsequence we can also assume that 
\begin{align*}
w_\infty : = \lim_{n \rightarrow \infty} w_n \in \overline{\Cc} \cap \overline{F_\Omega(x)}
\end{align*}
exists. In fact, $w_\infty \in \overline{\Cc} \cap \partial F_\Omega(x)$ since
\begin{align*}
\lim_{n \rightarrow \infty} H_{F_\Omega(x)}(w_n, x_\infty) \geq \lim_{n \rightarrow \infty} \Big( H_{F_\Omega(x)}(w_n, x_n)-H_{F_\Omega(x)}(x_n, x_\infty) \Big)= \infty.
\end{align*}
So Proposition~\ref{prop:local_equal_upgrade} implies that $w_\infty \in F_\Omega(y)$ for some $y \in \partial F_S(x)$.

Next fix $p \in [x_\infty, w_\infty) \subset F_{\Omega}(x)$. Then by Proposition~\ref{prop:Crampons_dist_est_2}
\begin{align*}
H_{F_\Omega(x)}(p, F_S(x))& \leq H_{F_\Omega(x)}(p, [x_\infty, y)) \leq H_{F_\Omega(x)}^{\Haus}\left( [x_\infty, w_\infty), [x_\infty,y)\right)\\
& \leq H_{F_\Omega(y)}(w_\infty,y). \nonumber
\end{align*}
Then fix a sequence $p_n \in [x_n, w_n]$ with $\lim_{n \to \infty} p_n = p$. Then
\begin{align*}
H_{F_{\Omega}(x)}(p_n,x_n) &=H_{F_{\Omega}(x)}(w_n,x_n)- H_{F_{\Omega}(x)}(w_n,p_n) \\
&=H_{F_{\Omega}(x)}(w_n,F_S(x))-H_{F_{\Omega}(x)}(w_n,p_n)\\
&\leq H_{F_{\Omega}(x)}(w_n,p_n)+H_{F_{\Omega}(x)}(p_n,F_S(x)) - H_{F_{\Omega}(x)}(w_n,p_n)\\
&= H_{F_{\Omega}(x)}(p_n,F_S(x)). 
\end{align*}
Taking $n \rightarrow \infty$ yields
\begin{align*}
H_{F_\Omega(x)}(p,x_\infty) \leq H_{F_\Omega(x)}(p, F_S(x))  \leq H_{F_\Omega(y)}(w_\infty,y).
\end{align*}
Since $p \in [x_\infty, w_\infty)$ is arbitrary, we have
\begin{align*}
\infty = \lim_{\substack{q \in [x_\infty, w_\infty) \\ q \rightarrow w_\infty}} H_{F_\Omega(x)}(q,x_\infty) \leq H_{F_\Omega(y)}(w_\infty,y) < \infty
\end{align*}
and we have a contradiction. This finishes the proof of Theorem \ref{thm:bd_faces}.

\section{Proof of Theorem~\ref{thm:S-core-exists}}\label{sec:intersection_of_nbhds}

For the rest of the section let $(\Omega, \Cc, \Lambda)$ be a naive convex co-compact triple with coarsely isolated simplices.  We will describe a procedure for producing a family of strongly isolated, coarsely complete, and $\Lambda$-invariant maximal properly embedded simplices in $\Cc$ of dimension at least two.

Let $\Sc_{\max}$ denote the family of all maximal properly embedded simplices in $\Cc$ of dimension at least two. Then let $X \subset \Sc_{\max}$ be the set of simplices $S \in \Sc_{\max}$ where there exists some $S^\prime \in \Sc_{\max}$ with $\dim S < \dim S^\prime$ and
\begin{align*}
 \sup_{p \in S} H_\Omega \left( p, S^\prime \right) < +\infty.
\end{align*}
Next let $\wh{\Sc}_{\max}:= \Sc_{\max} \setminus X$. That is, $\wh{\Sc}_{\max}$ consists of the maximal properly embedded simplices of dimension at least two that are not contained in a tubular neighborhood of a properly embedded simplex with strictly larger dimension. 

For each simplex $S \in \wh{\Sc}_{\max}$ we construct a new simplex $\Phi(S)$ as follows. Let $v_1,\dots, v_p$ be the vertices of $S$. Then define (see Remark \ref{rem:s_core_well_defined} to see why this is well defined)
 \begin{align*}
w_j: = {\rm CoM}_{F_\Omega(v_j)}\Big( \overline{\Cc} \cap F_\Omega(v_j) \Big) \text{ for } 1 \leq j \leq p
\end{align*}
and  
\begin{align*}
\Phi(S) := \Omega \cap \Pb \left( \Spanset \{w_1,\dots, w_p\} \right).
\end{align*}
Finally, define 
\begin{align*}
\Sc_{\core} :=\left\{ \Phi(S) : S \in \wh{\Sc}_{\max}\right\}.
\end{align*}

Theorem~\ref{thm:S-core-exists} will be a consequence of the following theorem. 

\begin{theorem} 
\label{thm:s_core}
$\Sc_{\core}$ is a well-defined, strongly isolated, coarsely complete, and $\Lambda$-invariant family of maximal properly embedded simplices in $\Cc$ of dimension at least two.
\end{theorem}

\begin{remark} \
\label{rem:s_core_well_defined}
\begin{enumerate}
\item To show that $\Sc_{\core}$ is well-defined we need to show that $\overline{\Cc} \cap F_\Omega(v)$ is a compact subset of $F_\Omega(v)$ for every simplex $S \in \wh{\Sc}_{\max}$ and vertex $v$ of $S$. 
\item The map $\Phi$ selects from each family of parallel simplices a canonical ``core'' simplex, thus motivating the notation $\Sc_{\core}$. 
\end{enumerate}

\end{remark}

The rest of this section is devoted to the proof of Theorem \ref{thm:s_core}.  Let $\Sc_0$ be an isolated, coarsely complete, and $\Lambda$-invariant family of maximal properly embedded simplices in $\Cc$ of dimension at least two. By Theorem~\ref{thm:bd_faces} and passing to a subfamily of $\Sc_0$ we can assume that there exists $R > 0$ such that: if $S \in \Sc_0$ and $x \in \partial S$, then 
\begin{align}
\label{eq:defn_of_R}
H_{F_\Omega(x)}^{\Haus}\left(\overline{\Cc} \cap F_\Omega(x), F_S(x) \right)=H_{F_\Omega(x)}^{\Haus}\left(\overline{\Cc} \cap F_\Omega(x), \overline{S} \cap F_\Omega(x) \right) \leq R.
\end{align}
By Lemma~\ref{lem:refining_isolated_simplices} and passing to another subfamily of $\Sc_0$ we can also assume that: If $S_1, S_2 \in \Sc_0$ and $\dim S_1 < \dim S_2$, then 
\begin{align}
\label{eq:second_prop_of_S0}
\infty = \sup_{p \in S_1} H_\Omega \left( p, S_2 \right). 
\end{align}

The next two lemmas show that $\Sc_0$ and $\wh{\Sc}_{\max}$ are ``coarsely the same.''

\begin{lemma} $\Sc_0 \subset \wh{\Sc}_{\max}$. \end{lemma}

\begin{proof} Fix $S_0 \in \Sc_0$ and suppose $S_1 \in \Sc_{\max}$ with $\dim S_1 > \dim S_0$. Since $\Sc_0$ is coarsely complete there exists $S_2 \in \Sc_0$ and $r>0$ such that 
\begin{align*}
S_1 \subset \Nc_\Omega(S_2; r).
\end{align*}
Then, Observation \ref{obs:QI_to_Rk} implies that $\dim S_1 \leq \dim S_2$. So by Equation~\eqref{eq:second_prop_of_S0} 
\begin{align*}
\sup_{p \in S_0} H_\Omega \left( p, S_1 \right) \geq -r + \sup_{p \in S_0} H_\Omega \left( p, S_2 \right) = \infty. 
\end{align*}
Since $S_1 \in \Sc_{\max}$ was an arbitrary simplex with $\dim S_1 > \dim S_0$, we see that $S_0 \in \Sc_{\max} \setminus X=\wh{\Sc}_{\max}$. 
\end{proof}

\begin{lemma}\label{lem:shadows} If $S \in \wh{\Sc}_{\max}$ has vertices $v_1,\dots, v_{p}$, then there exists $S_0 \in \Sc_0$ with $\dim S_0 = \dim S$ and a labelling $w_1,\dots,w_{p}$ of the vertices of $S_0$ such that 
\begin{align*}
F_\Omega(v_j) = F_\Omega(w_j)
\end{align*}
for all $1 \leq j \leq p$. Moreover,
\begin{align*}
H_{\Omega}^{\Haus}( S, S_0) \leq R. 
\end{align*}
\end{lemma}

\begin{proof} Since $\Sc_0$ is coarsely complete, there exists $S_0 \in \Sc_0$ and $r >0$ such that 
\begin{align*}
S \subset \Nc_\Omega(S_0;r).
\end{align*}
Then, Observation \ref{obs:QI_to_Rk} implies that $\dim S \leq \dim S_0$. Then, since $S \in \wh{\Sc}_{\max}$, we must have $\dim S = \dim S_0$. 

By Proposition~\ref{prop:dist_est_and_faces} there exist $w_1,\dots, w_p \in \partial S_0$ such that 
\begin{align*}
F_\Omega(v_j) = F_\Omega(w_j)
\end{align*}
for all $1 \leq j \leq p$. Then by Lemma~\ref{lem:slide_along_faces}
\begin{align*}
S_0^\prime:=\Pb( \Spanset\{w_1,\dots,w_p\}) \cap\Omega \subset S_0
\end{align*}
is a properly embedded simplex with vertices $w_1,\dots,w_p$. Then, since 
\begin{align*}
\dim S_0^\prime = \dim S = \dim S_0,
\end{align*}
we must have $S_0 = S_0^\prime$. This proves the first assertion in the lemma. 

 Now the ``moreover'' part is a consequence of Equation~\eqref{eq:defn_of_R} and Lemma~\ref{lem:slide_along_faces}.
\end{proof}

\begin{lemma}\label{lem:S0_phi_is_onto} If $S \in \wh{\Sc}_{\max}$, then $\Phi(S)$ is a well defined maximal properly embedded simplex in $\Cc$ and 
\begin{align}
\label{eq:phi_is_bd_in_haus}
H_{\Omega}^{\Haus}( S, \Phi(S)) \leq 2R. 
\end{align}
Moreover,
\begin{align*}
\Sc_{\core} = \{ \Phi(S) : S \in \Sc_0\}.
\end{align*}
 \end{lemma}

\begin{proof} Fix $S \in \wh{\Sc}_{\max}$ and let $v_1,\dots, v_p$ be the vertices of $S$. By Lemma~\ref{lem:shadows} there exist $S_0 \in \Sc_0$ and a labelling $w_1,\dots,w_p$ of the vertices of $S_0$ such that $F_\Omega(w_j) = F_\Omega(v_j)$ for all $1 \leq j \leq p$. Then by Equation~\eqref{eq:defn_of_R} 
\begin{align*}
H_{F_\Omega(v_j)}^{\Haus}\left(\overline{\Cc} \cap F_\Omega(v_j), \{v_j\} \right) 
&\leq H_{F_\Omega(w_j)}^{\Haus}\left(\overline{\Cc} \cap F_\Omega(w_j), \{w_j\} \right)+H_{F_\Omega(v_j)}( w_j, v_j)  \\
&\leq 2H_{F_\Omega(w_j)}^{\Haus}\left(\overline{\Cc} \cap F_\Omega(w_j), \{w_j\} \right) \leq  2R.
\end{align*}
So $\overline{\Cc} \cap F_\Omega(v_j)$ is a compact subset of $F_\Omega(v_j)$. Hence
 \begin{align*}
 {\rm CoM}_{F_\Omega(v_j)}\Big( \overline{\Cc} \cap F_\Omega(v_j) \Big) 
\end{align*}
is well defined. Thus $\Phi(S)$ is well defined. 

Then, Lemma~\ref{lem:slide_along_faces} implies that $\Phi(S)$ is a properly embedded simplex and   
\begin{align*}
H_{\Omega}^{\Haus}( S, \Phi(S)) \leq 2R. 
\end{align*}
Further $\Phi(S)$ is maximal since $S \in \wh{\Sc}_{\max}$. 

Finally, since $F_\Omega(w_j) = F_\Omega(v_j)$ for all $1 \leq j \leq p$, we have $\Phi(S) = \Phi(S_0)$. Since $S \in \wh{\Sc}_{\max}$ was arbitrary this implies the ``moreover'' part of the lemma. 
\end{proof}

\begin{lemma}\label{lem:S1=S2}
If $S_1, S_2 \in \Sc_{\core}$ and $H_{\Omega}^{\Haus}\left(S_1,  S_2\right) < \infty$, then $S_1 = S_2$. 
\end{lemma}

\begin{proof} Arguing as in the proof of Lemma~\ref{lem:shadows}, there exist a labelling $v_1,\dots, v_p$ of the vertices of $S_1$ and a labelling $w_1,\dots,w_p$ of the vertices of $S_2$ such that 
\begin{align*}
F_\Omega(v_j) = F_\Omega(w_j)
\end{align*}
for all $1 \leq j \leq p$. Then, by the definition of $\Phi$,
\begin{align*}
w_j= {\rm CoM}_{F_\Omega(w_j)}\left( \overline{\Cc} \cap F_\Omega(w_j) \right) = {\rm CoM}_{F_\Omega(v_j)}\left( \overline{\Cc} \cap F_\Omega(v_j) \right) = v_j
\end{align*}
for all $1 \leq j \leq p$. 
\end{proof}

\begin{lemma}\label{lem:S-has-prop-IS}
$\Sc_{\core}$ is coarsely complete and $\Lambda$-invariant.
\end{lemma}

\begin{proof}By construction $\Sc_{\core}$ is $\Lambda$-invariant. 

Since $\Sc_0$ is coarsely complete, there exists $D_0 > 0$ such that: If $S$ is a properly embedded simplex in $\Cc$ of dimension at least two, then there exists $S^\prime \in \Sc_0$ such that:
\begin{align*}
S \subset \Nc_\Omega(S^\prime;D_0).
\end{align*}
Then Equation~\eqref{eq:phi_is_bd_in_haus} implies that
\begin{align*}
S \subset \Nc_\Omega(\Phi(S^\prime);D_0+2R).
\end{align*}
So $\Sc_{\core}$ is coarsely complete.
\end{proof}

We complete the proof of the theorem by establishing the following lemma.

\begin{lemma} $\Sc_{\core}$ is strongly isolated: for any $r > 0$ there exists $D_2(r) > 0$ such that if $S_1, S_2 \in \Sc_{\core}$ are distinct, then
\begin{align*}
\diam_\Omega \Big( \Nc_\Omega(S_1;r) \cap \Nc_\Omega(S_2;r) \Big) \leq D_2(r).
\end{align*}
\end{lemma}

\begin{proof}Fix $r > 0$. Suppose for a contradiction that such a $D_2(r)> 0$ does not exist. Then by Lemma~\ref{lem:S0_phi_is_onto} for every $n \geq 0$ there exist $S_{1,n}, S_{2,n} \in \Sc_0$ such that $\Phi(S_{1,n}) \neq \Phi(S_{2,n})$ and
\begin{align*}
\diam_\Omega\Big( \Nc_\Omega(\Phi(S_{1,n});r) \cap \Nc_\Omega(\Phi(S_{2,n});r) \Big) > n.
\end{align*}
Then by Equation~\eqref{eq:phi_is_bd_in_haus}
\begin{align*}
\diam_\Omega\Big( \Nc_\Omega(S_{1,n};r_0) \cap \Nc_\Omega(S_{2,n};r_0) \Big) > n.
\end{align*}
where $r_0 : = r + 2R$.
 
Pick 
\begin{align*}
a_n, b_n \in   \Nc_\Omega(S_{1,n};r_0) \cap \Nc_\Omega(S_{2,n};r_0)
\end{align*}
with $H_\Omega(a_n,b_n) \geq n$. Let $m_n \in [a_n, b_n]$ be such that  
\begin{align}
\label{eq:un_selection}
H_\Omega(a_n,m_n)\geq n/2 ~~\text{ and }~~ H_\Omega(b_n, m_n) \geq n/2.
\end{align}
For each $n$, we can find $\gamma_n \in \Lambda$ such that 
\begin{align*}
\{ \gamma_n m_n : n \geq 0\} 
\end{align*}
is relatively compact in $\Omega$. So by passing to a subsequence we can suppose that $m:=\lim_{n \to \infty} \gamma_n m_n$ exists  in $\Cc$. Passing to another subsequence we can assume that $a:= \lim_{n \to \infty} \gamma_n a_n$ and $b:=\lim_{n \to \infty} \gamma_n b_n$ exist in $\overline{\Cc}$. Then Equation~\eqref{eq:un_selection} implies that $a,b \in \partiali \Cc$ and so $(a,b)$ is a properly embedded line in $\Cc$. Finally, using Proposition~\ref{prop:HK_section_3_1} and passing to another subsequence we can suppose that 
\begin{align*}
S_1:=\gamma_n S_{n,1} = \gamma_m S_{m,1}
\end{align*}
 and 
 \begin{align*}
S_2:=\gamma_n S_{n,2} = \gamma_m S_{m,2}
\end{align*}
for all $n,m \geq 0$. Then by construction $S_1,S_2 \in \Sc_{0}$ and $\Phi(S_1) \neq \Phi(S_2)$.

Notice that 
\begin{align*}
(a,b) \subset  \Nc_\Omega(S_{1};r_0+1) \cap \Nc_\Omega(S_{2};r_0+1).
\end{align*}
Proposition~\ref{prop:dist_est_and_faces} then implies that $\overline{S}_1$ and $\overline{S}_2$ both intersect $F_\Omega(a)$. Then Equation~\eqref{eq:defn_of_R}  implies that 
\begin{align*}
H_{F_\Omega(a)}^{\Haus}&\left(\overline{S}_1 \cap F_{\Omega}(a),  \overline{S}_2\cap F_{\Omega}(a)\right) \\
& \leq H_{F_\Omega(a)}^{\Haus}(\overline{S}_1 \cap F_{\Omega}(a), \overline{\Cc} \cap F_\Omega(a))+ H_{F_\Omega(a)}^{\Haus}(\overline{\Cc} \cap F_\Omega(a),  \overline{S}_2\cap F_{\Omega}(a))\leq 2R. 
\end{align*}
The same reasoning shows that 
\begin{align*}
H_{F_\Omega(b)}^{\Haus}&\left(\overline{S}_1 \cap F_{\Omega}(b),  \overline{S}_2\cap F_{\Omega}(b)\right)\leq 2R. 
\end{align*}

\noindent Now we claim that 
\begin{align*}
H_{\Omega}^{\Haus}\left(S_1,  S_2\right) \leq 2R. 
\end{align*}

By symmetry, it is enough to fix $p \in S_1$ and show that 
\begin{align*}
H_\Omega(p,S_2) \leq 2R.
\end{align*} 
Fix $a'\in \overline{S}_1 \cap F_{\Omega}(a)$ and $b'\in \overline{S}_1 \cap F_{\Omega}(b)$. Since $(a,b) \subset \Omega$,  Observation~\ref{obs:faces} part (4) implies that $(a',b') \subset \Omega$.

Then by Lemma~\ref{lem:line_segments_cover}, there exist $a_1 \in F_{S_1}(a') \subset \overline{S}_1 \cap F_\Omega(a)$ and $b_1 \in F_{S_1}(b') \subset \overline{S}_2 \cap F_\Omega(b)$ such that $p\in (a_1,b_1)$. Then there exist  $a_2 \in \overline{S}_2 \cap F_\Omega(a)$ and $b_2 \in \overline{S}_2 \cap F_\Omega(b)$ with 
\begin{align*}
\max\left\{ H_{F_\Omega(a)}(a_1,a_2), H_{F_\Omega(b)}(b_1,b_2) \right\} \leq 2R.
\end{align*}
Then by Proposition~\ref{prop:Crampons_dist_est_2}

\begin{align*}
H_\Omega(p,S_2) &\leq H_{\Omega}(p, (a_2, b_2)) \leq H_{\Omega}^{\Haus}\Big((a_1,b_1), (a_2, b_2)\Big)\\
& \leq \max \left\{ H_{F_\Omega(a)}(a_1,a_2), H_{F_\Omega(b)}(b_1,b_2) \right\} \leq 2R.
\end{align*}

So
\begin{align*}
H_{\Omega}^{\Haus}\left(S_1,  S_2\right) \leq 2R. 
\end{align*}
By Lemma \ref{lem:S0_phi_is_onto}, $\hil^{\Haus}(\Phi(S_1),\Phi(S_2)) \leq 6R.$
Then, by Lemma~\ref{lem:S1=S2}, $\Phi(S_1)=\Phi(S_2)$  and  we have a contradiction. 

Thus there exists $D_2(r) > 0$ such that: if $S_1, S_2 \in \Sc_{\core}$ are distinct, then
\begin{equation*}
\diam_\Omega \Big( \Nc_\Omega(S_1;r) \cap \Nc_\Omega(S_2;r) \Big) \leq D_2(r). \qedhere
\end{equation*}
\end{proof}

\section{Half triangles  in the ideal boundary}\label{sec:half_triangles}

In this section we verify property (6) of Theorem~\ref{thm:properties_of_ncc}.

\begin{definition}\label{defn:half_triangle}
Suppose $\Omega \subset \Pb(\Rb^d)$ is a properly convex domain. A list of three points $a,b,c$ form a \emph{half triangle in $\Omega$} if $[a,b],[b,c] \subset \partial \Omega$, $(a,c) \subset \Omega$, and $a \neq c$. 
\end{definition}

\begin{theorem}\label{thm:half_triangle_nearby_simplex} 
Suppose that $(\Omega, \Cc,\Lambda)$ is a naive convex co-compact triple with coarsely isolated simplices. Let $\Sc$ be a strongly isolated, coarsely complete, and $\Lambda$-invariant family of maximal properly embedded simplices in $\Cc$ of dimension at least two. If $a,b,c \in \partiali \Cc$ form a half triangle, then there exists $S \in \Sc$ such that $a,b,c \in F_\Omega(\partial S)$. 
\end{theorem}

As a corollary we observe that simplices in $\Sc$ cannot have ``half triangles sticking out.''

\begin{corollary}\label{cor:half_triangles_out_of_maximal_simplices} 
Suppose that $(\Omega, \Cc,\Lambda)$ is a naive convex co-compact triple with coarsely isolated simplices. Let $\Sc$ be a strongly isolated, coarsely complete, and $\Lambda$-invariant family of maximal properly embedded simplices in $\Cc$ of dimension at least two. If $S \in \Sc$; $a,c \in \partial S$; $b \in \partiali \Cc$; and $a,b,c$ form a half triangle in $\Omega$, then $b \in F_\Omega(\partial S)$. 
\end{corollary}

\begin{proof}[Proof of Corollary~\ref{cor:half_triangles_out_of_maximal_simplices} assuming Theorem~\ref{thm:half_triangle_nearby_simplex}]
By Theorem~\ref{thm:half_triangle_nearby_simplex} there exists $S' \in \Sc$ such that $a,b,c \in F_\Omega(\partial S')$. So there exist $a',b',c'  \in \partial S'$ such that $a \in F_\Omega(a')$, $b \in F_\Omega(b')$, and $c \in F_\Omega(c')$. Define 

\begin{align*}
M := \max \left\{ H_{F_\Omega(a)}(a,a'), H_{F_\Omega(c)}(c,c') \right\}. 
\end{align*}
By Proposition~\ref{prop:Crampons_dist_est_2}
\begin{align*}
H_\Omega^{\Haus}\Big( (a,c), (a', c') \Big) \leq M
\end{align*}
and so 
\begin{align*}
(a,c) \subset  S \cap \Nc_\Omega(S';M).
\end{align*}
Then
\begin{align*}
\infty=\diam_\Omega( \Nc_\Omega(S;M) \cap \Nc_\Omega(S';M) ).
\end{align*}
Since $\Sc$ is strongly isolated, $S=S'$. So $b' \in \partial S' = \partial S$ and hence $b \in F_\Omega(b') \subset F_\Omega(\partial S)$.  
\end{proof}

We begin the proof of Theorem~\ref{thm:half_triangle_nearby_simplex} with a lemma. 

\begin{lemma}\label{lem:geom_near_corner} Suppose $(\Omega, \Cc,\Lambda)$ is a naive convex co-compact triple. Assume $a,b,c \in \partiali \Cc$ form a half triangle and $V = \Spanset\{a,b,c\}$. For any $r >0$ and $\epsilon > 0$ there exists a neighborhood $U$ of $b$ in $\Pb(V)$ such that: if $x \in  U\cap \Cc$, then there exists a properly embedded simplex $S=S(x) \subset \Cc$ of dimension at least two such that 
\begin{align}
\label{eq:ball_near_simplex}
B_{\Omega}(x;r) \cap \Pb(V) \subset \Nc_\Omega(S;\epsilon).
\end{align}
\end{lemma}
\begin{proof} Fix $r >0$ and $\epsilon > 0$. Suppose for a contradiction that such a neighborhood $U$ does not exist. Then we can find a sequence $p_n \in \Cc \cap \Pb(V)$ such that $\lim_{n \to \infty} p_n = b$ and each $p_n$ does not satisfy Equation~\eqref{eq:ball_near_simplex} for any properly embedded simplex $S \subset \Cc$ of dimension at least two.

After passing to a subsequence we can find $\gamma_n \in \Lambda$ such that $\gamma_n p_n \rightarrow p \in \Cc$. Passing to a further subsequence we can suppose that $\gamma_n a \rightarrow a_\infty$, $\gamma_n b \rightarrow b_\infty$, and $\gamma_n c \rightarrow c_\infty$. Then $[a_\infty, b_\infty], [b_\infty, c_\infty] \subset \partiali\Cc$ and by the definition of the Hilbert metric
 \begin{align*}
\infty= \lim_{n \rightarrow \infty} H_\Omega\Big(p_n, (a,c)\Big) =  \lim_{n \rightarrow \infty} H_\Omega\Big(\gamma_n p_n, (\gamma_n a,\gamma_n c)\Big).
\end{align*}
So $[a_\infty, c_\infty] \subset \partiali\Cc$. Thus $a_\infty, b_\infty, c_\infty$ are the vertices of a properly embedded simplex $S \subset \Cc$. However, for $n$ sufficiently large we have 
\begin{align*}
B_{\Omega}(\gamma_n p_n;r) \cap \gamma_n\Pb(V) \subset \Nc_\Omega(S;\epsilon)
\end{align*}
and so 
\begin{align*}
B_{\Omega}(p_n;r) \cap \Pb(V) \subset \Nc_\Omega(\gamma_n^{-1}S;\epsilon).
\end{align*}
Hence we have a contradiction. 
\end{proof}

\begin{proof}[Proof of Theorem~\ref{thm:half_triangle_nearby_simplex}]
 By Theorem~\ref{thm:bd_faces}, there exists an isolated, coarsely complete, and $\Lambda$-invariant subfamily $\Sc^\prime \subset \Sc$  where 
\begin{align}
\label{eqn:D_0-used-for-half-triangle}
D_0:= \sup_{S \in \Sc^\prime} \sup_{x \in \partial S} H_{F_\Omega(x)}^{ \rm Haus}\Big( \overline{\Cc} \cap F_\Omega(x), F_S(x)\Big) <+\infty.
\end{align}
Since $\Sc^\prime$ is coarsely complete, there exists $D_1 > 0$ such that: if $S \subset \Cc$ is a properly embedded simplex of dimension at least two, then there exists $S^\prime \in \Sc^\prime$ with 
\begin{align*}
S \subset \Nc_\Omega(S^\prime; D_1).
\end{align*}

 As $\Sc$ is strongly isolated, so is $\Sc^\prime$. Thus there exists $D_2 > 0$ such that: if $S_1,S_2\in \Sc^\prime$ and 
\begin{align*}
\diam_\Omega \Big( \Nc_\Omega(S_1;1+D_1) \cap \Nc_\Omega(S_2;1+D_1)\Big) \geq D_2,
\end{align*}
then $S_1 = S_2$.

Define $V :=\Spanset\{a,b,c\}$. By Lemma~\ref{lem:geom_near_corner} there exists a neighborhood $U$ of $b$ in $\Pb(V)$ such that: if $x \in U\cap \Cc$, then there exists a properly embedded simplex $S=S(x)$ in $\Cc$ of dimension at least two with
\begin{align*}
B_{\Omega}(x;D_2) \cap \Pb(V) \subset \Nc_\Omega(S;1).
\end{align*}
Then for each $x \in U \cap \Cc$ there exists some $S_x \in \Sc^\prime$ such that 
\begin{align*}
B_{\Omega}(x;D_2) \cap \Pb(V) \subset \Nc_\Omega(S_x;1+D_1).
\end{align*}
By shrinking $U$ we can also assume that $U \cap \Cc$ is convex. 

We claim that $S_x = S_y$ for all $x,y \in U \cap \Cc$. Since $U \cap \Cc$ is convex, it is enough to verify this when $H_\Omega(x,y) \leq D_2/2$. In that case 
\begin{align*}
B_{\Omega}(y;D_2/2) \cap \Pb(V) \subset B_{\Omega}(x;D_2) \cap \Pb(V) \subset \Nc_\Omega(S_x;1+D_1).
\end{align*}
and so 
\begin{align*}
B_{\Omega}(y;D_2/2) \cap \Pb(V) \subset  \Nc_\Omega(S_x;1+D_1) \cap \Nc_\Omega(S_y;1+D_1).
\end{align*}
Since
\begin{align*}
\diam_\Omega& \Big( \Nc_\Omega(S_x;1+D_1) \cap \Nc_\Omega(S_y;1+D_1)\Big) \geq \diam_\Omega\Big( B_{\Omega}(y;D_2/2) \cap \Pb(V)\Big)  = D_2  
\end{align*}
we then have $S_x = S_y$. 

Next let $S = S_x$ for some (hence any) $x \in U \cap \Cc$. Then 
\begin{align*}
U \cap \Cc \subset \Nc_\Omega(S; 1+D_1).
\end{align*}
Fix some $a_1 \in (a,b) \cap U$ and $c_1 \in (b,c) \cap U$. Then by Proposition~\ref{prop:dist_est_and_faces} there exist $a^\prime_1, b^\prime, c^\prime_1 \in \partial S$ such that $a^\prime_1 \in F_\Omega(a_1)$, $b^\prime \in F_\Omega(b)$, and $c^\prime_1 \in F_\Omega(c_1)$. So $b \in F_\Omega(b^\prime) \subset F_\Omega(\partial S)$. 

We now show that $a \in F_\Omega(\partial S)$. We can find a sequence 
\begin{align*}
q_n \in \partiali \Cc \cap ~ [a_1,a) \subset \partiali \Cc \cap ~F_{\Omega}(a_1')
\end{align*}
 such that $\lim_{n \to \infty} q_n=a.$ Then, by Equation \eqref{eqn:D_0-used-for-half-triangle} there exists $q_n' \in F_S(a_1')$ with $\hil(q_n,q_n') \leq D_0$ . Then passing to a subsequence, $a':= \lim_{n \to \infty} q_n'$ exists in $\overline{F_S(a_1')}$ and by Proposition~\ref{prop:dist_est_and_faces}, $a \in F_{\Omega}(a').$ Thus, $a \in F_{\Omega}(\partial S).$  

The same argument shows that $c \in F_\Omega(\partial S)$. 
\end{proof}

\section{Proof of Theorem~\ref{thm:properties_of_ncc}}\label{sec:pf_of_properties_of_ncc}

Suppose $(\Omega, \Cc, \Lambda)$ is a naive convex co-compact triple with coarsely isolated simplices and $\Sc$ is a strongly isolated, coarsely complete, and $\Lambda$-invariant family of maximal properly embedded simplices in $\Cc$ of dimension at least two.
 
  \smallskip
 
\textbf{(1) and (2):} Proposition~\ref{prop:HK_section_3_1}.

\smallskip

 \textbf{(3):} Since $\Sc$ is coarsely complete there exists $D_0 > 0$ such that: if $S$ is a properly embedded simplex in $\Cc$ of dimension at least two, then there exists $S^\prime \in \Sc$ with 
\begin{align*}
S \subset \Nc_\Omega(S^\prime; D_0). 
\end{align*}

Applying Theorem~\ref{thm:max_abelian} to a maximal Abelian subgroup which contains $A$ shows that there exists a properly embedded simplex $S_0 \subset \Cc$ with $A \leq \Stab_{\Lambda}(S_0)$.  Since $\Sc$ is strongly isolated there exists a unique $S \in \Sc$ with 
\begin{align*}
S_0 \subset \Nc_\Omega(S;D_0).
\end{align*}
So by uniqueness $A \leq \Stab_{\Lambda}(S)$.

\smallskip

 \textbf{(4):} By Theorem~\ref{thm:bd_faces}, there exists a coarsely complete  subfamily $\Sc^\prime \subset \Sc$ and a constant $D_1 > 0$ such that: if $S \in \Sc^\prime$ and $x \in \partial S$, then 
 \begin{align*}
 H_{F_\Omega(x)}^{\Haus}\left( \overline{\Cc} \cap F_\Omega(x), F_S(x) \right) \leq D_1. 
 \end{align*}
We claim that $\Sc^\prime= \Sc$. Suppose that $S \in \Sc$. Since $\Sc^\prime$ is coarsely complete there exist $S^\prime \in \Sc^\prime$ and $D_0^\prime > 0$ such that 
\begin{align*}
S \subset \Nc_\Omega(S^\prime;D_0^\prime).
\end{align*}
But $S^\prime \in \Sc^\prime \subset \Sc$ and 
\begin{align*}
\diam_\Omega\left( \Nc_{\Omega}(S^\prime; D_0^\prime) \cap \Nc_{\Omega}(S; D_0^\prime)\right) \geq \diam_\Omega(S) = \infty,
\end{align*}
so $S = S^\prime \in \Sc^\prime$. Since $S \in \Sc$ was arbitrary we see that $\Sc^\prime = \Sc$. 

\smallskip

\textbf{(5):} Suppose $S_1,S_2\in \Sc$ and $\#(S_1 \cap S_2) > 1$. Then $S_1 \cap S_2$  contains a properly embedded line and hence 
\begin{align*}
\diam_\Omega\left(\Nc(S_1;r) \cap \Nc_\Omega(S_2;r) \right) =\infty
\end{align*}
for any $r > 0$. Thus $S_1 = S_2$ since $\Sc$ is strongly isolated. 

Suppose $S_1,S_2\in \Sc$ and $F_\Omega(\partial S_1) \cap F_\Omega(\partial S_2) \neq \emptyset$. Then there exist $s_1 \in \partial S_1$ and $s_2 \in \partial S_2$ with $F_\Omega(s_1) = F_\Omega(s_2)$. Fix $p_1 \in S_1$ and $p_2 \in S_2$. Then by Proposition~\ref{prop:Crampons_dist_est_2}
\begin{align*}
H_\Omega^{\Haus}\Big( [p_1,s_1), [p_2,s_2) \Big) \leq \max \{ H_\Omega(p_1,p_2), H_{F_\Omega(s_1)}(s_1,s_2)\}. 
\end{align*}
So for any $r > \max\{H_\Omega(p_1,p_2), H_{F_\Omega(s_1)}(s_1,s_2)\}$, 
\begin{align*}
\diam_\Omega(\Nc_\Omega(S_1;r) \cap \Nc_\Omega(S_2;r) ) =\infty.
\end{align*}
Thus  $S_1 = S_2$  since $\Sc$ is strongly isolated. 

\smallskip

 \textbf{(6):} Theorem~\ref{thm:half_triangle_nearby_simplex}.

\section{Proof of Theorem \ref{thm:IS-implies-rel-hyp}}
\label{sec:isolated-simplex-implies-rel-hyp}

 In this section we prove Theorem~\ref{thm:IS-implies-rel-hyp} which we recall here.

\begin{theorem}\label{thm:isolated-simplices-imply-rel-hyp}
Suppose $(\Omega, \Cc, \Lambda)$ is a naive convex co-compact triple with coarsely isolated simplices. Let $\Sc$ be a strongly isolated, coarsely complete, and $\Lambda$-invariant family of maximal properly embedded simplices in $\Cc$ of dimension at least two. Then
\begin{enumerate}
\item  $(\Cc, H_\Omega)$ is a relatively hyperbolic space with respect to $\Sc$.
\item $\Lambda$ has finitely many orbits in $\Sc$ and if $\{S_1,\dots, S_m\}$ is a set of orbit representatives, then $\Lambda$ is a relatively hyperbolic group with respect to 
\begin{align*}
\left\{\Stab_{\Lambda}(S_1),\dots, \Stab_{\Lambda}(S_m)\right\}.
\end{align*} 
Further each $\Stab_{\Lambda}(S_i)$ is virtually Abelian of rank at least two.
\end{enumerate}
\end{theorem}

For the rest of the section, fix a naive convex co-compact triple  $(\Omega, \Cc, \Lambda)$ with coarsely isolated simplices. Then fix a strongly isolated, coarsely complete, and $\Lambda$-invariant family $\Sc$ of maximal properly embedded simplices in $\Cc$ of dimension at least two. By Proposition~\ref{prop:HK_section_3_1}, $\Lambda$ has finitely many orbits in $\Sc$ and for each $S \in \Sc$, the group $\Stab_{\Lambda}(S)$ is virtually Abelian of rank at least two. Finally, fix  orbit representatives $S_1, \dots, S_m$ of the $\Lambda$ action on $\Sc$. 

By Proposition~\ref{prop:HK_section_3_1} again, if $S \in \Sc$, then $\Stab_{\Lambda}(S)$ acts co-compactly on $S$. Thus, by Theorem~\ref{thm:rh_quasi_isometry_inv},  $(\Cc,\hil)$ is relatively hyperbolic with respect to $\Sc$ if and only if $\Lambda$ is relatively hyperbolic relative to $\{\Stab_{\Lambda}(S_1), \dots, \Stab_{\Lambda}(S_m)\}$.

 So it is enough to prove that $(\Cc,\hil)$ is relatively hyperbolic with respect to $\Sc$. To accomplish this we will use Sisto's characterization of relative hyperbolicity stated in Theorem \ref{thm:Sisto_equiv}. 

Recall, from Definition \ref{defn:LS}, that for a properly embedded simplex $S$, $\Lc_S$ is the family of linear projections onto $S$.  For each $S \in \Sc$, choose  a set of $S$-supporting hyperplanes $\Hc_S$ to form a collection of linear projections 
\begin{equation*}
\Pi _{\Sc} :=\Big\{ L_{S,\Hc_S}  : S \in \Sc \Big\}.
\end{equation*}
 Next fix the geodesic path system on $(\Cc,H_\Omega)$ defined by 
$$\Gc:=\{ [x,y] : x,y \in \Cc\}.$$
By Theorem \ref{thm:Sisto_equiv}, it is enough to show that $\Pi_{\Sc}$ is an almost-projection system and $\Sc$ is asymptotically transverse-free relative to $\Gc.$ 

\begin{remark}
In general $\# \Lc_S > 1$, so there is some choice involved in the construction of  $\Pi_{\Sc}$. However, by Proposition \ref{prop:proj-coarsely-equiv} below,
\begin{align*}
\sup_{S \in \Sc} \ \sup_{L_1,L_2 \in \Lc_S} \ \sup_{x \in \Cc} \hil (L_{1}(x),L_{2}(x)) <+\infty.
\end{align*}
So $\Pi_{\Sc}$ will be an almost-projection system, independent of the choices involved in its construction.
\end{remark}

\subsection{$\Pi_{\Sc}$ is an almost-projection system}

\begin{theorem}\label{thm:metric-proj-is-almost-proj-sys}
$\Pi_{\Sc}$ is an almost-projection system for $\Sc$ on the complete geodesic metric space $(\Cc, \hil)$.
\end{theorem}

The proof of Theorem~\ref{thm:metric-proj-is-almost-proj-sys} will require a series of preliminary results. We first prove a continuity lemma for linear projections that will be used repeatedly in this section.

\begin{lemma}\label{lem:continuity-lin-proj-in-C}
If $S \in \Sc$, then the map 
\begin{align*}
(L, x) \in \Lc_S \times \overline{\Cc} \rightarrow L(x) \in \overline{S} 
\end{align*}
is continuous.
\end{lemma}

\begin{proof} We first show that $\Pb(\ker L) \cap \overline{\Cc} = \emptyset$ for all $L \in \Lc_S$. Suppose for a contradiction that $L \in \Lc_S$ and 
\begin{align*}
x \in\Pb(\ker L) \cap \overline{\Cc}.
\end{align*}
Proposition~\ref{prop:W-direct-sum} implies that $x \in \partiali \Cc$. Then Proposition~\ref{prop:W-intersect-boundary} implies that $[y,x] \subset \partiali\Cc$ for every $y \in \partial S$. Next fix $y_1, y_2 \in \partial S$ such that $(y_1, y_2) \subset S$. Then $y_1, x, y_2$ form a half triangle. So  $x \in F_\Omega(\partial S)$ by Corollary~\ref{cor:half_triangles_out_of_maximal_simplices}. But Proposition~\ref{prop:W-intersect-boundary} implies that 
\begin{align*}
F_\Omega(\partial S) \cap \Pb(\ker L) = \emptyset.
\end{align*}
So we have a contradiction. Thus $\Pb(\ker L) \cap \overline{\Cc} = \emptyset$ for all $L \in \Lc_S$. 

Now suppose that $\lim_{n \to \infty} (L_n, x_n)= (L,x)$ in  $\Lc_S \times \overline{\Cc}$. Let $\til{x}_n, \til{x} \in \Rb^d$ denote lifts of $x_n,x$ respectively such that $\lim_{n \to \infty} \til{x}_n = \til{x}$. Then 
\begin{align*}
L(\til{x}) = \lim_{n \rightarrow \infty} L_n( \til{x}_n) \in \Rb^d. 
\end{align*}
Since $\Pb(\ker L) \cap \overline{\Cc} = \emptyset$, we have $L(\til{x}) \neq 0$. So
\begin{equation*}
L(x) = \left[L(\til{x}) \right] =  \lim_{n \rightarrow \infty} \left[ L_n( \til{x}_n) \right] = \lim_{n \rightarrow \infty} L_n(x_n). \qedhere
\end{equation*}

\end{proof}

Next we introduce the ``closest points'' projection onto a properly embedded simplex. 

\begin{definition}\label{defn:closest-point-proj}
If $S \subset \Omega$ is a properly embedded simplex and $p \in \Omega$, the \emph{closest points projection of $p$ onto $S$} is the set
\begin{align*}
\pi_S(p) := S \cap \{ q \in \Omega : H_\Omega(p,q) \leq H_\Omega(p,S)\}. 
\end{align*}
\end{definition}

\begin{observation} \label{obs:pi-S}
Suppose $S \subset \Omega$ is a properly embedded simplex. Then:
\begin{enumerate}
\item If $p \in \Omega$, then $\pi_S(p)$ is compact and convex. 
\item If $g \in \Aut(\Omega)$, then $g \circ \pi_S =\pi_{gS} \circ g$.
\end{enumerate}
\end{observation}
 \begin{proof} Part (2) is obvious and part (1) follows the fact that metric balls in the Hilbert metric are convex. \end{proof}

Now, we establish the coarse equivalence between the two projections.

\begin{proposition}\label{prop:proj-coarsely-equiv}
There exists $\delta_1 \geq 0 $ such that: if $S \in \Sc$, $\Hc$ is a set of $S$-supporting hyperplanes, and $ x \in \Cc$, then 
$$\max_{p\in \pi_S(x)} \hil(L_{S,\Hc}(x),p)\leq \delta_1.$$ 
\end{proposition}
\begin{proof}
Since $\Sc$ has finitely many $\Lambda$ orbits (see Proposition \ref{prop:HK_section_3_1}), it is enough to prove the result for some fixed $S \in \Sc.$ 

Suppose the proposition is false. Then, for every $n \geq 0$, there exist  $x_n \in \Cc$, a set of $S$-supporting hyperplanes $\Hc_n$, and $p_n \in \pi_S(x_n)$ such that  
\begin{equation*}
\hil(p_n, L_{S,\Hc_n}(x_n)) \geq n.
\end{equation*}
Let $m_n$ be the midpoint of the projective line segment $[p_n,L_{S,\Hc_n}(x_n)]$ in the Hilbert distance. Since $\Stab_{\Lambda}(S)$ acts co-compactly on $S$ (see Proposition \ref{prop:HK_section_3_1}), translating by elements of $\Stab_{\Lambda}(S)$ and passing to a subsequence, we  can assume that $m:=\lim_{n \to \infty} m_n$ exists in $S$. Passing to a further subsequence and using Proposition \ref{prop:compactness-linear-proj}, we can assume that there exist $x, p, x' \in \partiali \Cc$ and $L_{S,\Hc} \in \Lc_S$ where $x:=\lim_{n \to \infty} x_n$, $p :=\lim_{n \to \infty}p_n$, $x' := \lim_{n \to \infty}L_{S,\Hc_n}(x_n)$, and $L_{S,\Hc}:=\lim_{n \to \infty}L_{S,\Hc_n}.$  By Lemma \ref{lem:continuity-lin-proj-in-C}, 
\begin{align*}
L_{S,\Hc}(x)=\lim_{n \to \infty} L_{S,\Hc_n}(x_n)=x'.
\end{align*}

We first show that $[x',x] \subset \partiali \Cc.$ Observe that $L_{S,\Hc}(v)=x'$ for all $v \in [x',x]$ since $L_{S,\Hc}$ is linear and $L_{S,\Hc}(x')=x'=L_{S,\Hc}(x).$  But $L_{S,\Hc}(\Omega)=S,$ implying $[x',x] \cap \Omega =\emptyset.$ Hence, $$[x',x] \subset \partiali \Cc.$$

Next we  show that $[p,x] \subset \partiali \Cc.$ Suppose not, then $(p,x) \subset \Cc$. Choose any  $ v \in (p,x) \cap \Cc$ and a sequence $v_n \in [p_n,x_n]$ such that $v=\lim_{n \to \infty} v_n.$ Since $p \in \partiali \Cc$ and $v \in \Cc$, \begin{align*}
\lim_{n \to \infty} \hil(v_n,p_n) = \infty.
\end{align*}
Fix any $v_S \in S.$  Then, choosing $n$ large enough so that $\hil(v_n,p_n) \geq 2 + \hil(v,v_S)$ and $\hil(v,v_n) \leq 1,$ 
\begin{align*}
\hil(x_n, v_S) & \leq \hil(x_n,v_n) + \hil(v_n, v) + \hil(v, v_S) \\
					& = \hil(x_n,p_n)-\hil(p_n,v_n) + \hil(v_n, v) + \hil(v,v_S) \\
					& \leq \hil(x_n, p_n) -1, 
\end{align*}
which is a contradiction since $p_n \in \pi_S(x_n).$  Hence, $[p,x] \subset \partiali \Cc.$

Thus, $[p,x] \cup [x,x'] \subset \partiali \Cc$ and by construction, $m \in (p,x) \subset \Cc$. Thus the three points $x, x', p$ form a half triangle. Then, by Corollary~\ref{cor:half_triangles_out_of_maximal_simplices}, $x \in F_{\Omega}(\bdry S)$. So, by  Proposition \ref{prop:image-lin-proj-in-simplex-bdry}, $x'=L_{S,\Hc}(x) \in F_{\Omega}(x).$ Since $[p,x] \subset \partiali \Cc$,  Observation~\ref{obs:faces} part (4) implies that $(p,x') \subset \partiali \Cc.$ This is a contradiction since 
\begin{equation*}
m \in (p,x') \cap \Cc \neq \emptyset. \qedhere
\end{equation*}
\end{proof}

The next step is to prove $\delta$-thinness of some special triangles built using  linear projections (see Proposition \ref{prop:lin-proj-hyperbolic-triangles}). The following lemma provides a criterion for $\delta$-thinness of triangles in Hilbert geometry.

\begin{lemma}\label{lem:one-side-enough-for-hyperbolic-triangles}
Suppose $\Omega \subset \Pb(\Rb^d)$ is a properly convex domain and $x,y,z \in \Omega$. If 
\begin{align*}
[x,y] \subset \Nc_\Omega( [x,z] \cup [z,y]; R),
\end{align*}
then the geodesic triangle $$[x,y] \cup [y,z] \cup [z,x]$$ is $(2R)$-thin.
\end{lemma}

\begin{proof} The sets
\begin{align*}
I_x = [x,y] \cap \Nc_\Omega( [x,z]; R) \text{ and } I_y = [x,y] \cap \Nc_\Omega( [y,z]; R)
\end{align*}
are non-empty and relatively open in $[x,y]$. Since $[x,y] = I_x \cup I_y$ and $[x,y]$ is connected, there exists $c \in I_x \cap I_y$. Then there exist $c_x \in [x,z]$ and $c_y \in [y,z]$ such that $H_\Omega(c,c_x) < R$ and $H_\Omega(c,c_y) < R$. Then 
\begin{align*}
H_\Omega^{\Haus}( [x,c_x],[x,c]) \leq H_\Omega(x,x) + H_\Omega(c_x, c) < R
\end{align*}
and
\begin{align*}
H_\Omega^{\Haus}( [c_x,z], [c_y,z]) \leq H_\Omega(c_x,c_y) + H_\Omega(z, z) < 2R.
\end{align*}
So
\begin{align*}
[x,z] \subset \Nc_\Omega( [x,y] \cup [y,z]; 2R).
\end{align*}
A similar argument shows that 
\begin{align*}
[y,z] \subset \Nc_\Omega( [z,x] \cup [x,y]; 2R).
\end{align*}
So the geodesic triangle is $(2R)$-thin. 
\end{proof}

\begin{proposition}\label{prop:lin-proj-hyperbolic-triangles}There exists $\delta_2 \geq 0$ such that:  if $x \in \Cc$, $S \in \Sc$, $z \in S$, and $\Hc$ is a set of $S$-supporting hyperplanes, then the geodesic triangle $$\big[ x,z \big] \cup \big[ z,L_{S,\Hc}(x) \big] \cup \big[ L_{S,\Hc}(x),x \big] $$ is $\delta_2$-thin.
\end{proposition}

\begin{proof} 
Since $\Sc$ has finitely many $\Lambda$ orbits (see Proposition \ref{prop:HK_section_3_1}), it is enough to prove the result for some fixed $S \in \Sc.$ By Lemma \ref{lem:one-side-enough-for-hyperbolic-triangles}, it is enough to show that there exists $\delta_2 \geq 0$ such that 
$$[L_{S,\Hc}(x),z] \subset \Nc_{\delta_2/2}([z,x] \cup [x,L_{S,\Hc}(x)])$$
for all $x \in \Cc$, $z \in S$, and $\Hc$ a set of $S$-supporting hyperplanes.

Suppose such a $\delta_2$ does not exist. Then, for every $n \geq 0$, there exist $z_n \in S$, a set of $S$-supporting hyperplanes $\Hc_n$, $p_n := L_{S,\Hc_n}(x_n)$, and $u_n \in [z_n,p_n]$ such that 
\begin{equation*}
\hil(u_n, [z_n,x_n] \cup [x_n,p_n]) \geq n.
\end{equation*}

Since $\Stab_{\Lambda}(S)$ acts co-compactly on $S$, translating by elements of $\Stab_{\Lambda}(S)$ and passing to a subsequence, we can assume that $u:= \lim_{n \to \infty} u_n$ exists and  $u \in S$. Passing to a further subsequence and using Proposition \ref{prop:compactness-linear-proj}, we can assume there exist $x, z, p \in \overline{\Cc}$ and $L_{S,\Hc} \in \Lc_S$ where $ x:= \lim_{n \to \infty} x_n$, $z:= \lim_{n \to \infty} z_n$, $p:= \lim_{n \to \infty} p_n$, and $L_{S,\Hc}:=\lim_{n \to \infty} L_{S,\Hc_n}. $
Since
 \begin{equation*}
 \begin{split}
 \lim_{n \to \infty} &\hil(u, [x_n,z_n]\cup [x_n,p_n]) \\
 &\geq\lim_{n \to \infty}  \Big( \hil(u_n,[x_n,z_n] \cup [x_n,p_n])- \hil(u,u_n) \Big) =\infty, 
 \end{split}
 \end{equation*}
 we have 
 \begin{equation*}
 [x,z] \cup [x,p] \subset \partiali \Cc.
 \end{equation*}  
 By construction, $u \in (p,z) \subset \Cc$. Thus, $p, x, z$ form a half triangle. Then, by Corollary~\ref{cor:half_triangles_out_of_maximal_simplices}, $x \in F_{\Omega}(\bdry S)$. Lemma \ref{lem:continuity-lin-proj-in-C} and Proposition \ref{prop:image-lin-proj-in-simplex-bdry} then imply 
 $$p =\lim_{n \to \infty} p_n= \lim_{n \to \infty} L_{S,\Hc_n}(x_n)=L_{S,\Hc}(x) \in F_{\Omega}(x).$$
Then, since $[x,z] \subset \partiali \Cc$, Observation \ref{obs:faces} part (4) implies that $[p,z] \subset \partiali \Cc.$ This is a contradiction, since 
 \begin{equation*}
 u \in (p,z) \cap \Cc  \neq \emptyset. \qedhere
 \end{equation*}
\end{proof}

\begin{proposition}\label{prop:projection-close-to-hypotenuse-hyperbolic-triangle} Set $\delta_3:=\delta_1 + 3 \delta_2$. If $x \in \Cc$, $S \in \Sc$, $\Hc$ is a set of $S$-supporting hyperplanes, and $z \in S$, then $\hil\big( L_{S,\Hc}(x),[x,z] \big) \leq \delta_3.$ 
\end{proposition}
\begin{proof}
By Proposition \ref{prop:lin-proj-hyperbolic-triangles},  the geodesic triangle 
\begin{align*}
[x,z] \cup [z,L_{S,\Hc}(x)] \cup [L_{S,\Hc}(x),x]
\end{align*}
 is $\delta_2$-thin. Thus, there exist $y \in [L_{S,\Hc}(x),z]$, $y_1 \in [x,L_{S,\Hc}(x)]$, and $y_2 \in [x,z]$ such that $\hil(y,y_1) \leq \delta_2$ and $\hil(y,y_2) \leq \delta_2.$

We claim that $\hil( L_{S,\Hc}(x),y_1) \leq \delta_1 + \delta_2.$ Choose any $p \in \pi_S(x).$ Since $[L_{S,\Hc}(x),z] \subset S$,  
\begin{align*}
\hil(x,p) =\hil(x,S) \leq \hil(x,y).
\end{align*} 
 Then, using Proposition \ref{prop:proj-coarsely-equiv}, 
\begin{align*}
\hil(x, L_{S,\Hc}(x)) \leq \hil(x,p)+\hil(p,L_{S,\Hc}(x))  \leq \hil(x,y)+ \delta_1.
\end{align*}
Then, 
\begin{align*}
\hil(L_{S,\Hc}(x),y_1) &= \hil(L_{S,\Hc}(x),x)-\hil(y_1,x) \\
								& \leq \hil(x,y) +\delta_1 -\hil(y_1,x) \\
								& \leq \hil(y,y_1) + \delta_1 \leq \delta_2 + \delta_1. 
\end{align*}
Hence, 
\begin{align*}
\hil(L_{S,\Hc}(x),[x,z]) &\leq \hil(L_{S,\Hc}(x),y_2) \\
									& \leq \hil(L_{S,\Hc}(x),y_1) + \hil(y_1,y) + \hil(y,y_2)\\
									& \leq \delta_1+3\delta_2=\delta_3. \qedhere
\end{align*}
\end{proof}

Our next goal is to prove if the distance between the linear projections of two points onto a simplex $S \in \Sc$ is large, then the geodesic between the two points spends a significant amount of  time in a tubular neighborhood of $S$. This is accomplished in Corollary \ref{cor:penetration-of-simplex-nbd} using the next result.

\begin{proposition}\label{prop:projection-to-simplices-contracting}
There exists $\delta_4 \geq 0$ such that: if $S \in \Sc$, $\Hc$ is a set of $S$-supporting hyperplanes, $x, y \in \Cc$, and $\hil(L_{S,\Hc}(x),L_{S,\Hc}(y)) \geq \delta_4,$ then 
\begin{align*}\hil(L_{S,\Hc}(x),[x,y]) \leq \delta_4 ~~\text{ and }~~ \hil(L_{S,\Hc}(y),[x,y]) \leq \delta_4.
\end{align*} 
\end{proposition}
\begin{proof} Observe that  the linear projections are $\Lambda$-equivariant, that is, 
$$L_{gS,g\Hc} \circ g = g \circ L_{S,\Hc}$$ 
for any $g \in \Lambda$, $S \in \Sc$, and $\Hc$ a set of $S$-supporting hyperplanes. Moreover, by  Proposition \ref{prop:HK_section_3_1} there are only finitely many $\Lambda$-orbits in $\Sc$. Thus, it is enough to prove this proposition for a fixed $S \in \Sc.$

Suppose the proposition is false. Then, for every $n \geq 0$, there exist $x_n, y_n \in \Cc$ and a set of $S$-supporting hyperplanes $\Hc_n$ with 
$$\hil(L_{S,\Hc_n}(x_n),L_{S,\Hc_n}(y_n)) \geq n$$ 
and 
$$\hil(L_{S,\Hc_n}(x_n),[x_n,y_n]) \geq n.$$
Let $a_n:=L_{S,\Hc_n}(x_n)$ and $b_n:=L_{S,\Hc_n}(y_n)$. Then pick $c_n \in [a_n,b_n]$ such that 
\begin{equation}\label{eqn:choose-un}
\hil(c_n,a_n) =n/2.
\end{equation}
Then, 
\begin{equation}\label{eqn:dist-un-pi-yn}
\hil(c_n,b_n) \geq \hil(a_n, b_n) -\hil(c_n, a_n) \geq n/2
\end{equation}
and 
\begin{equation}\label{eqn:dist-un-[xn,yn]}
\hil \Big(c_n,[x_n,y_n] \Big) \geq \hil \Big(a_n,[x_n,y_n] \Big)-\hil(c_n,a_n) \geq n/2.
\end{equation}

Since $\Stab_{\LG}(S)$ acts co-compactly on $S$ (see Proposition \ref{prop:HK_section_3_1}), translating by elements of $\Stab_{\Lambda}(S)$ and passing to a subsequence, we may assume that $c:=\lim_{n \to \infty} c_n$ exists and $c \in S.$  After taking a further subsequence, we can assume that the following limits exist in $\overline{\Cc}$: $a:=\lim_{n \to \infty} a_n$, $b:=\lim_{n \to \infty} b_n$, $x:=\lim_{n \to \infty} x_n$ and $y:=\lim_{n \to \infty} y_n$. 

We now observe that $a, b, x, y \in \partiali \Cc.$ Equation \eqref{eqn:choose-un} and~\eqref{eqn:dist-un-pi-yn} imply that $a, b \in \partiali \Cc.$ Equation \eqref{eqn:dist-un-[xn,yn]} implies that $[x,y] \subset \partiali \Cc.$ 

We claim that $x \in F_{\Omega}(a)$ and $y \in F_{\Omega}(b).$ Since $c_n \in S$, by Proposition \ref{prop:projection-close-to-hypotenuse-hyperbolic-triangle}, there exists $a_n' \in [x_n,c_n]$ such that $\hil(a_n,a_n') \leq \delta_3.$ Up to passing to a subsequence, we can assume that $a':=\lim_{n\to \infty} a_n'$ exists in $\overline{\Cc}$. Observe that $a'\in \partiali \Cc$ since  $$\lim_{n \to \infty} \hil(a_n',c) \geq\lim_{n \to \infty} \Big( \hil(a_n,c_n)- \hil(c_n,c) - \hil(a_n,a_n') \Big) =\infty.$$ Since $a_n' \in [x_n,c_n]$, $$a' \in \partiali \Cc \cap [x,c]=\{ x\}.$$ Thus, $\lim_{n \to \infty}a_n'=x$. Since $\lim_{n \to \infty}a_n=a$ and $\hil(a_n,a_n') \leq \delta_3$, Proposition \ref{prop:dist_est_and_faces} implies that $x \in F_{\Omega}(a).$  Similar reasoning shows that $y \in F_{\Omega}(b).$ 

Since $[x,y] \subset \partiali \Cc$, Observation~\ref{obs:faces} part (4) implies that $[a,b] \subset \partiali \Cc.$ This is a contradiction since $ c \in (a,b) \cap \Cc \neq \emptyset.$ 
\end{proof}

\begin{corollary}\label{cor:penetration-of-simplex-nbd}
If $S \in \Sc$, $\Hc$ is a set of $S$-supporting hyperplanes, $R >0$, $x, y \in \Cc$, and $\hil(L_{S,\Hc}(x), L_{S,\Hc}(y)) \geq R + 2 \delta_4$, then:
\begin{enumerate}
\item there exists $[x_0,y_0]\subset [x,y]$ such that $[x_0,y_0] \subset \Nc_{\Omega}(S;\delta_4),$
\item $\big[ L_{S,\Hc}(x),L_{S,\Hc}(y) \big] \subset \Nc_{\Omega} \big( [x,y]; \delta_4 \big),$ and,
\item $\diam_{\Omega} \Big( \Nc_{\Omega}(S;\delta_4) \cap [x,y] \Big) \geq R.$
\end{enumerate} 
\end{corollary}

\begin{proof}
Since $\hil(L_{S,\Hc}(x), L_{S,\Hc}(y)) >  \delta_4$, Proposition \ref{prop:projection-to-simplices-contracting} implies that there exist $x_0 , y_0 \in [x,y]$ such that 
\begin{equation*}
\hil(L_{S,\Hc}(x),x_0) \leq \delta_4 ~~\text{ and  }~~ \hil(L_{S,\Hc}(y),y_0) \leq \delta_4.
\end{equation*}
By Proposition~\ref{prop:Crampons_dist_est_2}, 
\begin{align*}
\hil^{\Haus} \Big( \big[ x_0,y_0 \big], \big[ L_{S,\Hc}(x),L_{S,\Hc}(y) \big] \Big) \leq  \delta_4
\end{align*}
and, by convexity, $[L_{S,\Hc}(x),L_{S,\Hc}(y)] \subset S.$ This proves parts (1) and (2). To prove part (3), observe that 
\begin{align*}
\hil(x_0,y_0) & \geq \hil \big( L_{S,\Hc}(x), L_{S,\Hc}(y) \big) -\hil(L_{S,\Hc}(x),x_0)-\hil(L_{S,\Hc}(y),y_0) \\
& \geq R. \nonumber
\end{align*}
Then, $\diam_{\Omega} \Big( \Nc_{\Omega}(S; \delta_4) \cap [x,y] \Big) \geq \hil(x_0,y_0) \geq R.$
\end{proof}

Using the properties of linear projections established so far, we prove that $\Pi_{\Sc}$ is an almost-projection system.

\begin{lemma}\label{lem:almost-proj-system-1}
If $S \in \Sc$, $\Hc$ a set of $S$-supporting hyperplanes, $x \in \Cc$, and $z \in S$, then $$\hil(x,z) \geq \hil(x,L_{S,\Hc}(x)) + \hil( L_{S,\Hc}(x),z)- 2 \delta_3.$$ 
\end{lemma}
\begin{proof}
By Proposition \ref{prop:projection-close-to-hypotenuse-hyperbolic-triangle}, there exists $q \in [x,z]$ such that $\hil(L_{S,\Hc}(x),q) \leq  \delta_3.$ Then, 
\begin{equation*}
\hil(x,z) = \hil(x,q)+\hil(q,z) \geq \hil(x,L_{S,\Hc}(x)) + \hil(L_{S,\Hc}(x),z) - 2 \delta_3. \qedhere
\end{equation*}
\end{proof}

\begin{lemma}\label{lem:almost-proj-system-2}
There exists $\delta_5 \geq 0$ such that: if $S \neq S' \in \Sc$ and $\Hc$ is a set of $S$-supporting hyperplanes, then $$\diam_{\Omega} (L_{S, \Hc}(S')) \leq \delta_5.$$
\end{lemma}

\begin{proof}
Since $\Sc$ is strongly isolated, for every $r>0$ there exists $D(r)>0$ such that
\begin{equation}
\label{eqn:almost-proj-2-contradict}
\diam_{\Omega}\Big( \Nc_\Omega\big(S_1;r \big) \cap \Nc_\Omega \big(S_2,r \big) \Big) \leq D(r)
\end{equation}
for all $S_1, S_2 \in \Sc$ distinct.

Let $\delta_5:= D \big( \delta_4 \big)+ 2 \delta_4 +1$. Fix $x, y \in S'$ and suppose for a contradiction that  $\hil(L_{S,\Hc}(x),L_{S,\Hc}(y)) > \delta_5.$  Then, by Corollary \ref{cor:penetration-of-simplex-nbd}, 
\begin{align*}
\diam_{\Omega}\Big(  \Nc_{\Omega}(S; \delta_4) \cap S^\prime \Big) \geq \diam_{\Omega} \Big( \Nc_{\Omega}(S; \delta_4) \cap [x,y] \Big) \geq D(\delta_4)+1.
\end{align*}
which contradicts Equation \eqref{eqn:almost-proj-2-contradict}.
\end{proof}

\begin{lemma}\label{lem:almost-proj-system-3}
If $x \in \Cc$, $S \in \Sc$, $\Hc$ is a set of $S$-supporting hyperplanes,  and $R:= \hil(x,S)$, then $$\diam_{\Omega}\bigg( L_{S,\Hc} \Big( B_{\Omega}(x;R) \cap \Cc \Big) \bigg) \leq 8 (\delta_4+\delta_1).$$ 
\end{lemma}

\begin{proof} Fix $y \in B_{\Omega}(x;R) \cap \Cc$. We claim that 
\begin{align*}
\hil(L_{S,\Hc}(x),L_{S,\Hc}(y)) \leq 4 (\delta_4+ \delta_1).
\end{align*} 
It is enough to consider the case when $\hil(L_{S,\Hc}(x),L_{S,\Hc}(y)) \geq \delta_4.$ Then by Proposition \ref{prop:projection-to-simplices-contracting}, there exists $x' \in [x,y]$ such that $\hil(L_{S,\Hc}(x),x') \leq \delta_4.$  By Proposition  \ref{prop:proj-coarsely-equiv}, 
$$\hil(x,y)\leq R=\hil(x,\pi_S(x)) \leq \hil(x,L_{S,\Hc}(x)) + \delta_1.$$   
Then, 
\begin{align*}
\hil(x',y) &=\hil(x,y)-\hil(x,x') \leq \hil(x,L_{S,\Hc}(x))-\hil(x,x') +\delta_1 \\ &\leq \hil(L_{S,\Hc}(x),x')+\delta_1 \\
&\leq \delta_4 + \delta_1.
\end{align*}
Thus, $$\hil(L_{S,\Hc}(x),y) \leq \hil(L_{S,\Hc}(x),x')+\hil(x',y) \leq 2 \delta_4+\delta_1.$$
Since $L_{S,\Hc}(x) \in S$, using Proposition \ref{prop:proj-coarsely-equiv} again,
\begin{align*}
\hil(y,L_{S,\Hc}(y)) & \leq \hil(y,\pi_S(y))+\delta_1  \leq \hil(y,L_{S,\Hc}(x)) + \delta_1 \\
							  & \leq 2 (\delta_4+ \delta_1).
\end{align*} 
Finally
\begin{align*}
\hil(L_{S,\Hc}(x),L_{S,\Hc}(y)) &\leq \hil( L_{S,\Hc}(x),x') + \hil(x',y)+ \hil(y,L_{S,\Hc}(y)) \\
&\leq 4 (\delta_4+\delta_1). \qedhere
\end{align*}
\end{proof}

\begin{proof}[Proof of Theorem \ref{thm:metric-proj-is-almost-proj-sys}]
Follows from Lemma \ref{lem:almost-proj-system-1}, Lemma \ref{lem:almost-proj-system-2}, and Lemma \ref{lem:almost-proj-system-3}.
\end{proof}

\subsection{$\Sc$ is asymptotically transverse-free relative to $\Gc$}
\begin{theorem}\label{thm:asympt-trans-free}
The family $\Sc$ is asymptotically transverse-free relative to the geodesic path system $\Gc$.
\end{theorem}

\begin{proof}
\label{proof:proof-of-asympt-trans-free}

Let $\delta_4$ be the constant in Proposition \ref{prop:projection-to-simplices-contracting}.  We will show that exists $\lambda>0$ such that for each $\triangleD \geq 1$ and $\kappa \geq 2 \delta_4$ the following holds: if $\Tc \subset \Cc$ is a geodesic triangle whose sides are in $\Gc$ and is $\Sc$-almost-transverse with constants $\kappa$ and $\triangleD$, then $\Tc$ is $(\lambda \triangleD)$-thin. 

Suppose such a $\lambda > 0$ does not exist. Then, for every $n \geq 1$, there exist $\kappa_n \geq2 \delta_4$, $\triangleD_n \geq 1$, and a $\Sc$-almost-transverse triangle $\Tc_n \subset \Cc$ with constants $\kappa_n$ and  $\triangleD_n$ such that $\Tc_n$ is not $(n \triangleD_n)$-thin. Let $a_n$, $b_n$, and $c_n$ be the vertices of $\Tc_n$, labeled in a such a way that there exists $u_n \in [a_n,b_n] \subset  \Tc_n $ with 
\begin{equation}\label{eqn:contradiction:transverse-Tn}
\hil \Big( u_n,[a_n,c_n] \cup [c_n,b_n] \Big) > n \triangleD_n \geq n. 
\end{equation} 
Note that Observation \ref{obs:almost-transverse} implies that the geodesic triangles $\Tc_n$ are also $\Sc$-almost-transverse with constants $2 \delta_4$ and  $\triangleD_n$ since $\kappa_n \geq 2\delta_4.$ \label{correction:the-simple-obs-is-used-here}

Since $\Lambda$ acts co-compactly on $\Cc$, translating by elements of $\Lambda$ and passing to a subsequence, we can assume that $u:=\lim_{n \to \infty}u_n$ exists and $u \in \Cc$. By passing to a further subsequence, we can  assume that $a:=\lim_{n \to \infty}a_n$, $b:=\lim_{n \to \infty}b_n$, and $c:=\lim_{n \to \infty}c_n$ exist in $\overline{\Cc}.$ By Equation \eqref{eqn:contradiction:transverse-Tn}, $$[a,c] \cup [c,b] \subset \partiali \Cc$$ whereas, by construction, $u \in (a,b) \subset \Cc.$ Thus, the points $a, b, c$ form a half triangle. Then, by Theorem~\ref{thm:half_triangle_nearby_simplex}, there exists $S \in \Sc$ such that $a, b, c \in F_{\Omega}(\bdry S).$

Fix a set of $S$-supporting hyperplanes $\Hc.$ Let $a_n':=L_{S,\Hc}(a_n)$, $b_n':=L_{S,\Hc}(b_n)$,  and  $c_n':=L_{S,\Hc}(c_n)$. Up to passing to a subsequence, we can assume that the limits $a':=\lim_{n \to \infty}a_n'$, $b':=\lim_{n \to \infty}b_n'$ and $c':=\lim_{n \to \infty}c_n'$ exist.  By  Lemma \ref{lem:continuity-lin-proj-in-C} and Proposition~\ref{prop:image-lin-proj-in-simplex-bdry}
$$a'=\lim_{n \to \infty} L_{S,\Hc}(a_n)=L_{S,\Hc}(a) \in F_{\Omega}(a).$$
Similarly, 
\begin{equation*}
b'=L_{S,\Hc}(b) \in F_{\Omega}(b)~~ \text{ and }~~ c'=L_{S,\Hc}(c) \in F_{\Omega}(c).
\end{equation*}
 
Using Observation \ref{obs:faces} part (4), $(a',b') \subset \Omega$ and $[a',c'] \cup [b',c'] \subset \partiali \Cc$. Then, Observation \ref{obs:faces} part (4) implies that the faces $F_\Omega(a^\prime)$, $F_\Omega(b^\prime)$, and $F_\Omega(c^\prime)$, are pairwise disjoint. Then, by Proposition \ref{prop:dist_est_and_faces},
\begin{equation*}
\lim_{n \to \infty} \hil(a_n',b_n')=\infty.
\end{equation*}
Thus, for $n$ large enough, Corollary \ref{cor:penetration-of-simplex-nbd} part (2) and part (3) implies
\begin{equation}\label{eqn:an-bn-nbd}
[a_n',b_n'] \subset \Nc_{\Omega}([a_n,b_n]; \delta_4)
\end{equation}
 and 
 \begin{equation}\label{eqn:an-bn-diam}
 \diam_{\Omega} \Big(\Nc_{\Omega}(S; \delta_4) \cap [a_n,b_n] \Big) \geq \hil(a_n',b_n')-2 \delta_4.
 \end{equation}
 Since $\Tc_n$ is $\Sc$-almost-transverse with constants $2 \delta_4$ and $\triangleD_n$, by Equation~\eqref{eqn:an-bn-diam}, 
 \begin{equation}\label{eqn:an-bn-dist-est}
 \hil(a_n',b_n') \leq \triangleD_n + 2\delta_4.
 \end{equation} 

\noindent Similarly,  for $n$ large enough, 
\begin{align}
&[b_n',c_n'] \subset \Nc_{\Omega}([b_n,c_n]; \delta_4) ~~\text{ and }~~ \hil(b_n',c_n') \leq \triangleD_n+2 \delta_4 \label{eqn:bn-cn}\\
&[c_n',a_n'] \subset \Nc_{\Omega}([c_n,a_n]; \delta_4) ~~\text{ and }~~ \hil(c_n', a_n') \leq \triangleD_n+2 \delta_4. \label{eqn:cn-an}
\end{align}

Let $m_n^{ab}$, $m_n^{bc}$, and $m_n^{ca}$ be the Hilbert distance midpoints of $[a'_n,b'_n]$, $[b'_n,c'_n]$,  and $[c'_n,a'_n]$ respectively. By Equations \eqref{eqn:an-bn-nbd}, \eqref{eqn:bn-cn}, and \eqref{eqn:cn-an}, there exist $w^{ab}_n$, $w^{bc}_n$, and $w^{ca}_n$ in $[a_n,b_n]$, $[b_n,c_n]$, and $[c_n,a_n]$ respectively such that: 
$$\hil(w^{ab}_n, m^{ab}_n) \leq \delta_4, ~\hil(w^{bc}_n, m^{bc}_n) \leq \delta_4, ~\text{ and }~\hil(w^{ca}_n, m^{ca}_n) \leq \delta_4.$$
Then, 
\begin{align*}
\hil(w_n^{ab},w_n^{bc}) & \leq \hil(w_n^{ab},m_n^{ab}) + \hil(m_n^{ab},m_n^{bc}) + \hil(m_n^{bc},w_n^{bc})  \\
&\leq \delta_4 + \hil(m_n^{ab},b_n') + \hil(b_n',m_n^{bc}) + \delta_4 \\
& = 2 \delta_4 + \dfrac{\hil(a_n',b_n')+ \hil(b_n',c_n')}{2} \\
& \leq 4 \delta_4 + \triangleD_n ~~\text{ (by Equations } \eqref{eqn:an-bn-dist-est} \text{ and } \eqref{eqn:bn-cn} \text{)}.
\end{align*}

\noindent Similarly,  
\begin{equation}
\hil(w_n^{bc},w_n^{ca}) \leq \triangleD_n + 4 \delta_4~~ \text{ and }~~ \hil(w_n^{ca},w_n^{ab}) \leq \triangleD_n + 4 \delta_4.
\end{equation}
 
\noindent Then, for $n$ large enough, the triangles $\Tc_n$ are $( \triangleD_n + 4 \delta_4)$-thin, since 
\begin{align*}
H^{\Haus}_{\Omega} &\Big( [a_n,w_n^{ab}],[a_n,w_n^{ca}] \Big) \leq \triangleD_n + 4 \delta_4, \\
H^{\Haus}_{\Omega}&\Big( [b_n,w_n^{bc}],[b_n,w_n^{ab}] \Big) \leq \triangleD_n + 4 \delta_4, \text{ and}\\
H^{\Haus}_{\Omega}& \Big( [c_n,w_n^{ca}],[c_n,w_n^{bc}] \Big) \leq \triangleD_n + 4 \delta_4.
\end{align*}
Since $\triangleD_n \geq 1$, we have $\triangleD_n + 4 \delta_4 \leq (1+ 4 \delta_4) \triangleD_n.$ Thus, for $n$ large enough, $\Tc_n$ is $(\lambda \triangleD_n)$-thin for $\lambda :=1+4 \delta_4$, which contradicts the assumption that $\Tc_n$ is not $(n \triangleD_n)$-thin.\end{proof}

\section{Proof of Theorem~\ref{thm:main_ncc}}\label{sec:pf_of_thm_main_ncc}

For the rest of the section suppose that $(\Omega, \Cc, \Lambda)$ is a naive convex co-compact triple. 

\subsection{(1) implies (2) and (3)} Theorem \ref{thm:isolated-simplices-imply-rel-hyp}.

\subsection{(3) implies (2)} Suppose that $\Lambda$ is a relatively hyperbolic group with respect to a collection of subgroups $\{H_1,\dots, H_k\}$ each of which is virtually Abelian of rank at least two. For each $1 \leq j \leq k$, let $A_j \leq H_j$ be a finite index Abelian subgroup with rank at least two. Then, by definition, $\Lambda$ is a relatively hyperbolic group with respect to  $\{A_1,\dots, A_k\}$. 

Fix a word metric $\dist_{\Lambda}$ on $\Lambda$. Then for $U \subset \Lambda$ and $r > 0$ define 
\begin{align*}
\Nc_{\Lambda}(U;r) :=\{ g \in \Lambda : \dist_{\Lambda}(g,U) < r\}
\end{align*}
and 
\begin{align*}
\diam_{\Lambda}(U) = \sup\{ \dist_{\Lambda}(g_1, g_2) : g_1, g_2 \in U\}.
\end{align*}

Next, for each $1 \leq j \leq k$, let $\wh{A}_j$ be a maximal Abelian subgroup of $\Lambda$ that contains $A_j$.  By Theorem~\ref{thm:max_abelian}, there exists a properly embedded simplex $S_j \subset \Cc$ such that $\wh{A}_j \leq \Stab_{\Lambda}(S_j)$, $\wh{A}_j$ acts co-compactly on $S_j$, and $\wh{A}_j$ has a finite index subgroup isomorphic to $\Zb^{\dim(S_j)}$.  Since $A_j$ (and hence $\wh{A}_j$) has rank at least two, this implies that $\dim S_j \geq 2$.

We claim that $A_j \leq \wh{A}_j$ has finite index and hence $A_j$ also acts co-compactly on $S_j$. By Observation~\ref{obs:QI_to_Rk}, the metric space $(S_j,H_\Omega)$ is quasi-isometric to $\Rb^{\dim S_j}$. So, by the fundamental lemma of geometric group theory \cite[Chapter I, Proposition 8.19]{BH2013}, $(\wh{A}_j, \dist_{\Lambda})$ is also quasi-isometric to $\Rb^{\dim S_j}$. Since $\dim S_j \geq 2$, Theorem~\ref{thm:rh_embeddings_of_flats} implies that there exist $r_1 > 0$, $g_j \in \Lambda$, and $1 \leq i_j \leq k$ such that 
\begin{align*}
\wh{A}_j \subset \Nc_{\Lambda}(g_j A_{i_j}; r_1).
\end{align*}
Then
\begin{align*}
\diam_{\Lambda} \left( \Nc_{\Lambda}(g_j A_{i_j}; r_1) \cap \Nc_{\Lambda}(A_j; r_1)\right) \geq\diam_{\Lambda} \left( A_j \right) = \infty.
\end{align*}
So Theorem~\ref{thm:rh_intersections_of_neighborhoods} implies that $g_j A_{i_j} = A_j$. Then,
\begin{align*}
\wh{A}_j \subset   \Nc_{\Lambda}(A_{j}; r_1)
\end{align*}
and hence $A_j \leq \wh{A}_j$ has finite index. 

Fix some $x_0 \in \Cc$ and consider the orbit map
\begin{align*}
F& :(\Lambda, \dist_{\Lambda}) \rightarrow (\Cc, H_\Omega)\\
F&(g) = g x_0.
\end{align*}
By the fundamental lemma of geometric group theory \cite[Proposition 8.19]{BH2013} this is a quasi-isometry. Let $G : \Cc \rightarrow \Lambda$ be a quasi-inverse. Then, using the fact that $A_j$ acts co-compactly on $S_j$, there exists $r_2 > 0$ such that 
\begin{align*}
F(gA_j) \subset \Nc_\Omega(gS_j; r_2)
\end{align*}
and 
\begin{align*}
G(gS_j) \subset \Nc_\Lambda(gA_j; r_2)
\end{align*}
for all $g \in \Lambda$ and $1 \leq j \leq k$. Then, by definition and Theorem \ref{thm:rh_quasi_isometry_inv}, $(\Cc,H_\Omega)$ is relatively hyperbolic with respect to the family of properly embedded simplices of dimension at least two
\begin{align*}
\Sc : = \{ g S_j : g \in \Lambda, \ 1 \leq j \leq k\}.
\end{align*}

\subsection{(2) implies (1)} Suppose that $(\Cc, H_\Omega)$ is a relatively hyperbolic space with respect to a family $\Sc_0$ of properly embedded simplices in $\Cc$ of dimension at least two. It is fairly easy to show that $\Sc_0$ is isolated and coarsely complete, but  we will have to modify $\Sc_0$ to construct a $\Lambda$-invariant family.

 By Theorem~\ref{thm:rh_intersections_of_neighborhoods} for any $r > 0$ there exists $Q_1(r) > 0$ such that:
\begin{align*}
\diam_\Omega( \Nc_\Omega(S_1;r) \cap  \Nc_\Omega(S_2;r) ) \leq Q_1(r)
\end{align*}
when $S_1, S_2 \in \Sc_0$ are distinct. Further, by Observation~\ref{obs:QI_to_Rk} and Theorem~\ref{thm:rh_embeddings_of_flats} there exists $Q_2 > 0$ such that: if $S \subset \Cc$ is a properly embedded simplex of dimension at least two, then there exists $S^\prime \in \Sc_0$ such that 
\begin{align}
\label{eq:inclsuion_pf_of_2_implies}
S \subset \Nc_\Omega(S^\prime; Q_2).
\end{align}

\begin{lemma}\label{lem:rh_simplices_faces_of_vertices} If $S \in \Sc_0$ and $v \in \partial S$ is a vertex, then 
\begin{align*}
H_{F_\Omega(v)}^{\Haus}\left( \{v\}, \overline{\Cc} \cap F_\Omega(v)\right)\leq Q_2.
\end{align*}
\end{lemma}

\begin{proof}
Suppose $v,  v_2, \dots, v_p$ are the vertices of $S$. If $w \in \overline{\Cc} \cap F_\Omega(v)$, then by Lemma~\ref{lem:slide_along_faces}
\begin{align*}
\wt{S} := \Omega \cap \Pb(\Spanset\{  w,v_2,\dots, v_p\})
\end{align*}
is a properly embedded simplex in $\Cc$ with 
\begin{align*}
H_\Omega^{\Haus}\left(S, \wt{S}\right) \leq H_{F_\Omega(v)}(v, w). 
\end{align*}
Then there exists $S^\prime \in \Sc_0$ such that 
\begin{align*}
\wt{S} \subset \Nc_\Omega(S^\prime; Q_2).
\end{align*}
Then when $r > Q_2+H_{F_\Omega(v)}(v, w)$
\begin{align*}
\diam_\Omega( \Nc_\Omega(S;r) \cap  \Nc_\Omega(S^\prime;r) ) \geq \diam_\Omega \left(\wt{S}\right)=\infty. 
\end{align*}
So $S=S^\prime$ and 
\begin{align*}
\wt{S} \subset \Nc_\Omega(S; Q_2).
\end{align*}
Then Proposition~\ref{prop:dist_est_and_faces} implies that there exists $v^\prime \in \overline{S} \cap F_\Omega(v)$ with 
\begin{align*}
H_{F_\Omega(v)}(v^\prime, w) \leq Q_2.
\end{align*}
But by Observation~\ref{obs:faces_of_simplices_are_properly_embedded}
\begin{align*}
\{v\} = F_S(v) = \overline{S} \cap F_\Omega(v)
\end{align*}
and so $v = v^\prime$. So 
\begin{align*}
H_{F_\Omega(v)}(v, w) \leq Q_2.
\end{align*}
Finally since $w \in \overline{\Cc} \cap F_\Omega(v)$ was arbitrary this proves the lemma. 
\end{proof}

Now we repeat part of the argument from Section~\ref{sec:intersection_of_nbhds}. In particular, for each simplex $S \in \Sc_0$ we construct a new simplex $\Phi(S)$ as follows. Let $v_1,\dots, v_p$ be the vertices of $S$. Then, by the previous lemma,
\begin{align*}
H_{F_\Omega(v_j)}^{\Haus}\left(\overline{\Cc} \cap F_\Omega(v_j), \{ v_j\} \right) \leq Q_2.
\end{align*}
So $\overline{\Cc} \cap F_\Omega(v_j)$ is a compact subset of $F_\Omega(v_j)$. Then let 
\begin{align*}
w_j: = {\rm CoM}_{F_\Omega(v_j)}\left( \overline{\Cc} \cap F_\Omega(v_j) \right). 
\end{align*}
Then Lemma~\ref{lem:slide_along_faces} implies that 
\begin{align*}
\Phi(S):=\Omega \cap \Pb(\Spanset\{  w_1,\dots, w_p\})
\end{align*}
is a properly embedded simplex with vertices $w_1,\dots, w_p$ and
\begin{align}
\label{eq:phi_is_bd_in_haus_2}
H_{\Omega}^{\Haus}( S, \Phi(S)) \leq Q_2. 
\end{align}

Then define
\begin{align*}
\Sc :=\{ \gamma \Phi(S) : \gamma \in \Lambda, S \in \Sc_0\}.
\end{align*}
We will show that $\Sc$ is isolated, coarsely complete, and $\Lambda$-invariant, but first a preliminary lemma.

\begin{lemma}\label{lem:phi_equals_in_2_implies} If $S \in \Sc$, $S^\prime \in \Sc_0$, and 
\begin{align*}
S \subset \Nc_{\Omega}(S^\prime;r)
\end{align*}
for some $r > 0$, then $\Phi(S^\prime)=S$. 
\end{lemma}

\begin{proof} Let $v_1,\dots, v_p$ be the vertices of $S$. Then by Proposition~\ref{prop:dist_est_and_faces} there exist $v_1^\prime,\dots, v_p^\prime \in \partial S^\prime$ such that 
\begin{align*}
v_j^\prime \in F_\Omega(v_j)
\end{align*}
for all $1 \leq j \leq p$. Lemma~\ref{lem:rh_simplices_faces_of_vertices} and the definition of $\Phi$ implies that $\overline{\Cc} \cap F_\Omega(v_j)$ is a compact neighbourhood of $\{v_j\}$ in $F_\Omega(v_j)$. Since 
\begin{align*}
F_{S^\prime}(v_j^\prime) \subset \overline{\Cc} \cap F_\Omega(v_j),
\end{align*}
 Obervation~\ref{obs:faces_of_simplices_are_properly_embedded} implies that $v_j^\prime$ is a vertex of $S^\prime$. Further, by Lemma~\ref{lem:slide_along_faces}, 
 \begin{align*}
 {\rm ConvHull}_\Omega\{ v_1^\prime, \dots, v_p^\prime\}
\end{align*}
intersects $\Omega$. Since $S'$ is a properly embedded simplex, $v_1^\prime, \dots, v_p^\prime$ must be all of the vertices of $S^\prime$. Then by definition $\Phi(S^\prime)=S$. 
\end{proof}

\begin{lemma} $\Sc$ is an isolated, coarsely complete, and $\Lambda$-invariant family of maximal properly embedded simplices in $\Cc$  of dimension at least two. Hence $(\Omega, \Cc, \Lambda)$ has coarsely isolated simplices. \end{lemma}

\begin{proof} By construction $\Sc$ is $\Lambda$-invariant. 

We next argue that $\Sc$ is isolated. Suppose $S_n \in \Sc$ converges to a closed set $S$ in the local Hausdorff topology. Then $S$ is a properly embedded simplex by Observation~\ref{obs:PES_closed}. For each $n$ there exists $S_n^\prime \in \Sc_0$ such that 
\begin{align*}
S_n \subset \Nc_\Omega(S_n^\prime;Q_2).
\end{align*}
Since $\lim_{n \to \infty} S_n= S$ we have 
\begin{align*}
\infty = \lim_{n \rightarrow \infty} \diam_\Omega \Big( \Nc_\Omega(S_n^\prime;Q_2+1) \cap \Nc_\Omega(S_{n+1}^\prime;Q_2+1)\Big).
\end{align*}
So there exists $N \geq 0$ such that $S_n^\prime = S_N^\prime$ for all $n \geq N$. Then by Lemma~\ref{lem:phi_equals_in_2_implies}
\begin{align*}
S_n = \Phi(S_n^\prime) = \Phi(S_N^\prime)
\end{align*}
for $n \geq N$. So $S = S_N$. Since $S_n \in \Sc$ was an arbitrary convergent subsequence, the set $\Sc$ is closed and discrete in the local Hausdorff topology, hence isolated.

Finally, we show that $\Sc$ is coarsely complete. Since $\Sc_0$ is coarsely complete, if $S \subset \Cc$ is a properly embedded simplex of dimension at least two, then there exists $S^\prime \in \Sc_0$ such that 
\begin{align*}
S \subset \Nc_\Omega(S^\prime;Q_2).
\end{align*}
Then by Equation~\eqref{eq:phi_is_bd_in_haus_2} 
\begin{align*}
S \subset \Nc_\Omega(S^{\prime\prime};2Q_2)
\end{align*}
where $S^{\prime\prime}:=\Phi(S^\prime) \in \Sc$. 
\end{proof}

\part{The convex co-compact case}

\section{Lines and corners in the boundary}\label{sec:lines_corners_in_the_bd}

In this section we prove the following result which we will use to verify properties (7) and (8) in Theorem~\ref{thm:properties_of_cc}.

\begin{proposition}\label{prop:lines_and_corners}   Suppose $\Omega \subset \Pb(\Rb^d)$ is a properly convex domain and $\Lambda \leq \Aut(\Omega)$ is convex co-compact. Assume that the family $\Sc_{\max}$ of all maximal properly embedded simplices in $\Cc:=\Cc_\Omega(\Lambda)$ of dimension at least two satisfies the following:
\begin{enumerate}
\item $\Sc_{\max}$ is strongly isolated.
\item If $S \in \Sc_{\max}$ and $x \in \partial S$, then $F_{\Omega}(x) = F_S(x)$. 
\end{enumerate}
Then:
\begin{enumerate}[(a)]
\item If $\ell \subset \partiali \Cc$ is a non-trivial line segment, then there exists $S \in \Sc_{\max}$ with $\ell \subset \partial S$. 
\item If $y \in \partiali \Cc$ is not a $C^1$-smooth point of $\partial \Omega$, then there exists $S \in \Sc_{\max}$ with $y \in \partial S$. 
\end{enumerate}
\end{proposition}

\begin{remark} In Section \ref{sec:pf_of_thm_main_cc} we will show that if $\Lambda \leq \Aut(\Omega)$ is convex-compact  and $\Sc_{\max}$ is an isolated family of properly embedded simplices,  then conditions (1) and (2) are  automatically satisfied. 
\end{remark}

We will need the following observation about convex co-compact subgroups. 

\begin{proposition}\cite[Lemma 4.1 part (1)]{DGF2017}\label{prop:full_face} Suppose $\Omega \subset \Pb(\Rb^d)$ is a properly convex domain and $\Lambda \leq \Aut(\Omega)$ is convex co-compact. If $x \in \partiali \Cc_\Omega(\Lambda)$, then $F_\Omega(x) \subset \partiali \Cc_\Omega(\Lambda)$. 
\end{proposition}

We start the proof of Proposition~\ref{prop:lines_and_corners} with some general lemmas. 

\begin{lemma}\label{lem:geom_near_line_segment_in_bd} Suppose $\Omega \subset \Pb(\Rb^d)$ is a properly convex domain, $\Lambda \leq \Aut(\Omega)$ is convex co-compact, and $\Cc:=\Cc_\Omega(\Lambda)$. Assume $\ell \subset \partiali \Cc$ is a non-trivial open line segment, $m \in \ell$, $q \in \Cc$, and $V = \Spanset\{ \ell, q\}$. For any $r > 0$ and $\epsilon > 0$ there exists a neighborhood $U$ of $m$ in $\Pb(V)$ such that: if $p \in  U \cap \Cc$, then there exists a properly embedded simplex $S=S(p) \subset \Cc$ of dimension at least two such that 
\begin{align}
\label{eq:ball_near_simplex_2}
B_{\Omega}(p;r) \cap \Pb(V) \subset \Nc_\Omega(S;\epsilon).
\end{align}
\end{lemma}

\begin{proof} The argument is very similar to the proof of Lemma~\ref{lem:geom_near_corner}.

Fix $r >0$ and $\epsilon > 0$. Suppose for a contradiction that such a neighborhood $U$ does not exist. Then we can find $p_n \in \Cc \cap \Pb(V)$ such that $\lim_{n \to \infty} p_n = m$ and $p_n$ does not satisfy Equation~\eqref{eq:ball_near_simplex_2} for any properly embedded simplex in $\Cc$ of dimension at least two. 

By replacing $\ell$ with the maximal open line segment containing it we can assume that $\ell = (a,b)$ where $a,b \in \partial F_\Omega(m)$. 

After passing to a subsequence we can find $\gamma_n \in \Lambda$ such that $\gamma_n p_n \rightarrow p_\infty \in \Cc$. Passing to a further subsequence we can suppose that $\gamma_n a \rightarrow a_\infty$, $\gamma_n b \rightarrow b_\infty$, and $\gamma_n q \rightarrow q_\infty$. Then $[a_\infty, b_\infty] \subset \partiali\Cc$. Since $a,b \in \partial F_\Omega(m)$, we have  
 \begin{align*}
\infty= \lim_{n \rightarrow \infty} H_\Omega\Big(p_n, (a,q) \cup (q,b) \Big) =  \lim_{n \rightarrow \infty} H_\Omega\Big(\gamma_n p_n, (\gamma_n a,\gamma_n q) \cup (\gamma_n q,\gamma_n b)\Big).
\end{align*}
So $[a_\infty, q_\infty] \cup [q_\infty, b_\infty] \subset \partiali\Cc$. Thus $a_\infty, b_\infty, q_\infty$ are the vertices of a properly embedded simplex $S \subset \Cc$. However, for $n$ sufficiently large we have 
\begin{align*}
B_{\Omega}(\gamma_n p_n;r) \cap \gamma_n\Pb(V) \subset \Nc_\Omega(S;\epsilon)
\end{align*}
and so 
\begin{align*}
B_{\Omega}(p_n;r) \cap \Pb(V) \subset \Nc_\Omega(\gamma_n^{-1}S;\epsilon).
\end{align*}
Hence we have a contradiction. 
\end{proof}

\begin{lemma} 
\label{lem:geom_near_non_smooth_point}
Suppose $\Omega \subset \Pb(\Rb^d)$ is a properly convex domain, $\Lambda \leq \Aut(\Omega)$ is convex co-compact, and $\Cc:=\Cc_\Omega(\Lambda)$. Assume $z \in \partiali \Cc$ is not a $C^1$-smooth point of $\partial \Omega$ and $q \in \Cc$. For any $r > 0$ and $\epsilon > 0$ there exists $q_{r,\epsilon} \in (z,q]$ such that: if $p \in  (z,q_{r,\epsilon}]$, then there exists a properly embedded simplex $S=S(p) \subset \Cc$ of dimension at least two such that 
\begin{align}
\label{eq:ball_near_simplex_3}
B_{\Omega}(p;r) \cap (z,q] \subset \Nc_\Omega(S;\epsilon).
\end{align}
\end{lemma}

\begin{proof} Once again, the argument is very similar to the proof of Lemma~\ref{lem:geom_near_corner}.

Fix $r >0$ and $\epsilon > 0$. Suppose for a contradiction that such a $q_{r,\epsilon} \in (z,q]$ does not exist. Then we can find $p_n \in (z,q]$ such that $\lim_{n \to \infty} p_n = z$ and $p_n$ does not satisfy Equation~\eqref{eq:ball_near_simplex_3} for any properly embedded simplex in $\Cc$ of dimension at least two.

We can find a 3-dimensional linear subspace $V$ such that $(z, q] \subset \Pb(V)$ and $z \in \partiali \Cc$ is not a $C^1$-smooth boundary point of $\Pb(V) \cap \Omega$. By changing coordinates we can suppose that 
\begin{align*}
\Pb(V)&= \{ [w:x:y:0:\dots:0] : w,x,y \in \Rb \}, \\
\Pb(V) \cap \Omega&\subset \{ [1:x:y:0:\dots:0] : x \in \Rb,  \ y > \abs{x} \}, \\
z&=[1:0:0:\dots:0], \text{ and }\\
q&=[1:0:1:0\dots:0].
\end{align*}
We may also assume that $\Pb(V) \cap \Omega$ is bounded in the affine chart 
\begin{align*}
\{ [1:x:y:0:\dots:0] : x, y \in \Rb \}
\end{align*}
of $\Pb(V)$.

Then
\begin{align*}
p_n = [1:0:y_n:0:\dots:0]
\end{align*}
where $0<y_n<1$ and $y_n$ converges to 0. Let 
\begin{align*}
L_n := \{ [1:x:y_n:0:\dots :0] : x \in \Rb\} \cap \Omega.
\end{align*}
By passing to a subsequence we can suppose that $(y_n)_{n \geq 1}$ is a decreasing sequence and
\begin{align}
\label{eqn:pn_dist_Ln}
\lim_{n \rightarrow \infty} H_\Omega(p_n,L_{n-1}) = \infty. 
\end{align}
Then 
\begin{align*}
\lim_{n \to \infty}\frac{y_{n-1}}{y_{n}} =\infty.
\end{align*}

Let $a_n, b_n \in \partial \Omega$ be the endpoints of $L_n = (a_n, b_n)$. We claim that 
\begin{align}
\label{eqn:pn_dist_to_lines}
\lim_{n \rightarrow \infty} H_\Omega\Big(p_n, (z, a_{n-1})\Big) = \infty = \lim_{n \rightarrow \infty} H_\Omega\Big(p_n, (z, b_{n-1})\Big).
\end{align}
Consider  $g_n \in \PGL(V)$ defined by 
\begin{align*}
g_n( [w:x:y:0:\dots:0]) = \left[ w : \frac{1}{y_n}x : \frac{1}{y_n} y: \dots : 0\right].
\end{align*}
Since $(y_n)_{n \geq 1}$ is a decreasing sequence converging to zero, $D_n : =g_n( \Pb(V) \cap \Omega)$ is an increasing sequence of properly convex domains in $\Pb(V)$ and 
\begin{align*}
D: = \cup_{n \geq 1} D_n \subset \{ [1:x:y:0:\dots:0] : x \in \Rb, \ y > \abs{x} \}
\end{align*}
is also a properly convex domain. Notice that $H_{D_n}$ converges to $H_D$ uniformly on compact subsets of $D$. Also, by construction, there exist $t \leq -1$ and $1 \leq s$ such that 
\begin{align*}
D = \{ [1:x:y:0:\dots:0] : x \in \Rb, \ y > \max\{ sx, tx\} \}.
\end{align*}
Then $a_{n}=[1: t_{n}^{-1} y_{n}:y_{n}:0:\ldots:0]$ where $t_n \rightarrow t$.

Now pick $v_n \in (z, a_{n-1})$ such that 
\begin{align*}
H_\Omega\Big(p_n, (z, a_{n-1})\Big) =H_\Omega(p_n,v_n).
\end{align*}
Since 
$$\lim_{n \to \infty} g_n a_{n-1}=\lim_{n \to \infty} \left[ 1:t_{n-1}^{-1}\frac{y_{n-1}}{y_n}:\frac{y_{n-1}}{y_n}:0:\ldots:0 \right]=[0:t^{-1}:1:0:\dots:0]$$
any limit point of $g_n v_n$ is in 
\begin{align*}
\{ [0:t^{-1}:1:0:\dots:0]\} \cup \{ [1: rt^{-1} : r : 0 : \dots : 0] : r \geq 0\} \subset \partial D.
\end{align*}
Then
\begin{align*}
\lim_{n \rightarrow \infty} H_\Omega\Big(p_n, (z, a_{n-1})\Big) = \lim_{n \rightarrow \infty} H_\Omega(p_n,v_n) = \lim_{n \rightarrow \infty} H_{D_n}\Big(g_np_n, g_nv_n\Big)=\infty
\end{align*}
since $g_n p_n \rightarrow [1:0:1:0:\dots:0] \in D$.

 For the same reasons, 
\begin{align*}
 \lim_{n \rightarrow \infty} H_\Omega\Big(p_n, (z, b_{n-1})\Big) = \infty. 
\end{align*}
This establishes Equation~\eqref{eqn:pn_dist_to_lines}.

Next we can pass to a subsequence and find $\gamma_n \in \Lambda$ such that $\gamma_n p_n \rightarrow p_\infty \in \Cc$. Passing to a further subsequence we can suppose that $\gamma_n a_{n-1} \rightarrow a_\infty$, $\gamma_n b_{n-1} \rightarrow b_\infty$, $\gamma_n z \rightarrow z_\infty$, and $\gamma_n q \rightarrow q_\infty$. 

Equation~\eqref{eqn:pn_dist_Ln} implies that $[a_\infty, b_\infty] \subset \partial \Omega$ and Equation~\eqref{eqn:pn_dist_to_lines} implies that 
\begin{align*}
[z_\infty, a_\infty] \cup [z_\infty, b_\infty] \subset \partial \Omega.
\end{align*}
Thus $a_\infty, b_\infty, z_\infty$ are the vertices of a properly embedded simplex $S \subset \Omega$ which contains $p_\infty$. Further, for $n$ sufficiently large we have 
\begin{align*}
B_{\Omega}(\gamma_n p_n;r) \cap \gamma_n(z,q] \subset \Nc_\Omega(S;\epsilon)
\end{align*}
and so 
\begin{align*}
B_{\Omega}(p_n;r) \cap(z,q] \subset \Nc_\Omega(\gamma_n^{-1}S;\epsilon).
\end{align*}

To obtain a contradiction we have to show that $\gamma_n^{-1}S \subset \Cc$  for every $n$ or equivalently that $S \subset \Cc$. By construction, $q_\infty \in \partiali\Cc \cap (a_\infty, b_\infty)$. Then Proposition~\ref{prop:full_face} implies that $[a_\infty, b_\infty] \subset \partiali\Cc$. Since $z_\infty \in \partiali\Cc$ and $S$ has vertices $a_\infty, b_\infty, z_\infty$ we then see that $S \subset \Cc$. 
\end{proof}

The rest of the section is devoted to the proof of Proposition~\ref{prop:lines_and_corners}, so suppose that $\Omega \subset \Pb(\Rb^d)$ is a properly convex domain, $\Lambda \leq \Aut(\Omega)$ is convex co-compact, and the family $\Sc_{\max}$ of all maximal properly embedded simplices in $\Cc:=\Cc_\Omega(\Lambda)$ of dimension at least two satisfies the hypothesis of the proposition.

\begin{lemma}\label{lem:line_segments_in_bd}  If $\ell \subset \partiali \Cc$ is a non-trivial line segment, then there exists $S \in \Sc_{\max}$ with $\ell \subset \partial S$.
\end{lemma}

\begin{proof} We can assume that $\ell$ is an open line segment. Then fix some $m \in \ell$ and $q \in \Cc$. Since $\Sc_{\max}$ is strongly isolated, there exists some $D > 0$ such that if $S_1, S_2 \in \Sc_{\max}$ are distinct, then 
\begin{align*}
\diam_\Omega(\Nc_\Omega(S_1;1) \cap \Nc_\Omega(S_2;1) ) < D.
\end{align*}

Let $V := \Spanset\{ \ell, q\}$. By Lemma~\ref{lem:geom_near_line_segment_in_bd} there exists a neighborhood $U$ of $m$ in $\Pb(V)$ such that: if $x \in  U \cap \Cc$, then there exists a maximal properly embedded simplex $S_x \subset \Cc$ of dimension at least two such that 
\begin{align*}
B_{\Omega}(x;D) \cap \Pb(V) \subset \Nc_\Omega(S_x;1).
\end{align*}
By shrinking $U$ we can assume that $U \cap \Cc$ is convex.

We claim that $S_x = S_y$ for every $x,y \in U \cap \Cc$. Since $U \cap \Cc$ is convex, it is enough to show this when $H_\Omega(x,y) \leq D/2$. Then 
\begin{align*}
B_{\Omega}(x;D/2) \cap \Pb(V) \subset B_{\Omega}(y;D) \cap \Pb(V) \subset \Nc_\Omega(S_y;1).
\end{align*}
So
\begin{align*}
B_{\Omega}(x;D/2) \cap \Pb(V) \subset \Nc_\Omega(S_x;1) \cap \Nc_\Omega(S_y;1) 
\end{align*}
and hence 
\begin{align*}
\diam_\Omega(\Nc_\Omega(S_x;1) \cap \Nc_\Omega(S_y;1) ) \geq \diam_{\Omega}\big( B_{\Omega}(x;D/2) \cap \Pb(V) \big) = D.
\end{align*}
So $S_x = S_y$. 

Now let $S := S_x$ for some (hence any) $x \in U \cap \Cc$. Then 
\begin{align*}
U \cap \Cc \subset \Nc_\Omega(S;1).
\end{align*}
So by Proposition~\ref{prop:dist_est_and_faces} there exists $m^\prime \in \partial S$ with $m \in F_\Omega(m^\prime)$. Then, since $\ell$ is an open line segment, $\ell \subset F_\Omega(m^\prime)$. Finally, by condition (2) of the hypothesis, $F_{\Omega}(m')=F_S(m') \subset \partial S.$ Hence, $\ell \subset \partial S.$
\end{proof}

\begin{lemma}
If $z \in \partiali \Cc$ is not a $C^1$-smooth point of $\partial \Omega$, then there exists  $S \in \Sc_{\max}$ with $z \in \partial S$. 
\end{lemma}

\begin{proof} Fix $q \in \Cc$. Arguing as in the proof of Lemma~\ref{lem:line_segments_in_bd} and using Lemma \ref{lem:geom_near_non_smooth_point} shows that there exist some $q_0 \in (z,q]$ and a maximal properly embedded simplex $S \subset \Cc$ of dimension at least two such that 
\begin{align*}
(z,q_0] \subset \Nc_\Omega(S;1).
\end{align*}
Then by Proposition~\ref{prop:dist_est_and_faces} there exists $z^\prime \in \partial S$ with $z \in F_\Omega(z^\prime)$. Finally, by condition (2) of hypothesis  on $\Sc_{\max}$, 
\begin{equation*}
z \in F_\Omega(z') = F_S(z^\prime) \subset \partial S. \qedhere
\end{equation*} 
\end{proof}

\section{Proof of Theorem~\ref{thm:main_cc} and Theorem~\ref{thm:properties_of_cc}}\label{sec:pf_of_thm_main_cc}

For the rest of this section, suppose $\Omega \subset \Pb(\Rb^d)$ is a properly convex domain, $\Lambda \leq \Aut(\Omega)$ is a convex co-compact subgroup (see Definition \ref{defn:cc}), and $\Sc_{\simplexcc}$ is the family of all maximal properly embedded simplices in $\Cc_{\Omega}(\Lambda)$ of dimension at least two.

\subsection{Proof of Theorem \ref{thm:main_cc}}
We claim that the following are equivalent:
\begin{enumerate}[(A)]
\item $\Sc_{\simplexcc}$ is closed and discrete in the local Hausdorff topology induced by $H_\Omega$.
\item $(\Omega, \Cc_{\Omega}(\Lambda),\Lambda)$ has coarsely isolated simplices.
\item $(\Cc_\Omega(\Lambda), H_\Omega)$ is a relatively hyperbolic space with respect to $\Sc_{\simplexcc}$.
\item $(\Cc_\Omega(\Lambda), H_\Omega)$ is a relatively hyperbolic space with respect to a family of properly embedded simplices in $\Cc_\Omega(\Lambda)$ of dimension at least two.
\item $\Lambda$ is a relatively hyperbolic group with respect to a collection of virtually Abelian subgroups of rank at least two.
\end{enumerate}

By definition (A) implies (B) and (C) implies (D). Further, Theorem \ref{thm:main_ncc}  implies that (B), (D), and (E) are all equivalent. So it is enough to assume (B) and show that (A) and (C) hold. We establish this using Theorem \ref{thm:IS-implies-rel-hyp}  and the next lemma. 

\begin{lemma}\label{lem:pf_of_main_cc}
If $(\Omega, \Cc_{\Omega}(\Lambda),\Lambda)$ has coarsely isolated simplices, then $\Sc_{\max}$ is strongly isolated, coarsely complete, and $\Lambda$-invariant. Moreover, if $S \in \Sc_{\max}$ and $x \in \partial S$, then $F_{\Omega}(x)=F_S(x).$ 
\end{lemma}

\begin{remark}
 A careful reading of the proof shows that $\Sc_{\max}$ is actually the \textbf{unique} family of strongly isolated, coarsely complete, and $\Lambda$-invariant  maximal properly embedded simplices in $\Cc_{\Omega}(\Lambda)$ of dimension at least two. 
\end{remark}

\begin{proof} By Theorem \ref{thm:S-core-exists} there exists  $\Sc_{\core}$, a strongly isolated, coarsely complete, and $\Lambda$-invariant family of maximal properly embedded simplices in $\Cc_{\Omega}(\Lambda)$ of dimension at least two.

We first claim that if $S \in \Sc_{\core}$ and $x \in \partial S$, then 
\begin{align}
\label{eq:equivalence_of_faces}
F_{\Omega}(x)=F_S(x).
\end{align} 
By definition $F_S(x) \subset F_{\Omega}(x).$ To establish the other inclusion,  it suffices to show: if $e \in \partial F_\Omega(x)$ is an extreme point, then $e \in \partial F_S(x)$. 

So, let $e \in \partial F_\Omega(x)$ be an extreme point. Theorem \ref{thm:properties_of_ncc} part (4) implies that there exists $D_1 > 0$ such that:
\begin{align*}
H_{F_\Omega(x)}^{ \rm Haus}\Big( \overline{\Cc_{\Omega}(\Lambda)} \cap F_\Omega(x), F_S(x)\Big) \leq D_1. 
\end{align*}
By Proposition~\ref{prop:full_face}, $F_{\Omega}(x)=\overline{\Cc_{\Omega}(\Lambda)} \cap F_{\Omega}(x).$ Thus, 
\begin{align}
\label{eq:haus_dist_S_in_proof_of_main_thm}
H^{\Haus}_{F_{\Omega}(x)}(F_{\Omega}(x),F_S(x)) \leq D_1.
\end{align}
Then, by Proposition~\ref{prop:dist_est_and_faces} and Equation~\eqref{eq:haus_dist_S_in_proof_of_main_thm} there exists 
\begin{align*}
e^\prime \in \partial F_S(x)\cap F_{F_\Omega(x)}(e).
\end{align*}
But since $e$ is an extreme point $F_{F_\Omega(x)}(e)=F_\Omega(e)=\{e\}$. So $e=e^\prime \in \partial F_S(x)$. This proves the claim. 

Next we show that $\Sc_{\core}=\Sc_{\max}$. By definition $\Sc_{\core} \subset \Sc_{\simplexcc}$, so it is enough to show that $\Sc_{\max} \subset \Sc_{\core}$. Fix $S \in \Sc_{\simplexcc}$.  Since $\Sc_{\core}$ is coarsely complete, there exist $S^\prime \in \Sc_{\core}$ and $r > 0$ such that 
\begin{align*}
S \subset \Nc_\Omega(S^\prime; r).
\end{align*}
Then by Proposition~\ref{prop:dist_est_and_faces} and Equation~\eqref{eq:equivalence_of_faces}
\begin{align*}
\partial S \subset \cup_{x \in \partial S^\prime} F_\Omega(x) = \cup_{x \in \partial S^\prime} F_{S^\prime}(x) = \partial S^\prime. 
\end{align*}
Hence  $S \subset S^\prime$. Since $S$ is a maximal properly embedded simplex we then have $S=S^\prime \in \Sc_{\core}$. 

Finally, the ``moreover'' part follows from the Equation~\eqref{eq:equivalence_of_faces} and the equality $\Sc_{\max} = \Sc_{\core}$. 
\end{proof}

\subsection{Proof of Theorem \ref{thm:properties_of_cc}}
Now assume, in addition  to the hypothesis at the beginning of Section \ref{sec:pf_of_thm_main_cc}, that $\Sc_{\max}$ is closed and discrete in the local Hausdorff topology induced by $H_\Omega$.

By Lemma~\ref{lem:pf_of_main_cc}, $\Sc_{\max}$ is strongly isolated, coarsely complete, and $\Lambda$-invariant. Then properties $(1)$, $(2)$, $(3)$, and $(5)$ follow immediately from Theorem~\ref{thm:properties_of_ncc}. Property $(6)$ holds since $\Sc_{\max}$ is strongly isolated. Property  $(4)$ is the ``moreover'' part of Lemma~\ref{lem:pf_of_main_cc}. Finally, properties $(7)$ and $(8)$ follow from Proposition \ref{prop:lines_and_corners}.

\appendix

\section{Remarks on Theorem~\ref{thm:Sisto_equiv}}\label{sec:pf_of_thm_Sisto}

In this appendix we explain how  to modify Sisto's arguments in \cite{S2013} to establish Theorem \ref{thm:Sisto_equiv}.  In fact,  we will explain why a more general result is true. Before stating the result, we introduce a generalization of the notion of \emph{asymptotically transverse-free} obtained by replacing geodesics in Definition \ref{defn:asym-trans-free} with $(1,c)$-quasi-geodesics.

\begin{definition} Let $(X,\dist)$ be a complete geodesic metric space, $\alpha \geq 1$, $\beta \geq 0$, and  $\Sc$ be a collection of subsets of $X$.
\begin{enumerate}
\item  If $I \subset \Rb$ is an interval, then $\theta : I \rightarrow X$ is a \emph{$(\alpha, \beta)$-quasi-geodesic} in $(X,\dist)$ if 
\begin{equation*}
\dfrac{1}{\alpha}|t_1-t_2|-\beta \leq \dist(\theta(t_1), \theta(t_2)) \leq \alpha |t_1-t_2|+\beta
\end{equation*}
for all $t_1, t_2 \in I$. 
\item A  \emph{$(\alpha, \beta)$-quasi-geodesic triangle} in $(X,\dist)$  is a choice of three points in X and $(\alpha, \beta)$-quasi-geodesics connecting these points.
\item A quasi-geodesic triangle $\Tc$ in $X$ is \emph{$\Sc$-almost-transverse with constants $\kappa$ and  $\triangleD$} if
\begin{align*}
\diam_X(\Nc_X(S;\kappa) \cap \gamma) \leq \triangleD
\end{align*}
for every $S \in \Sc$ and edge $\gG$ of $\Tc$.
 \item The collection $\Sc$ is \emph{strongly asymptotically transverse-free} if there exist $\lambda,~\sigma$ such that for each $c \geq 1$, $\triangleD \geq 1$, $\kappa \geq \sG$ the following holds: if $\Tc$ is a $(1,c)$-quasi-geodesic triangle in $X$ which is $\Sc$-almost-transverse with constants $\kappa$ and $\triangleD$, then $\Tc$ is $(\lambda \triangleD+\lambda c)$-thin.
 \end{enumerate}
\end{definition}

We will prove the following generalization of Theorem \ref{thm:Sisto_equiv} which connects the three different notions of `asymptotically transverse-free.'

\begin{proposition}\label{prop:asym-trans-geodesic-path-system}
Let $(X,\dist)$ be a complete geodesic metric space and $\Sc$ a collection of subsets of $X$. Then the following are equivalent: 
\begin{enumerate}
\item $\Sc$ is asymptotically transverse-free relative to a geodesic path system and there exists an almost-projection system for $\Sc$,
\item $\Sc$ is asymptotically transverse-free and there exists an almost-projection system for $\Sc$,
\item $\Sc$ is strongly asymptotically transverse-free and there exists an almost-projection system for $\Sc$.
\end{enumerate}
\end{proposition}

In Proposition~\ref{prop:asym-trans-geodesic-path-system}, observe that (3) implies (2) and (2) implies (1) by definition. Thus, in order to prove Proposition \ref{prop:asym-trans-geodesic-path-system}, it suffices to prove (1) implies (3). Sisto~\cite[Lemma 2.13]{S2013} previously proved that  (2) implies (3) and in the rest of this section we modify Sisto's argument to show that (1) implies (3). 

For the rest of this section fix: $(X,\dist)$ a complete geodesic metric  space, $\Gc$ a geodesic path system on $X$, $\Sc$ a collection of  subsets of $X$, and $\Pi_{\Sc}=\{ \pi_S : X \to S : S \in \Sc \}$ an almost-projection system with constant $C$.  Then fix a constant
\begin{equation*}
\sigma_0\geq \max \{10C,1\}.
\end{equation*}
Finally, for any pair of distinct points $x, y \in X$, let $\gamma_{x,y}$ denote a path in $\Gc$ connecting $x$ and $y$.

The proof of Proposition~\ref{prop:asym-trans-geodesic-path-system} will require the following two lemmas. Informally, the first one says that if $\theta$ is an ``$\Sc$-almost-transverse quasi-geodesic,'' then any geodesic joining points on $\theta$ is also ``$\Sc$-almost-transverse.''
 
 \begin{lemma}\label{lem:geodesic-penetration} \cite[pg. 176]{S2013} Suppose $c>0$,  $\triangleD \geq 1$, $\kappa \geq c\sigma_0$, $\theta:[0,T] \to X$ is a $(1,c)$-quasi-geodesic, and 
 \begin{align*}
 \diam_X \Big( \Nc_{X} \big( S;\kappa \big) \cap \theta \Big) \leq \triangleD
 \end{align*}
 for every $S \in \Sc$. Then $$\diam_X \Big( \Nc_{X} \big( S;c\sigma_0 \big) \cap \gamma_{x,y} \Big) \leq \triangleD+10 \sigma_0 + 18 c\sigma_0$$
 for every $S \in \Sc$ and $x,y \in \theta$. 
\end{lemma}

Lemma~\ref{lem:geodesic-penetration} follows from Sisto's proof of Lemma 2.13 in~\cite{S2013}. For the reader's convenience we will provide the argument at the end of this section. 

We  also need the following variant of the Morse lemma. 

\begin{lemma}\label{lem:morse-lemma-type-result}
 Suppose $c>0$, $x,y \in X$, and $\theta:[0,T] \to X$ is a $(1,c)$-quasi-geodesic with $x=\theta(0)$ and $y=\theta(T)$. Moreover, suppose that there exists $\delta \geq 0 $ such that any triangle with all its vertices on $\theta$ and all its edges in $\Gc$ is $\delta$-thin. Then 
 $$\dist^{\Haus}(\theta, \gamma_{x,y}) \leq 4 \delta+ 10 c.$$

\end{lemma}

Delaying the proof of Lemma~\ref{lem:morse-lemma-type-result} we prove Proposition~\ref{prop:asym-trans-geodesic-path-system}.

\begin{proof}[Proof of Proposition~\ref{prop:asym-trans-geodesic-path-system}] By the remarks above it suffices to show that $(1)$ implies $(3)$. So suppose that  $\Sc$ is asymptotically transverse-free relative to the geodesic path system $\Gc$ with constants $\lambda_{\Gc}$ and  $\sigma_{\Gc}$.

 By increasing $\sigma_0$ if necessary we can assume that 
 \begin{equation*}
\sigma_0= \max \{10C,1,\sigma_{\Gc}\}.
\end{equation*}
Then fix
 \begin{equation*}
 \lambda_0:=\max \{ 9\lambda_{\Gc}(1+10 \sigma_0), 20\sigma_0 (1+ 9 \lambda_{\Gc})\}.
 \end{equation*}
 
Fix a $(1,c)$-quasi-geodesic triangle $\Tc:=\big( \theta_1 \cup \theta_2 \cup \theta_3 \big)$ that is $\Sc$-almost transverse with constants $\kappa$ and $\triangleD$ where $c \geq 1$, $\kappa \geq c\sigma_0$, and $\triangleD \geq 1$.  We will show that $\Tc$ is $\big( \lambda_0\triangleD+\lambda_0c \big)$-thin. Since $\Tc$ is arbitrary, this will complete the proof that $\Sc$ is strongly asymptotically transverse free and hence that $(1)$ implies $(3)$ in Proposition~\ref{prop:asym-trans-geodesic-path-system}.

Let $\Tc_{\Gc}$ be a geodesic triangle with the same vertices as $\Tc$ but edges in $\Gc.$ Let $\gamma_1,\gamma_2,\gamma_3$ be the edges of $\Tc_{\Gc}$ labelled so that the edge $\gamma_i$ corresponds to the edge $\theta_i$ for all $1 \leq i \leq 3$. By Lemma \ref{lem:geodesic-penetration}, $\Tc_{\Gc}$ is $\Sc$-almost-transverse  with constants $c\sigma_0$ and $(\triangleD+10 \sigma_0 +18 c \sigma_0)$. Notice that $c \sigma_0 \geq \sigma_{\Gc}$ and $\triangleD+10 \sigma_0 +18 c \sigma_0 \geq 1.$ Since $\Sc$ is asymptotically transverse-free relative to the geodesic path system $\Gc$, the triangle $\Tc_{\Gc}$ is $\delta$-thin where
\begin{equation}\label{eqn:delta-for-morse-lemma}
\delta:=\lambda_{\Gc} \big( \triangleD + 10 \sigma_0 +18 c\sigma_0 \big). 
\end{equation}

Lemma \ref{lem:geodesic-penetration} also show that for each $1 \leq i \leq 3$, the $(1,c)$-quasi-geodesic $\theta_i$ and the geodesic $\gamma_i \in \Gc$ satisfy the hypothesis in Lemma \ref{lem:morse-lemma-type-result} with $\delta$ as in \eqref{eqn:delta-for-morse-lemma}. Thus, 
\begin{equation}\label{eqn:theta-i-gamma-i}
\max_{1 \leq i \leq 3} \dist^{\Haus}(\theta_i,\gamma_i) \leq 4 \delta +10c.
\end{equation}
So  $\Tc$ is $\big( 9 \delta + 20c \big)$-thin. Further, 
\begin{align*}
9 \delta + 20 c &= 9 \lambda_{\Gc}\big( \triangleD+10 \sigma_0 + 18 c \sigma_0 \big)+20c\\
& < 9 \lambda_{\Gc} \big( 1+10 \sigma_0 \big) \triangleD + 20\sigma_0 \big( 1+ 9 \lambda_{\Gc} \big) c \\
& \leq \lambda_0 \big( \triangleD+c \big).
\end{align*} 
\noindent Thus, $\Tc$ is $\big( \lambda_0  \triangleD+\lambda_0 c  \big)$-thin. 
\end{proof}

\subsection{Proof of Lemma~\ref{lem:geodesic-penetration}} Before proving the lemma we need to recall two other estimates from Sisto's paper. 

\begin{proposition}[{Sisto~\cite[Corollary 2.7]{S2013}}]\label{prop:sisto-2.7} If $r \geq 2C$, $x_1,x_2 \in X$, $S \in \Sc$, $\pi_S \in \Pi_{\Sc}$, and $\xi$ is any geodesic in $X$ connecting $x_1$ and $x_2$, then $$\diam_X \Big( \xi \cap \Nc_X(S;r) \Big) \leq \dist \Big( \pi_S(x_1),\pi_S(x_2) \Big)+18r + 62C.$$ 
\end{proposition}

\begin{proposition}[{Sisto~\cite[Lemma 2.10]{S2013}}]\label{prop:sisto-2.10}
If $x_1, x_2 \in X$, $S \in \Sc$, $\pi_S \in \Pi_{\Sc}$, $\xi$ is any geodesic in $X$ connecting $x_1$ and $x_2$, and $\dist \Big( \pi_S(x_1),\pi_S(x_2) \Big) \geq 8C +1$,  then $\xi$  intersects $B_X \Big( \pi_S(x_1);10C \Big)$, $B_X \Big( \pi_S(x_2);10C \Big)$, and $\Nc_X(S;2C). $
\end{proposition}

\medskip

\noindent We now claim that 
\begin{equation}\label{eq:lemA3_est}
\dist \Big( \pi_S(x),\pi_S(y) \Big) \leq \triangleD + 20 C + 1.
\end{equation} If this is not true, then by Proposition~\ref{prop:sisto-2.10}, $\gamma_{x,y}$ intersects $B_{\Omega} \big( \pi_S(x);10C \big)$ and $B_{\Omega} \big( \pi_S(y); 10 C \big)$. Thus, 
\begin{align*}
\diam_X \left( \Nc_{X} \big( S;\kappa \big) \cap \gamma_{x,y} \right)\geq \diam_X \left( \Nc_{X} \big( S;10C \big) \cap \gamma_{x,y} \right) \geq \triangleD+1, 
\end{align*}
 which is a contradiction.   Hence, the estimate in Equation~\eqref{eq:lemA3_est} is true. Since $c \sigma_0 \geq 2C$, Proposition~\ref{prop:sisto-2.7}  implies that 
\begin{align*}
\diam_X \Big( \gamma_{x,y} \cap \Nc_X \big( S;c\sigma_0 \big) \Big) &\leq \triangleD+ 18 c \sigma_0 + 82 C +1 \\
&\leq \triangleD + 10 \sigma_0 +18 c\sigma_0. \qedhere
\end{align*}

\subsection{Proof of Lemma~\ref{lem:morse-lemma-type-result}} 

Let $M := 2 \delta +5c$. By a standard argument (see for instance~\cite[Proof of Theorem 1.7, pg. 404]{BH2013}), it suffices to prove: 
\begin{align*}
\theta \subset \Nc_X(\gamma_{x,y};M-c).
\end{align*}

 Fix $z \in \theta$ and consider the geodesic triangle $\gamma_{x,y} \cup \gamma_{x,z} \cup \gamma_{z,y}$. By hypothesis, this triangle is $\delta$-thin. Next pick $a \in \gamma_{x,z}$ such that $\dist(z,a) = \delta + 4c$. If such a point does not exist, then $\dist(x,z) < \delta+4c$ which implies that
 \begin{align*}
 z \in \Nc_X(\gamma_{x,y}; \delta + 4c) \subset \Nc_X(\gamma_{x,y}; M-c)
 \end{align*} 
 and we are done. Now since $\gamma_{x,y} \cup \gamma_{x,z} \cup \gamma_{z,y}$ is $\delta$-thin there exists $b \in \gamma_{x,y} \cup \gamma_{z,y}$ such that $\dist(a,b) \leq \delta$. We will show that $b \in \gamma_{x,y}$. Since $\theta$ is a $(1,c)$-quasi-geodesic, 
 $$\dist(x,z) + \dist(z,y)  \leq \dist(x,y) + 3c.$$
 Then for all $y^\prime \in \gamma_{z,y}$ we have
\begin{align*}
\dist(a,y^\prime) &\geq \dist(x,y)-\dist(x,a)-\dist(y^\prime,y) \\
&\geq \dist(x,z) + \dist(z,y) -3c -\dist(x,a)-\dist(y^\prime,y) \\
& = \dist(a,z)+\dist(y^\prime, z) - 3 c \geq \delta+c. 
\end{align*}
So we must have $b \in \gamma_{x,y}$. Then 
\begin{align*}
\dist(z,\gamma_{x,y}) \leq \dist(z,a)+\dist(a,b) \leq 2\delta+4c = M-c.
\end{align*} 
Since $z \in \theta$ was arbitrary 
 \begin{equation*}
\theta \subset \Nc_X(\gamma_{x,y};M-c). \qedhere
\end{equation*}

\bibliographystyle{alpha}
\bibliography{geom}

\begin{thebibliography}{{Bob}20}

\bibitem[BDL18]{BDL2018}
Samuel~A. Ballas, Jeffrey Danciger, and Gye-Seon Lee.
\newblock Convex projective structures on nonhyperbolic three-manifolds.
\newblock {\em Geom. Topol.}, 22(3):1593--1646, 2018.

\bibitem[Ben60]{B1960}
Jean-Paul Benz\'{e}cri.
\newblock Sur les vari\'{e}t\'{e}s localement affines et localement
  projectives.
\newblock {\em Bull. Soc. Math. France}, 88:229--332, 1960.

\bibitem[Ben03]{B2003b}
Yves Benoist.
\newblock Convexes divisibles. {II}.
\newblock {\em Duke Math. J.}, 120(1):97--120, 2003.

\bibitem[Ben04]{B2004}
Yves Benoist.
\newblock Convexes divisibles. {I}.
\newblock In {\em Algebraic groups and arithmetic}, pages 339--374. Tata Inst.
  Fund. Res., Mumbai, 2004.

\bibitem[Ben06]{B2006}
Yves Benoist.
\newblock Convexes divisibles. {IV}. {S}tructure du bord en dimension 3.
\newblock {\em Invent. Math.}, 164(2):249--278, 2006.

\bibitem[Ben08]{B2008}
Yves Benoist.
\newblock A survey on divisible convex sets.
\newblock In {\em Geometry, analysis and topology of discrete groups}, volume~6
  of {\em Adv. Lect. Math. (ALM)}, pages 1--18. Int. Press, Somerville, MA,
  2008.

\bibitem[BH99]{BH2013}
Martin~R. Bridson and Andr\'{e} Haefliger.
\newblock {\em Metric spaces of non-positive curvature}, volume 319 of {\em
  Grundlehren der Mathematischen Wissenschaften [Fundamental Principles of
  Mathematical Sciences]}.
\newblock Springer-Verlag, Berlin, 1999.

\bibitem[BK53]{BK1953}
Herbert Busemann and Paul~J. Kelly.
\newblock {\em Projective geometry and projective metrics}.
\newblock Academic Press Inc., New York, N. Y., 1953.

\bibitem[{Bob}20]{MB2020}
Martin~D. {Bobb}.
\newblock {Codimension-$1$ Simplices in Divisible Convex Domains}.
\newblock {\em arXiv e-prints}, page arXiv:2001.11096, January 2020.

\bibitem[CLM20]{CLM2016}
Suhyoung Choi, Gye-Seon Lee, and Ludovic Marquis.
\newblock Convex projective generalized {D}ehn filling.
\newblock {\em Ann. Sci. \'{E}c. Norm. Sup\'{e}r. (4)}, 53(1):217--266, 2020.

\bibitem[CLT15]{CLT2015}
D.~Cooper, D.D. Long, and S.~Tillmann.
\newblock On convex projective manifolds and cusps.
\newblock {\em Advances in Mathematics}, 277:181 -- 251, 2015.

\bibitem[DGK17]{DGF2017}
Jeffrey {Danciger}, Fran{\c{c}}ois {Gu{\'e}ritaud}, and Fanny {Kassel}.
\newblock {Convex cocompact actions in real projective geometry}.
\newblock {\em arXiv e-prints}, page arXiv:1704.08711, Apr 2017.

\bibitem[DGK18]{DGF2018}
Jeffrey Danciger, Fran\c{c}ois Gu\'{e}ritaud, and Fanny Kassel.
\newblock Convex cocompactness in pseudo-{R}iemannian hyperbolic spaces.
\newblock {\em Geom. Dedicata}, 192:87--126, 2018.

\bibitem[dlH93]{dlH1993}
Pierre de~la Harpe.
\newblock On {H}ilbert's metric for simplices.
\newblock In {\em Geometric group theory, {V}ol. 1 ({S}ussex, 1991)}, volume
  181 of {\em London Math. Soc. Lecture Note Ser.}, pages 97--119. Cambridge
  Univ. Press, Cambridge, 1993.

\bibitem[Dru02]{D2002}
Cornelia Dru\c{t}u.
\newblock Quasi-isometry invariants and asymptotic cones.
\newblock {\em Internat. J. Algebra Comput.}, 12(1-2):99--135, 2002.
\newblock International Conference on Geometric and Combinatorial Methods in
  Group Theory and Semigroup Theory (Lincoln, NE, 2000).

\bibitem[DS05]{DS2005}
Cornelia Dru\c{t}u and Mark Sapir.
\newblock Tree-graded spaces and asymptotic cones of groups.
\newblock {\em Topology}, 44(5):959--1058, 2005.
\newblock With an appendix by Denis Osin and Mark Sapir.

\bibitem[Fra89]{Fra1989}
Sidney Frankel.
\newblock Complex geometry of convex domains that cover varieties.
\newblock {\em Acta Math.}, 163(1-2):109--149, 1989.

\bibitem[GW71]{GW1971}
Detlef Gromoll and Joseph~A. Wolf.
\newblock Some relations between the metric structure and the algebraic
  structure of the fundamental group in manifolds of nonpositive curvature.
\newblock {\em Bull. Amer. Math. Soc.}, 77:545--552, 1971.

\bibitem[HK05]{HK2005}
G.~Christopher Hruska and Bruce Kleiner.
\newblock Hadamard spaces with isolated flats.
\newblock {\em Geom. Topol.}, 9:1501--1538, 2005.
\newblock With an appendix by the authors and Mohamad Hindawi.

\bibitem[Hru05]{H2005}
G.~Christopher Hruska.
\newblock Geometric invariants of spaces with isolated flats.
\newblock {\em Topology}, 44(2):441--458, 2005.

\bibitem[IZ21]{IZ2019}
Mitul Islam and Andrew Zimmer.
\newblock A flat torus theorem for convex co-compact actions of projective
  linear groups.
\newblock {\em J. Lond. Math. Soc. (2)}, 103(2):470--489, 2021.

\bibitem[Kap07]{K2007}
Michael Kapovich.
\newblock Convex projective structures on {G}romov-{T}hurston manifolds.
\newblock {\em Geom. Topol.}, 11:1777--1830, 2007.

\bibitem[KL06]{KL2006}
Bruce Kleiner and Bernhard Leeb.
\newblock Rigidity of invariant convex sets in symmetric spaces.
\newblock {\em Invent. Math.}, 163(3):657--676, 2006.

\bibitem[KL18]{KL2018}
Michael {Kapovich} and Bernhard {Leeb}.
\newblock {Relativizing characterizations of Anosov subgroups, I}.
\newblock {\em arXiv e-prints}, page arXiv:1807.00160, June 2018.

\bibitem[KS58]{KS1958}
Paul Kelly and Ernst Straus.
\newblock Curvature in {H}ilbert geometries.
\newblock {\em Pacific J. Math.}, 8:119--125, 1958.

\bibitem[LY72]{LY1972}
H.~Blaine Lawson, Jr. and Shing~Tung Yau.
\newblock Compact manifolds of nonpositive curvature.
\newblock {\em J. Differential Geometry}, 7:211--228, 1972.

\bibitem[Mar14]{L2014}
Ludovic Marquis.
\newblock Around groups in {H}ilbert geometry.
\newblock In {\em Handbook of {H}ilbert geometry}, volume~22 of {\em IRMA Lect.
  Math. Theor. Phys.}, pages 207--261. Eur. Math. Soc., Z\"{u}rich, 2014.

\bibitem[Nus88]{N1988}
Roger~D. Nussbaum.
\newblock Hilbert's projective metric and iterated nonlinear maps.
\newblock {\em Mem. Amer. Math. Soc.}, 75(391):iv+137, 1988.

\bibitem[Qui05]{Q2005}
J.-F. Quint.
\newblock Groupes convexes cocompacts en rang sup\'erieur.
\newblock {\em Geom. Dedicata}, 113:1--19, 2005.

\bibitem[Qui10]{Q2010}
Jean-Fran{\c{c}}ois Quint.
\newblock Convexes divisibles (d'apr\`es {Y}ves {B}enoist).
\newblock {\em Ast\'erisque}, (332):Exp. No. 999, vii, 45--73, 2010.
\newblock S{\'e}minaire Bourbaki. Volume 2008/2009. Expos{\'e}s 997--1011.

\bibitem[Sis13]{S2013}
Alessandro Sisto.
\newblock Projections and relative hyperbolicity.
\newblock {\em Enseign. Math. (2)}, 59(1-2):165--181, 2013.

\bibitem[Ver14]{V2014}
Constantin Vernicos.
\newblock On the {H}ilbert geometry of convex polytopes.
\newblock In {\em Handbook of {H}ilbert geometry}, volume~22 of {\em IRMA Lect.
  Math. Theor. Phys.}, pages 111--125. Eur. Math. Soc., Z\"{u}rich, 2014.

\bibitem[Wis96]{W1996}
Daniel~T. Wise.
\newblock {\em Non-positively curved squared complexes: {A}periodic tilings and
  non-residually finite groups}.
\newblock ProQuest LLC, Ann Arbor, MI, 1996.
\newblock Thesis (Ph.D.)--Princeton University.

\bibitem[{Zhu}19]{FZ2019}
Feng {Zhu}.
\newblock {Relatively dominated representations}.
\newblock {\em arXiv e-prints}, page arXiv:1912.13152, December 2019.

\bibitem[{Zim}17]{Z2017}
Andrew {Zimmer}.
\newblock {Projective {A}nosov representations, convex cocompact actions, and
  rigidity}.
\newblock {\em arXiv e-prints}, page arXiv:1704.08582, Apr 2017.

\end{thebibliography}

\end{document}